\documentclass{article}

\usepackage[utf8]{inputenc}
\usepackage{amsmath, amssymb, amsthm, bm, mathrsfs, mathtools, thmtools}
\usepackage[cal=cm]{mathalpha}
\usepackage{stackengine} 
\DeclareMathAlphabet{\dutchcal}{U}{dutchcal}{m}{n} 
\usepackage{centernot}
\usepackage[top=1.5cm, bottom=1.5cm, left=2cm, right=2cm]{geometry}
\usepackage{csquotes, authblk, cases, enumitem}
\usepackage{graphicx, float, caption, subcaption, adjustbox}
\usepackage{tikz-cd}
\usepackage[numbers,sort]{natbib}
\usepackage[section]{algorithm}
\usepackage{algpseudocode}
\usepackage{array, tabularx, xcolor, colortbl, booktabs}
\usepackage{color, xcolor}
\usepackage[para]{footmisc}
\usepackage[normalem]{ulem}
\usepackage[hidelinks,hypertexnames=false]{hyperref}
\counterwithin{figure}{section}
\counterwithin{table}{section}
\numberwithin{equation}{section}
\definecolor{darkred}{rgb}{0.8,0,0}
\definecolor{darkgreen}{rgb}{0,0.8,0}

\DeclareMathOperator{\diag}{\operatorname{diag}}

\DeclareMathOperator{\supp}{\operatorname{supp}}
\DeclareMathOperator{\keff}{\kappa_{\operatorname{eff}}}

\newcommand{\dxi}{d\xi}
\newcommand{\dx}{d\bm{x}}

\newcommand{\dbxi}{d\bm{\xi}}
\newcommand{\restr}[2]{\left.{#1}\right|_{#2}}
\newenvironment{frcseries}{\fontfamily{frc}\selectfont}{}

\usepackage{etoolbox}
\usepackage{appendix}
\usepackage{cleveref}
\AtBeginEnvironment{appendices}{\crefalias{section}{appendix}}

\theoremstyle{definition}
\newtheorem{definition}{Definition}[section]
\newtheorem{remark}[definition]{Remark}
\newtheorem{example}[definition]{Example}
\newtheorem{theorem}[definition]{Theorem}
\newtheorem{lemma}[definition]{Lemma}
\newtheorem{corollary}[definition]{Corollary}

\AtBeginEnvironment{remark}{\crefalias{definition}{remark}}         
\AtBeginEnvironment{example}{\crefalias{definition}{example}}       
\AtBeginEnvironment{theorem}{\crefalias{definition}{theorem}}       
\AtBeginEnvironment{lemma}{\crefalias{definition}{lemma}}           
\AtBeginEnvironment{corollary}{\crefalias{definition}{corollary}}   
\AtBeginEnvironment{conjecture}{\crefalias{definition}{conjecture}} 

\crefname{definition}{Definition}{Definitions}
\crefname{remark}{Remark}{Remarks}
\crefname{example}{Example}{Examples}
\crefname{theorem}{Theorem}{Theorems}
\crefname{lemma}{Lemma}{Lemmas}
\crefname{corollary}{Corollary}{Corollaries}
\crefname{conjecture}{Conjecture}{Conjectures}

\title{Deflation-based preconditioning for immersed finite element methods and immersogeometric analysis}
\author[1]{Yannis Voet \thanks{yannis.voet@epfl.ch}}
\author[2]{Matthias Möller \thanks{M.Moller@tudelft.nl}}
\author[1]{Pablo Antolin \thanks{pablo.antolin@epfl.ch}}
\author[2]{Cornelis Vuik \thanks{C.Vuik@tudelft.nl}}
\affil[1]{\small MNS, Institute of Mathematics, École polytechnique fédérale de Lausanne, Station 8, CH-1015 Lausanne, Switzerland}
\affil[2]{\small Numerical Analysis, Applied Mathematics, Faculty of Electrical Engineering Mathematics and Computer Science, \protect\\ Delft University of Technology, Delft, The Netherlands}
\date{\today}

\begin{document}
\maketitle

\begin{abstract}
Trimming is a ubiquitous operation in computer-aided-design whereby parts of a geometry are merged, intersected, or simply discarded. While it grants virtually unlimited flexibility in geometric design, it introduces a plethora of other difficulties when such geometries are used within immersed finite element methods. In particular, small cut elements lead to severely ill-conditioned system matrices requiring dedicated penalization, stabilization, or preconditioning techniques. In this work, we highlight the limitations of existing preconditioning strategies by first carefully examining the condition number of the diagonally scaled matrix and later providing realistic counter-examples for some well-established preconditioning strategies. Building on those insights, we propose a robust deflation-based preconditioning technique tailored to immersed finite element methods. \\

\noindent \textbf{Keywords}:
Finite element analysis, Isogeometric analysis, Trimming, Preconditioning, Deflation.
\end{abstract}

\section{Introduction}
\label{se: introduction}
Finite element analyses are the gold standard for approximating the solution of partial differential equations (PDEs) and are routinely executed for engineering applications around the world. In most cases, the geometry is exceedingly complex and implicitly delimited by trimming curves or surfaces within computer-aided-design (CAD) software. Fixing CAD files and meshing the geometry soon becomes very time-consuming, if not practically impossible. Immersed finite element methods circumvent this issue by embedding the geometry in a simpler fictitious domain that is easily meshed. The idea is applicable to classical $C^0$ finite elements and advanced spline-based discretizations alike. The latter, commonly known as immersogeometric analysis, applies the fundamental concepts of isogeometric analysis (IGA) \cite{hughes2005isogeometric,cottrell2009isogeometric} in an immersed setting. IGA borrows the spline functions from CAD such as B-splines and non-uniform rational B-splines (NURBS) for representing the approximate solution. Although it was originally designed with the intention of exactly representing common geometries, it too falls short in describing the highly complex geometries commonly encountered in industrial applications. Therefore, immersed techniques, whether built from classical finite elements or smooth spline functions, have steadily gained a strong foothold in both academic research \cite{parvizian2007finite,schillinger2015finite} and industry \cite{messmer2022efficient,leidinger2019explicit}. Unfortunately, immersed methods introduce many other complications, including cumbersome techniques for (weakly) imposing essential boundary conditions \cite{de2018note}, integration on trimmed elements \cite{antolin2019isogeometric}, potential instabilities, and severe ill-conditioning. See \cite{de2023stability} for an overview of these issues. In this context, stability relates to the well-posedness of the discrete formulation, whereas conditioning pertains to the difficulty of solving the algebraic system of equations \cite{de2023stability}. This contribution specifically focuses on the latter. In immersed finite element methods, the ill-conditioning is caused by extremely small eigenvalues associated with eigenfunctions that are only supported on small trimmed elements \cite{de2017condition,de2019preconditioning,de2020multigrid,de2023stability}. This issue has driven intensive research, and multiple solutions have been proposed in the literature. Some of them also treat the aforementioned instability issue.
\begin{enumerate}[noitemsep]
    \item The so-called $\alpha$-stabilization (or ``material'' stabilization) technique within the finite cell method \cite{parvizian2007finite,schillinger2015finite} integrates the discrete variational formulation over the fictitious domain, weighted by a small penalty parameter $\alpha > 0$. While this method certainly improves the conditioning, it directly alters the (discrete) variational formulation and effectively computes the solution to a different problem. Large stabilization parameters lead to better conditioned systems at the price of larger errors in the solution \cite{garhuom2022eigenvalue,sartorti2025stabilization}. 
    \item Instead, eigenvalue stabilization techniques \cite{garhuom2022eigenvalue,eisentrager2024eigenvalue,sartorti2025stabilization} locally modify the system matrices by stabilizing the small eigenvalues of element matrices prior to assembly. However, this technique also modifies the original formulation, and selecting suitable stabilization parameters may not be obvious. Similarly to the $\alpha$-stabilization method, the error also increases for larger stabilization parameters, although to a milder extent \cite{garhuom2022eigenvalue}.
    \item In the CutFEM community, ghost penalty methods were proposed for improving the conditioning of the stiffness matrix by penalizing the jump of the normal derivatives near the boundaries \cite{burman2012fictitious,burman2015cutfem}. However, just like other penalty methods, large penalty values might undermine the accuracy.
    \item Polynomial extension (or aggregation) techniques locally substitute the polynomial segments of basis functions supported on small trimmed elements by their polynomial extension from good neighboring elements \cite{buffa2020minimal,burman2023extension}. Extended B-splines \cite{hollig2003finite,hollig2005introduction,hollig2003nonuniform,marussig2017stable,marussig2018improved} are a closely related concept that instead combines problematic basis functions with well-behaved ones to form a subspace. All those techniques stabilize the approximation space and have been successfully applied in a wide range of different settings, from the classical Poisson problem \cite{buffa2020minimal,burman2023extension} to challenging applications in explicit dynamics, also in conjunction with mass lumping \cite{burman2022explicit,voet2025stabilization,guarino2025stabilization}.
    \item In contrast, preconditioning techniques \cite{de2017condition,de2019preconditioning,de2020multigrid,jomo2019robust} change neither the variational formulation nor the approximation space but the finite element basis. Therefore, they preserve the original formulation and its solution. However, they only target the ill-conditioning, not the instability that is deeply ingrained in the approximation space.
\end{enumerate}

Instabilities are a major hurdle for time-dependent problems by causing extremely small step sizes in explicit dynamic analyses \cite{stoter2023critical,bioli2025theoretical,hollweck2026analysis}. However, if the only trouble is caused by the ill-conditioning, then there are good reasons for preconditioning the system matrices (i.e., choosing a different basis for the same approximation space) rather than altering the variational formulation or the approximation space itself. Firstly, preconditioning does not interfere with the finite element discretization. The preconditioner is a separate component, which is often built from the discretization but certainly does not alter it. Therefore, its implementation is less intrusive than ghost penalty methods or polynomial extension techniques that meddle with the variational formulation or the approximation space. Secondly, since preconditioning only changes the finite element basis, it preserves all the approximation properties of the original space. Although stabilization or aggregation techniques based on polynomial extensions also ensure good approximation properties, the same cannot be said of ``material'' stabilization techniques that significantly degrade the accuracy \cite{garhuom2022eigenvalue,sartorti2025stabilization}. Thus, preconditioning is an outstanding solution both from the implementation and approximation perspective. 

However, the design of efficient preconditioners is a different story. The subtleties reside in the understanding of the mechanisms that cause ill-conditioning. Two distinct mechanisms were clearly identified in \cite{de2017condition}. The first one is attributed to discrepancies in the scaling of the basis functions, while the second one is tied to near linear dependencies. Scaling discrepancies are easily resolved with a simple diagonal scaling (or Jacobi preconditioner), also for extended finite element methods \cite{lehrenfeld2017optimal}. However, eliminating near linear dependencies is much more difficult and has led to diverging views and strategies. For symmetric positive definite (SPD) matrices arising from immersed finite element or immersogeometric analysis, de Prenter et al.\ \cite{de2017condition} first suggested the Symmetric Incomplete Permuted Inverse Cholesky (SIPIC) preconditioner by combining diagonal scaling with local orthonormalizations of nearly linearly dependent (rescaled) basis functions. Although the method often performs well, it sometimes fails to detect near linear dependencies \cite{de2019preconditioning}. A more general and robust strategy built around similar concepts was later proposed for indefinite or nonsymmetric systems in the form of an additive \cite{de2019preconditioning} or multiplicative \cite{de2020multigrid} Schwarz preconditioner. Unfortunately, this method is also flawed. However, its vulnerability does not stem from the concept itself but rather from the selection strategy for the index blocks defining the preconditioner. Those index blocks must regroup the indices of nearly linearly dependent basis functions that later form the submatrices to be inverted within the Schwarz preconditioner. The block selection strategy first suggested in \cite{de2019preconditioning} sometimes fails to treat near linear dependencies, and more advanced strategies later proposed in \cite{de2020multigrid} suffer the same fate. To this date, the index block selection strategy has not been settled and essentially boils down to a compromise between efficiency and robustness. Many small index blocks are more efficient but may miss some near linear dependencies, while fewer and larger blocks are more robust but less efficient \cite{de2023stability}. 

Of course, discretization parameters such as the mesh size $h$ or polynomial degree $p$ may also contribute to the ill-conditioning of finite element system matrices \cite{gahalaut2012condition,gervasio2020computational}. For resolving those dependencies, some of the aforementioned strategies have been incorporated as smoothers in the multigrid framework \cite{de2020multigrid,jomo2021hierarchical}. However, these techniques address a different problem and clearly deserve a separate investigation. This contribution specifically focuses on the ill-conditioning originating from badly trimmed elements. After carefully examining the effectiveness of Jacobi preconditioning, we provide counter-examples highlighting the limitations of advanced preconditioning techniques. Building on those findings, we later propose a new deflation-based preconditioner that takes care of near linear dependencies in a different way. Deflation-based preconditioning is well-established in the numerical linear algebra community \cite{vuik1999efficient,frank2001construction,vermolen2004deflation,dwarka2022scalable} but has, to the best of our knowledge, never been applied to immersed finite element methods and immersogeometric analysis.

The outline for the rest of the article is as follows: after presenting in \Cref{se: model_problems} some model problems and elementary concepts of unfitted finite element discretizations, we recall in \Cref{se: ill_conditioning_origin} the sources of ill-conditioning for the finite element system matrices before analyzing existing preconditioning strategies in \Cref{se: preconditioning}. The most significant part of our analysis is devoted to the diagonal scaling (or Jacobi) preconditioner for which we derive detailed scaling relations for the condition number in dimensions 1 and 2 for various cut configurations. The insights drawn from this analysis later serve to identify counter-examples for some of the aforementioned preconditioning techniques. Finally, we present in \Cref{se: deflation} a novel deflation-based preconditioner tailored to immersed finite element methods. Deflation techniques are ideally suited when the convergence of iterative solvers is plagued by relatively few small eigenvalues \cite{vuik1999efficient}. Numerical experiments in \Cref{se: numerical_results} confirm its effectiveness by solving problems earlier preconditioners could not handle. Finally, \Cref{se: conclusion} draws the article to a close by summarizing our findings and suggesting directions for future research.

\section{Model problems and discretization}
\label{se: model_problems}
This article examines preconditioners for iteratively solving ill-conditioned linear systems $A\bm{x}=\bm{b}$, where $A$ is a symmetric positive definite (SPD) matrix originating from an immersed finite element discretization. In this setting, let $V_h \subset V$ denote a finite-dimensional subspace (of dimension $n$) of an abstract infinite-dimensional Hilbert space $V$ of functions defined over a domain $\Omega \subset \mathbb{R}^d$. Three common examples for $A\bm{x}=\bm{b}$ are listed below.

\begin{itemize}
    \item Finite element discretizations of elliptic PDEs commonly require finding a solution $u_h \in V_h$ of a discrete variational problem
    \begin{equation}
    \label{eq: discrete_weak_form_elliptic}
        a(u_h,v_h)=F(v_h) \qquad \forall v_h \in V_h
    \end{equation}
    where $a \colon V \times V \to \mathbb{R}$ is a continuous and coercive (or elliptic) bilinear form and $F \colon V \to \mathbb{R}$ is a linear functional. By expanding the discrete solution $u_h$ in a basis $\Phi=\{\varphi_1,\dots,\varphi_n\}$ of $V_h$ and testing against every basis function, \eqref{eq: discrete_weak_form_elliptic} is equivalent to a linear system $K\bm{u}=\bm{f}$ involving a ``stiffness'' matrix $K \in \mathbb{R}^{n \times n}$. For instance, for the standard Poisson problem, 
    \begin{equation*}
        K_{ij} = \int_{\Omega} \nabla \varphi_i \cdot \nabla \varphi_j \dx.
    \end{equation*}
    \item The $L^2$ projection of a function $f \in L^2(\Omega)$ is the function $u_h \in V_h$ that solves the variational problem
    \begin{equation}
    \label{eq: L^2_projection}
    (u_h,v_h)_{L^2(\Omega)} = (f,v_h)_{L^2(\Omega)} \qquad \forall v_h \in V_h.
    \end{equation}
    Similarly to \eqref{eq: discrete_weak_form_elliptic}, this problem is again equivalent to solving a linear system $M\bm{u}=\bm{f}$ involving this time a ``mass'' matrix, whose entries are given by
    \begin{equation*}
        M_{ij} = \int_{\Omega} \varphi_i \varphi_j \dx.
    \end{equation*}
    Analogous systems arise in explicit time integration of finite element discretizations of parabolic or hyperbolic PDEs \cite{quarteroni2009numerical,hughes2012finite,bathe2006finite}.  
    \item Conversely, implicit time integration typically requires solving a linear system $A\bm{x}=\bm{b}$, where $A=a_1M+a_2K$ is a linear combination of the mass and stiffness matrices with $a_1,a_2>0$. Oftentimes, $a_1=1$ while $a_2$ is related to the step size $\Delta t$ and a parameter $\alpha$ of the method. For instance, $a_2=\alpha \Delta t$ for parabolic PDEs, while $a_2=\alpha \Delta t^2$ for hyperbolic ones \cite{quarteroni2009numerical,hughes2012finite,bathe2006finite}.
\end{itemize}

In summary, all three cases lead to a linear system $A\bm{x}=\bm{b}$, where $A$ is the Gram matrix of a certain symmetric elliptic bilinear form $a$ in a basis $\Phi$; i.e., $A_{ij}=a(\varphi_i,\varphi_j)$. This bilinear form induces a norm $\|v_h\|_a = \sqrt{a(v_h,v_h)}$, commonly called the ``energy'' or operator norm (coinciding with the $L^2$ norm if $A=M$ and the $H^1$ seminorm if $A=K$). Distinct finite element methods originate from distinct approximation spaces $V_h$, with necessarily distinct bases $\Phi$ that ultimately affect the conditioning of $A$. Two common constructions are recalled below.

\begin{itemize}
    \item In classical finite element methods, the approximation space $V_h$ is a continuous space of piecewise polynomials constructed over a number of finite elements $T_i$ that discretize $\Omega$. In 1D, we form a mesh 
    \begin{equation*}
        \mathcal{T}_h = \{[\xi_{i-1},\xi_{i}] \colon i = 1,\dots,N_s\}
    \end{equation*}
    that partitions $\Omega=(a,b)$ into $N_s$ subintervals (or elements) $T_i=[\xi_{i-1},\xi_{i}]$ with
    \begin{equation*}
        a=\xi_0 < \xi_1 < \dots < \xi_{N_s-1} < \xi_{N_s} = b
    \end{equation*}
    and defines the approximation space of degree $p$ as
    \begin{equation*}
        V_h = \{v_h \in C^0(\Omega)\colon v_h|_{T} \in \mathbb{P}_p \ \forall T \in \mathcal{T}_h\}.
    \end{equation*}
    The Lagrange interpolating polynomials are a popular basis for $V_h$. They are constructed locally on each element and later coupled together across neighboring elements. Given a generic element $T_i=[\xi_{i-1},\xi_{i}]$ and a set $\{\hat{\xi}_j\}_{j=0}^p \subset T_i$ of distinct interpolation points (with $\hat{\xi}_0=\xi_{i-1}$ and $\hat{\xi}_p = \xi_{i}$), the Lagrange interpolating polynomials on $T_i$ are defined as
    \begin{equation}
    \label{eq: lagrange_polynomials}
        L_i(\xi) = \prod_{\substack{j=0 \\ j \neq i}}^p \frac{(\xi-\hat{\xi}_j)}{(\hat{\xi}_i - \hat{\xi}_j)} \qquad i=0,\dots,p.
    \end{equation}
    In higher dimensions, finite elements may take various shapes, but we will mostly focus on tensor product elements formed by quadrilaterals in 2D and hexahedra in 3D that simply repeat the same construction along separate directions. In dimension $d$, we obtain tensor product basis functions
    \begin{equation*}
        L_{\bm{i}}=L_{1,i_1} \dots L_{d,i_d}
    \end{equation*}
    where $L_{k,i}$ denotes the $i$th function along the $k$th direction and $\bm{i}=(i_1,\dots,i_d)$ is a multi-index. 

    \item In isogeometric analysis (IGA), the approximation space $V_h$ is a space of smooth piecewise polynomials (called splines). Splines are an established technology in computer-aided-design (CAD), and originally, the main reason for adopting them in finite element analysis (FEA) was to facilitate communication between FEA and CAD by using a single model for both. Later, the intrinsic smoothness of spline functions stood out as an important property, which has greatly contributed to the success of IGA. Among the profusion of spline functions, B-splines are certainly the most popular. In 1D, they are constructed recursively from a knot vector $\Xi = (\xi_1,\dots,\xi_{n+p+1})$ forming a sequence of non-decreasing real numbers. Analogously to classical FEM, distinct knots partition $\Omega$ into knot spans rather than elements. However, contrary to the Lagrange basis functions, B-splines may have increased smoothness between knot spans, an attribute directly controlled by the knot multiplicity. This property is recalled in the next theorem and will have important consequences in the analysis.

    \begin{theorem}[{\cite[Theorem 3.19]{lyche2018spline}}]
    \label{thm: spline_continuity}
    Suppose that the knot $\xi$ occurs $m$ times among the knots $\xi_i, \xi_{i+1}, \dots, \xi_{i+p+1}$, defining the B-spline $B_i$. Then, if $1 \leq m \leq p$,
    \begin{equation*}
        B_i \in C^{p-m}(\xi) \setminus C^{p+1-m}(\xi).
    \end{equation*}
    \end{theorem}
    Thus, varying the multiplicity $1 \leq m \leq p$ of a knot allows constructing $C^k$ continuous spline spaces, where $k=p-m$. Distinct knot values partition $\Omega$ into knot spans (or subintervals) $T_i$ that form a mesh $\mathcal{T}_h$ and analogously to standard FEM we define
    \begin{equation*}
        V_h = \{v_h \in C^k(\Omega)\colon v_h|_{T} \in \mathbb{P}_p \ \forall T \in \mathcal{T}_h\}.
    \end{equation*}
    The standard finite element space is recovered for $k=0$ and the B-spline basis in this case reduces to the so-called Bernstein basis. In practice, open knot vectors are commonly adopted whereby the first and last knots are repeated $p+1$ times, thereby ensuring interpolatory basis functions at the boundaries. Quadratic Lagrange and B-spline bases are exemplarily shown in \Cref{fig: 1D_bases_p2_n4} for $N_s=4$ subdivisions. As usual, a tensor product argument allows generalizing the construction to multiple dimensions. For a more detailed introduction to splines and IGA, interested readers may consult \cite{lyche2018spline,kunoth2018splines,cottrell2009isogeometric} among many other references.
\end{itemize}

\begin{figure}[H]
     \centering
     \begin{subfigure}[t]{0.31\textwidth}
    \centering
    \includegraphics[width=\textwidth]{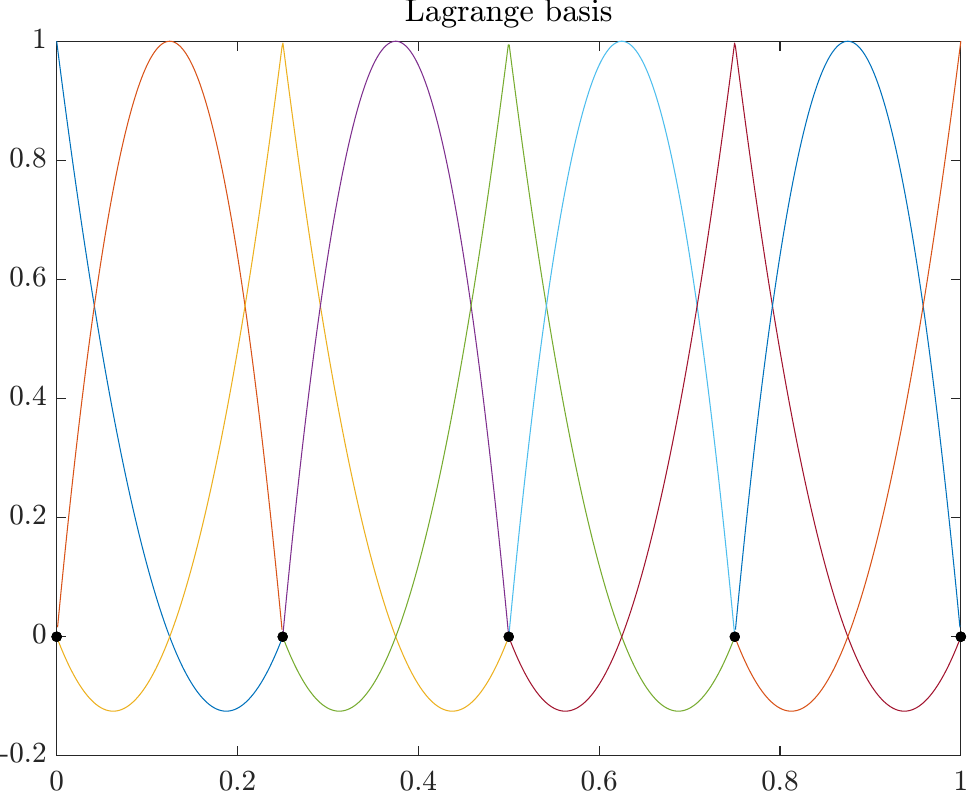}
    \caption{Lagrange basis}
    \label{fig: 1D_Lagrange_basis_p2_C0_n4}
     \end{subfigure}
     \hfill
    \begin{subfigure}[t]{0.31\textwidth}
    \centering
    \includegraphics[width=\textwidth]{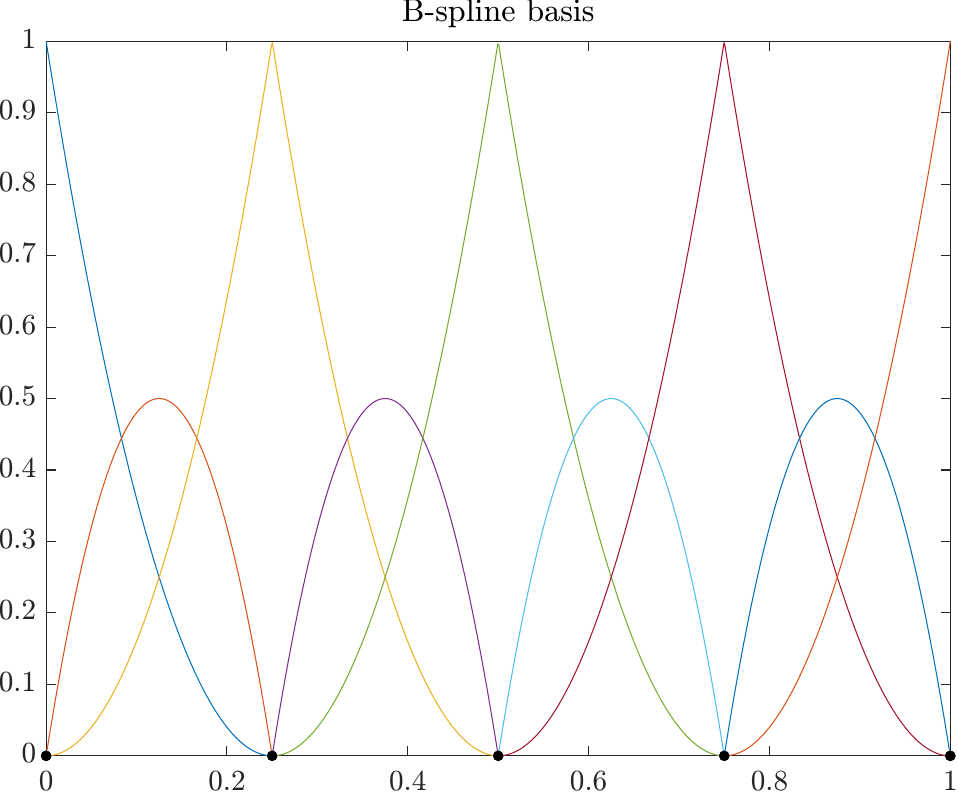}
    \caption{$C^0$ B-spline basis}
    \label{fig: 1D_Bspline_basis_p2_C0_n4}
     \end{subfigure}
     \hfill
    \begin{subfigure}[t]{0.31\textwidth}
    \centering
    \includegraphics[width=\textwidth]{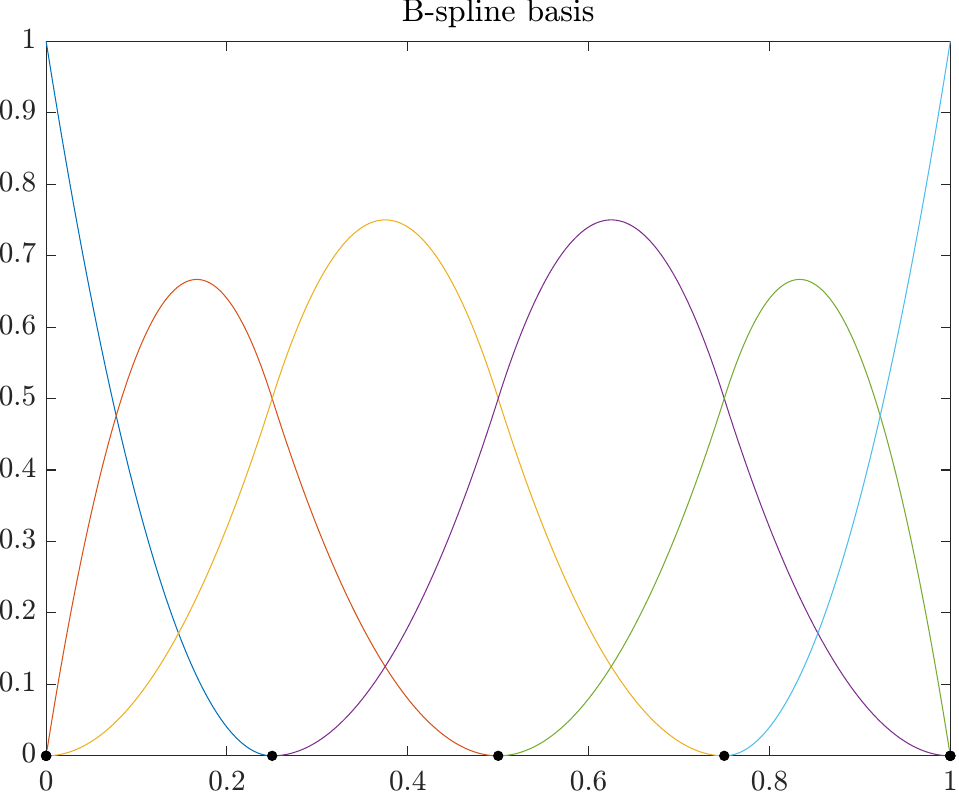}
    \caption{$C^1$ B-spline basis}
    \label{fig: 1D_Bspline_basis_p2_Cp-1_n4}
     \end{subfigure}
     \caption{Quadratic Lagrange and B-spline bases for $N_s=4$ subdivisions}
    \label{fig: 1D_bases_p2_n4}
\end{figure}

Contrary to standard finite element discretizations, spline functions may exactly represent complex curved shapes, including conic sections, which explains their widespread usage in CAD. This same property sometimes allows creating boundary-fitted meshes in IGA without committing any geometric discretization error and was originally one of the method's main assets. However, in most engineering applications, the geometry is often too complex to be represented exactly, and its boundary is instead implicitly defined through intersections of spline curves or surfaces that carve out regions from a simpler geometry. This process, called trimming, is routinely used in CAD but is not without causing any trouble in analysis. Indeed, trimmed geometries are rarely watertight, and due to their complexity, creating boundary-fitted meshes is often infeasible or excessively time-consuming, whether with standard finite elements or with IGA. Immersed finite element methods may then emerge as an attractive alternative. Loosely speaking, immersed methods embed a complex physical domain $\Omega$ in a much simpler fictitious domain $\widehat{\Omega}$ that is easily meshed. The basis functions are then constructed over the background mesh before retaining those whose support intersects the physical domain. The concept is plainly illustrated in \Cref{fig: ring_mesh}, where a quarter of an annulus is discretized with a boundary-fitted mesh in \Cref{fig: 2D_Laplace_ring_boundary_fitted_mesh} and with an unfitted mesh in \Cref{fig: 2D_Laplace_ring_unfitted_mesh}.

\begin{figure}[H]
    \centering
    \begin{subfigure}{0.49\textwidth}
    \centering
    \includegraphics[width=0.9\textwidth]{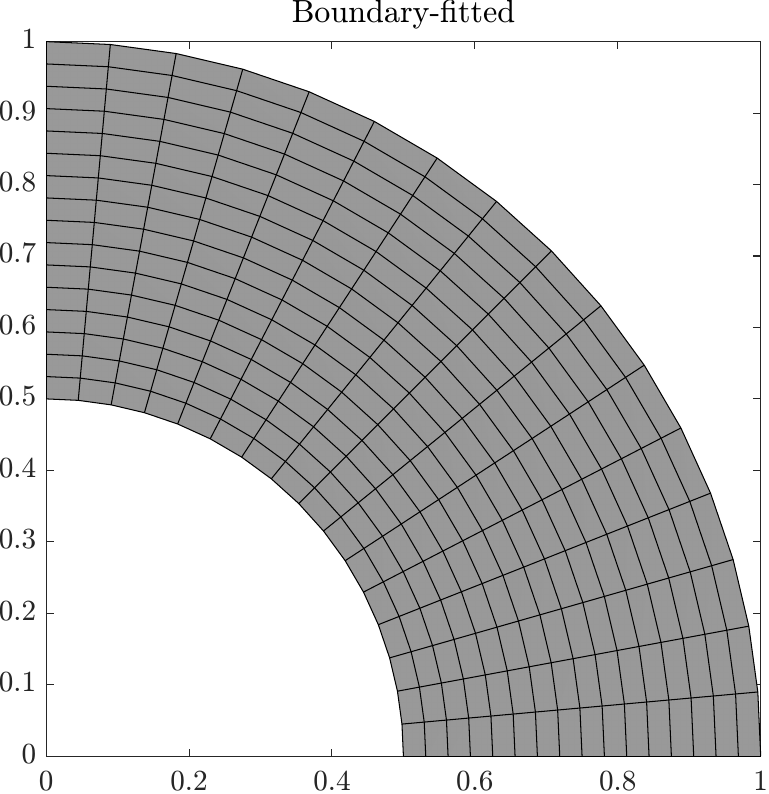}
    \caption{Boundary-fitted mesh}
    \label{fig: 2D_Laplace_ring_boundary_fitted_mesh}
    \end{subfigure}
    \hfill
    \begin{subfigure}{0.49\textwidth}
    \centering
    \includegraphics[width=0.9\textwidth]{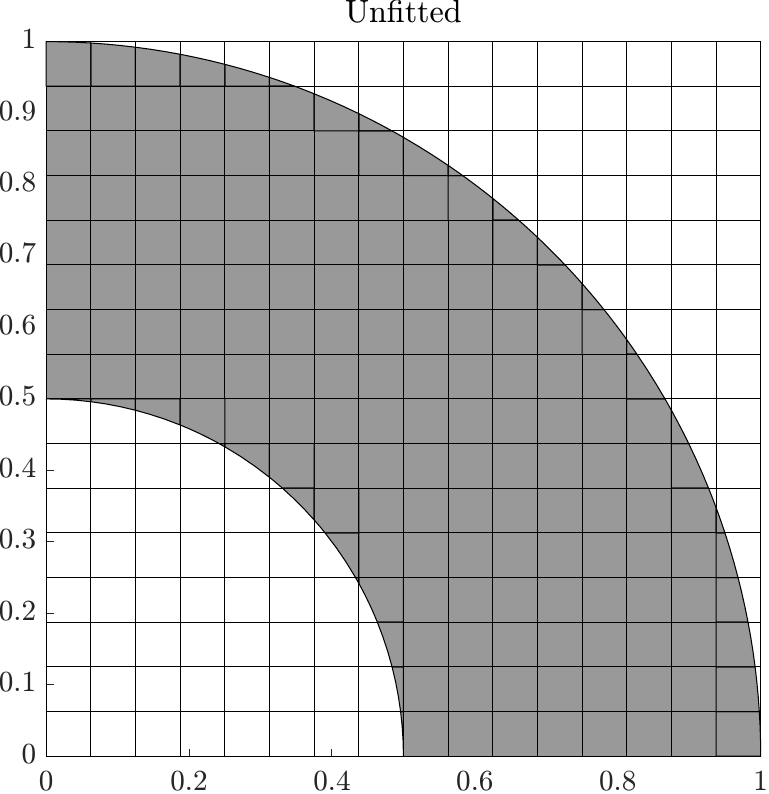}
    \caption{Unfitted mesh}
    \label{fig: 2D_Laplace_ring_unfitted_mesh}
    \end{subfigure}
    \caption{Quarter of annulus}
    \label{fig: ring_mesh}
\end{figure}
At a first glance, immersed methods are a remarkably simple workaround to an otherwise dreadful meshing problem. However, while it certainly alleviates the meshing issue, it instead introduces serious difficulties elsewhere:

\begin{enumerate}[noitemsep]
    \item Firstly, the integration on cut elements requires special integration rules or techniques. Some of the solutions proposed include quadtree/octree subdivisions \cite{verhoosel2015image,duster2008finite} or reparameterizations of trimmed elements \cite{antolin2019isogeometric}.
    \item Secondly, strongly imposing essential boundary conditions on trimmed boundaries becomes impossible, and they must either be weakly imposed via Nitsche's method \cite{de2018note,burman2012fictitious}, penalization techniques \cite{stoter2023critical} or enforced through Lagrange multipliers \cite{burman2010fictitious}. However, none of these methods are perfect, as they either introduce unknown parameters or require identifying suitable multiplier spaces.
    \item Thirdly, immersed methods may compromise the numerical stability of the problem by causing continuity constants to blow up \cite{buffa2020minimal} or stable step sizes to nearly vanish \cite{stoter2023critical,bioli2025theoretical,hollweck2026analysis}. Extended B-splines \cite{hollig2005introduction,marussig2017stable} and other polynomial extension techniques \cite{buffa2020minimal,burman2023extension} that (locally) modify the approximation space are a natural remedy for this problem.
    \item Finally, the existence of arbitrarily small cut elements may cause arbitrarily ill-conditioned system matrices \cite{de2017condition,de2019preconditioning,de2023stability}.
\end{enumerate}

This last issue has already been carefully examined in several contributions and may be alleviated in multiple ways. The most prominent ones include fictitious domain stabilization (or $\alpha$-stabilization) \cite{parvizian2007finite,schillinger2015finite}, eigenvalue stabilization \cite{garhuom2022eigenvalue,eisentrager2024eigenvalue}, polynomial extension techniques \cite{buffa2020minimal,hollig2005introduction,marussig2017stable}, preconditioning \cite{de2017condition,de2019preconditioning,de2023stability} and specific combinations \cite{eisentrager2024eigenvalue}. Although these techniques share common goals, they target different quantities. Fictitious domain stabilization modifies the variational formulation (and the problem that is ultimately solved) by integrating over the fictitious domain and weighting it by a small parameter $\alpha > 0$. Eigenvalue stabilization preserves the original formulation but instead alters the system matrices by locally stabilizing near zero eigenvalues. In contrast, polynomial extension techniques (including extended B-splines) modify the approximation space by extending polynomial segments into badly trimmed elements. Finally, preconditioning chooses a different basis for the original approximation space. Therefore, it certainly cannot treat the aforementioned instabilities that are rooted in the approximation space. Nevertheless, if the only concern is the ill-conditioning, there are good reasons to opt for preconditioning, primarily from the implementation point of view. Indeed, preconditioning is merely an add-on to an iterative solver and does not meddle with the finite element discretization. Therefore, it is certainly one of the least intrusive methods and may be easily incorporated in existing finite element software. This article exclusively focuses on the ill-conditioning issue. The next section recalls its origin before examining existing preconditioning techniques, highlighting their limitations, and later proposing a novel strategy.

\section{Sources of ill-conditioning}
\label{se: ill_conditioning_origin}
In FEA, the conditioning of a Gram matrix is directly related to the stability of the basis $\Phi$ \cite{lyche2018spline}. Since $V_h$ is finite-dimensional, there exist two constants $c_1,c_2>0$ such that
\begin{equation}
\label{eq: basis_stability}
c_1 \|\bm{v}\|_2^2 \leq \|v_h\|_a^2 \leq c_2 \|\bm{v}\|_2^2 \qquad \forall v_h \in V_h
\end{equation}
where $\bm{v}$ is the coefficient vector of $v_h$ in the basis $\Phi$. Since $\|v_h\|_a^2=\bm{v}^TA\bm{v}$, the constants $c_1$ and $c_2$ that measure the strength of the equivalence between the energy and Euclidean norms also directly yield an estimate of the condition number. Herein, the condition number of an SPD matrix $A \in \mathbb{R}^{n \times n}$ (or similar to an SPD matrix) is defined as
\begin{equation*}
    \kappa(A) = \frac{\lambda_n(A)}{\lambda_1(A)}
\end{equation*}
where $\lambda_1(A)$ and $\lambda_n(A)$ denote the smallest and largest eigenvalues of $A$, respectively. The condition number $\kappa(A)$ is an important quantity in numerical analysis since it directly enters the convergence bounds of classical iterative solvers, such as the conjugate gradient method (see e.g., \cite{greenbaum1997iterative,saad2003iterative}). From the variational characterization of eigenvalues \cite[Theorem 4.2.6]{horn2012matrix} and \eqref{eq: basis_stability}, we immediately obtain
\begin{equation*}
    \kappa(A) \leq \frac{c_2}{c_1}.
\end{equation*}
Thus, $A$ is well-conditioned if $c_2/c_1$ remains moderate. For immersed finite element discretizations, the constant $c_2$ is clearly bounded independently of the trimming configuration. Indeed, by appending a subscript to the energy norm for marking the dependency on the integration domain, we immediately obtain
\begin{equation*}
    \|v_h\|_{a(\Omega)}^2 \leq \|v_h\|_{a(\widehat{\Omega})}^2 \leq c_2 \|\bm{v}\|_2^2
\end{equation*}
where $c_2$ is the stability constant of the spline basis constructed over $\widehat{\Omega}$. Thus, $c_2$ remains bounded from above independently of the trimming configuration\footnote{Clearly, $c_2$ is also bounded from below by choosing a specific function in \eqref{eq: basis_stability} nowhere near the trimmed boundary.}. However, this may not hold if the bilinear form accounts for additional terms resulting from the weak imposition of Dirichlet boundary conditions with Nitsche-type methods \cite{de2017condition}. Indeed, global penalization parameters may indirectly introduce a dependency on the trimming configuration. Thankfully, this issue does not occur when the penalization parameters are instead chosen locally \cite{de2017condition}. Therefore, the weak imposition of boundary conditions is not the main cause of ill-conditioning, and we ignore it altogether in the remaining part of this work. Instead, our greatest concern is tied to the constant $c_1$ since nothing guarantees that it remains bounded from below. As a matter of fact, there might exist a function $v_h \in V_h$ such that $\|v_h\|_a \ll \|\bm{v}\|_2$, in which case $c_1$ is necessarily very small and $A$ is most likely ill-conditioned. In immersed finite element methods, this problem is caused by basis functions that are only supported on small trimmed elements \cite{de2023stability}. In \cite{de2017condition}, the authors identified two distinct mechanisms causing ill-conditioning. The first one is attributed to a poor scaling of the basis functions, with energy norm values spanning several orders of magnitude. A simple 1D setup suffices to illustrate the problem. Consider a uniform partition of the unit line $\widehat{\Omega}=(0,1)$ and let $\Omega=(0,\xi_l+\delta)$ denote the physical domain whose right boundary is at a distance $\delta > 0$ from the breakpoint $\xi_l$, as depicted in \Cref{fig: 1D_illustration}. The basis over the physical domain $\Omega$ is a subset of the basis functions over the fictitious domain $\widehat{\Omega}$. Those basis functions are exemplarily shown in \Cref{fig: 1D_Bspline_basis_p2_C1,fig: 1D_Bspline_basis_p2_C0} for quadratic B-spline bases of $C^1$ and $C^0$ continuities, respectively, constructed over a mesh of $N_s=8$ elements. In both cases, $\xi_l=0.75$ is the last knot within the shaded physical domain $\Omega$. In the $C^1$ case, one of the basis functions becomes inactive over $\Omega$ and is marked in gray in \Cref{fig: 1D_Bspline_basis_p2_C1}. All other functions remain active, but one of them is only supported on the small trimmed element $[\xi_l, \xi_l+\delta)$. Therefore, its energy norm may become arbitrarily small as $\delta \to 0$ while the Euclidean norm of its associated coefficient vector remains identically $1$. Consequently, $\lambda_1(A)$ may become \emph{arbitrarily} small and sends the condition number skyrocketing. The same issue obviously arises for Lagrange basis functions. The growth of the condition number was estimated by de Prenter et al.\ \cite{de2017condition} for regular trimmed elements (also in higher dimensions), independently of the finite element basis. For the stiffness matrix of second-order PDEs, the authors found that
\begin{equation}
\label{eq: cond_K}
    \kappa(K) \gtrsim \eta^{-(2p+1-2/d)},
\end{equation}
where $\eta = \min_{T \in \mathcal{T}_h} |T \cap \Omega|/|T|$ denotes the smallest volume fraction. For our 1D example above, $\eta = \delta/h$. With similar arguments (and identical assumptions), one can easily prove that for the mass matrix
\begin{equation}
\label{eq: cond_M}
    \kappa(M) \gtrsim \eta^{-(2p+1)}.
\end{equation}
For spline bases, those rates are quickly recovered by simply evaluating the energy norm of the smoothest basis function near the trimmed boundary \cite{ballotta2023preconditioners,de2023stability}. They are confirmed in \Cref{fig: 1D_Laplace_trimming_cond_no_prec_Bspline_C0} and were systematically observed in all our examples. As shown in \Cref{fig: 1D_Laplace_trimming_cond_no_prec_Bspline_C0}, the condition number rapidly reaches prohibitively large values and requires effective preconditioning. In the maximally smooth case, rescaling the basis functions by their energy norm is sometimes a simple and effective solution, as we will later prove. Unfortunately, the solution is far less obvious in case multiple basis functions are supported on the same trimmed element, as for the $C^0$ case in \Cref{fig: 1D_Bspline_basis_p2_C0}. Indeed, even after rescaling, there could potentially still exist a linear combination of basis functions that nearly vanishes on the trimmed element. This was the second mechanism identified in \cite{de2017condition}. However, this intuition hides many subtleties, and to the best of our knowledge, the effectiveness of rescaling was never satisfactorily explained in the literature. The next section is meant to fill the gaps. We begin by carefully analyzing the condition number after rescaling (or Jacobi preconditioning) before examining more advanced preconditioning techniques. Note that in the finite element context, preconditioning seeks an alternative basis that strengthens the norm equivalence in \eqref{eq: basis_stability}.

\begin{figure}[H]
    \centering
    \includegraphics[width=0.5\textwidth]{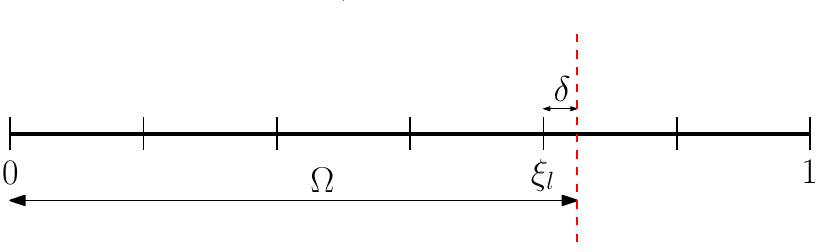}
    \caption{Trimmed 1D line}
    \label{fig: 1D_illustration}
\end{figure}

\begin{figure}[H]
    \centering
    \begin{subfigure}{0.48\textwidth}
    \centering
    \includegraphics[width=\textwidth]{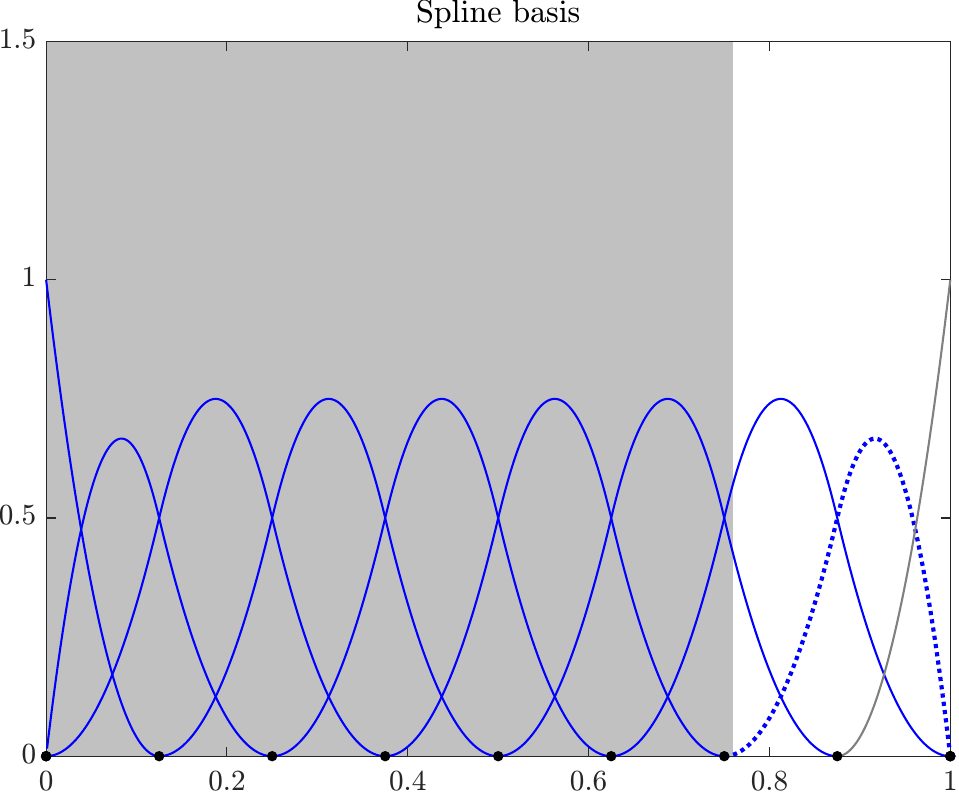}
    \caption{$C^1$ smoothness}
    \label{fig: 1D_Bspline_basis_p2_C1}
    \end{subfigure}
    \hfill
    \begin{subfigure}{0.48\textwidth}
    \centering
    \includegraphics[width=\textwidth]{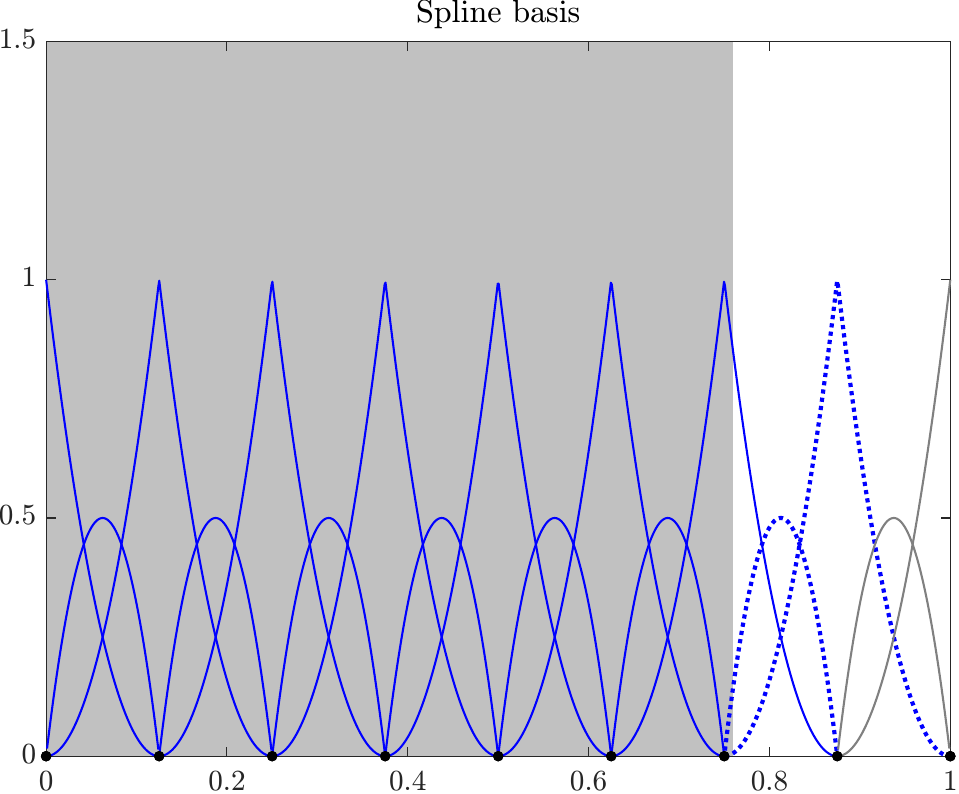}
    \caption{$C^0$ smoothness}
    \label{fig: 1D_Bspline_basis_p2_C0}
    \end{subfigure}
    \caption{Quadratic spline bases on the fictitious domain $(0,1)$. Active basis functions are shown in blue, inactive ones in gray. Among active basis functions, some are supported on non-trimmed elements (solid blue), while others are only supported on small trimmed elements (dashed blue).}
    \label{fig: 1D_Bspline_basis_p2}
\end{figure}

\begin{figure}[H]
    \centering
    \begin{subfigure}{0.48\textwidth}
    \centering
    \includegraphics[width=\textwidth]{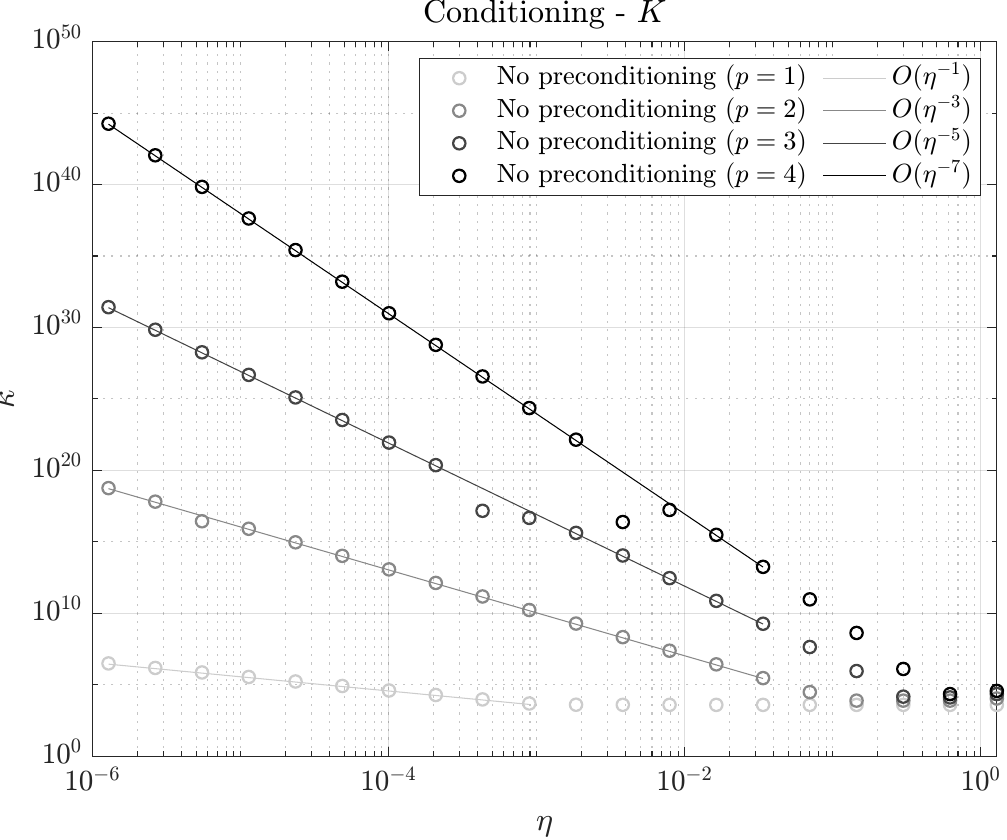}
    \caption{$\kappa(K)$}
    \label{fig: 11D_Laplace_trimming_cond_K_no_prec_Bspline_C0}
    \end{subfigure}
    \hfill
    \begin{subfigure}{0.48\textwidth}
    \centering
    \includegraphics[width=\textwidth]{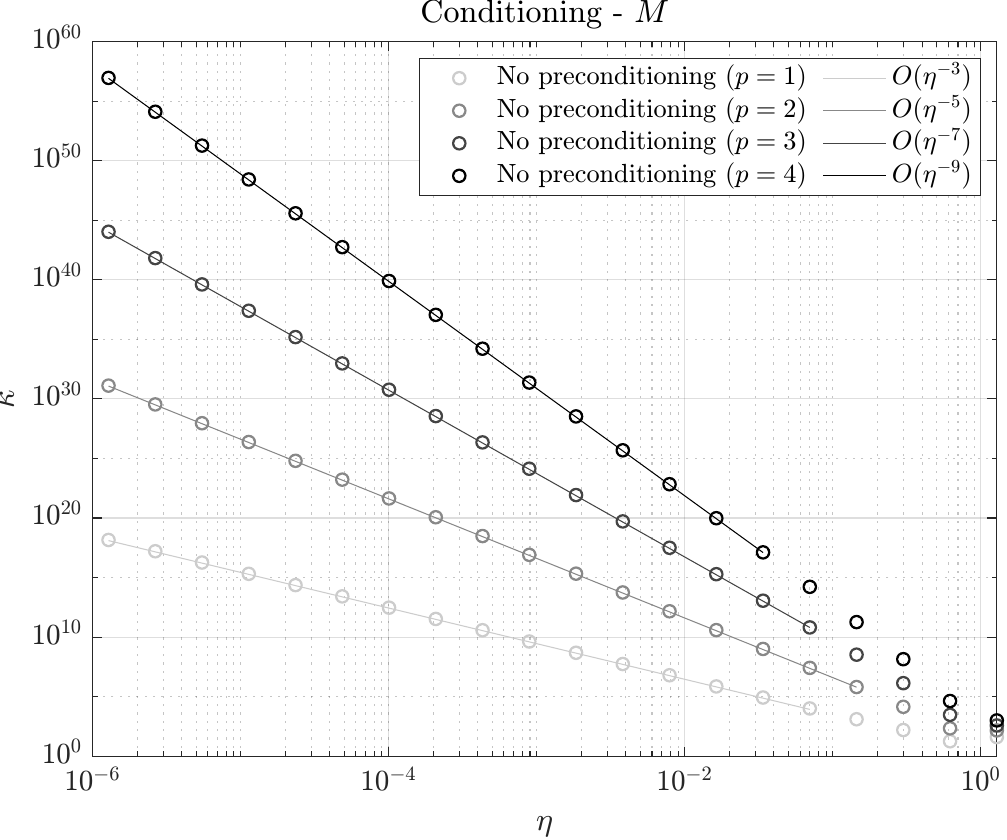}
    \caption{$\kappa(M)$}
    \label{fig: 1D_Laplace_trimming_cond_M_no_prec_Bspline_C0}
    \end{subfigure}
    \caption{Condition numbers of the stiffness and mass matrices in the Bernstein basis for the setup in \Cref{fig: 1D_illustration} with $128$ subdivisions and spline degrees ranging from $1$ to $4$.}
    \label{fig: 1D_Laplace_trimming_cond_no_prec_Bspline_C0}
\end{figure}

\section{Existing preconditioning strategies}
\label{se: preconditioning}
\subsection{Diagonal scaling}
\label{se: diagonal_scaling}
The diagonal scaling (or Jacobi) preconditioner is the most straightforward preconditioning strategy and simply amounts to forming the Gram matrix of the bilinear form $a$ in the rescaled basis
\begin{equation*}
    \hat{\Phi} = \{\hat{\varphi}_1,\dots,\hat{\varphi}_n\},
\end{equation*}
where $\hat{\varphi}_j=\varphi_j/\|\varphi_j\|_a$. For the variational problems of \Cref{se: model_problems}, the Jacobi preconditioned system becomes $\hat{A}\hat{\bm{x}}=\hat{\bm{b}}$, where
\begin{equation*}
    \hat{A}_{ij} = a(\hat{\varphi}_i,\hat{\varphi}_j), \quad \hat{b}_i = F(\hat{\varphi}_i),
\end{equation*}
and $\hat{\bm{x}}$ is the coefficient vector of the solution in the rescaled basis. However, in order to accommodate arbitrary right-hand sides, we will adopt a more algebraic viewpoint, where
\begin{equation*}
    \hat{A}=D^{-1}AD^{-1}, \quad \hat{\bm{x}}=D\bm{x}, \quad \text{and} \quad \hat{\bm{b}}=D^{-1}\bm{b},
\end{equation*}
where $D=\sqrt{\diag(A)}$ denotes the Jacobi preconditioner. By ensuring that all basis functions have unit operator norm, the Jacobi preconditioner effectively eliminates the first source of ill-conditioning and is (nearly) the best diagonal preconditioner one can hope for \cite{van1969condition}. However, near linear dependencies among rescaled basis functions might survive, and this strategy alone is generally insufficient \cite{de2017condition}. In \cite{de2017condition,de2023stability}, the authors argued that diagonal scaling was often sufficient for maximally smooth spline discretizations when the first source of ill-conditioning usually prevails. But when multiple basis functions are supported on small trimmed elements, near linear dependencies might arise and the Jacobi preconditioner may no longer be sufficient, although it still significantly reduces the condition number. As far as we can tell, the arguments laid out in \cite{de2023stability} were merely intuitive, and the conditions under which diagonal scaling is robust are far more subtle than engineers like to believe. While continuity is certainly one of the main drivers, it does not fully explain the behavior of the eigenvalues of the Jacobi preconditioned matrix. In particular, our analysis reveals how the choice of basis, the cut configuration (in higher dimensions), and even the location of the knots or interpolation points may also critically influence the eigenvalues of the preconditioned matrix. The next two sections provide detailed analyses in 1D and 2D, where various trimming configurations are considered for the latter.

\subsubsection{1D analysis}
\label{se: 1D_analysis}
The trimmed 1D line shown in \Cref{fig: 1D_illustration} is a simple, albeit instructive, example to begin with. The physical domain $\Omega=(0,\xi_l+\delta) \subset \widehat{\Omega}$ is obtained by trimming the right boundary of the unit line $\widehat{\Omega}=(0,1)$ for some $\delta>0$. This example is obviously contrived since trimming operations are completely unnecessary in 1D. Nevertheless, this preliminary analysis is not without a few surprises and will prove useful in more complicated future settings.

As explained in \Cref{se: ill_conditioning_origin}, ill-conditioning stems from basis functions that are only supported on small trimmed elements. At this stage, it is already important to distinguish the classical Lagrange basis from (smooth) spline bases. Those two cases are treated separately in the next two sections to highlight important differences. \\

\noindent \textbf{Lagrange basis}:
For the Lagrange basis of degree $p$, there exist $p$ basis functions $\{L_1,\dots,L_p\}$ that are only supported on $\Omega_l = [\xi_l,\xi_l+\delta)$. Moreover, since they are only $C^0$ at $\xi_l$ (and not $C^1$), the Taylor form of the Lagrange polynomials restricted to $\Omega_l$ is
\begin{equation*}
    \restr{L_i}{\Omega_l}(\xi) = \sum_{j=1}^p c_{ij}(\xi-\xi_l)^j \qquad i=1,\dots,p
\end{equation*}
with $c_{i1} \neq 0$ for all $i=1,\dots,p$. Consequently, we deduce that
\begin{align}
    \|L_i\|_{L^2}^2 &= \int_{\Omega_l} (L_i(\xi))^2 \dxi \sim \delta^{3} & &\text{as } \delta \to 0, \label{eq: L2_Lagrange} \\
    \|L_i'\|_{L^2}^2 &= \int_{\Omega_l} (L_i'(\xi))^2 \dxi \sim \delta & &\text{as } \delta \to 0, \label{eq: H1_Lagrange}
\end{align}
independently of $i$. Now consider the polynomial $u$ defined as
\begin{equation*}
    u(\xi) =
    \begin{cases}
        (\xi-\xi_l)^p & \text{if } \xi \geq \xi_l, \\
        0 & \text{otherwise}.
    \end{cases}
\end{equation*}
We first show that $\restr{u}{\Omega}$ (the restriction of $u$ to $\Omega$) is a function in $V_h$. Indeed, since $\restr{u}{\Omega_l}$ is a degree $p$ polynomial, there exist coefficients $u_j$ independent of $\delta$ such that
\begin{equation*}
    \restr{u}{\Omega_l}(\xi) = (\xi-\xi_l)^p = \sum_{j=0}^p u_j L_j(\xi).
\end{equation*}
Denoting $\{\hat{\xi}_j\}_{j=0}^p$ the local Lagrange interpolation points (with $\hat{\xi}_0=\xi_l$ and $\hat{\xi}_p=\xi_{l+1}$), the coefficients $u_j$ are simply $u_j = u(\hat{\xi}_j)$. Consequently, $u_0=0$ and $\restr{u}{\Omega_l}$ is merely a linear combination of $\{L_1,\dots,L_p\}$ and, therefore, $\restr{u}{\Omega} \in V_h$. In the rescaled basis, this function becomes
\begin{equation*}
    \restr{u}{\Omega_l}(\xi) = \sum_{j=1}^p \hat{u}_j \hat{L}_j(\xi),
\end{equation*}
where $\hat{u}_j = \alpha_ju_j$ with $\alpha_j=\|L_j\|_a$. Denoting $\hat{\bm{u}}$ the coefficient vector of $\restr{u}{\Omega}$ in the rescaled Lagrange basis, we obtain
\begin{equation}
\label{eq: upper_bound_lambda_min}
    \lambda_1(\hat{A}) = \min_{\substack{\hat{\bm{v}} \in \mathbb{R}^n \\ \hat{\bm{v}} \neq 0}} \frac{\hat{\bm{v}}^T\hat{A}\hat{\bm{v}}}{\|\hat{\bm{v}}\|_2^2} \leq \frac{\hat{\bm{u}}^T\hat{A}\hat{\bm{u}}}{\|\hat{\bm{u}}\|_2^2} = \frac{\|u\|_a^2}{\|\hat{\bm{u}}\|_2^2}.
\end{equation}
The scaling of the numerator and denominator depends on the energy norm $a$. For the numerator,
\begin{align*}
    \|u\|_{L^2}^2 &= \int_{\Omega_l} (\xi-\xi_l)^{2p} \dxi \sim \delta^{2p+1} & &\text{as } \delta \to 0, \\
    \|u'\|_{L^2}^2 &= \int_{\Omega_l} p^2(\xi-\xi_l)^{2(p-1)} \dxi \sim \delta^{2p-1} & &\text{as } \delta \to 0,
\end{align*}
while for the denominator, $\|\hat{\bm{u}}\|_2^2 = \sum_{j=1}^p u_j^2 \|L_j\|_a^2$. Combining those results with \cref{eq: L2_Lagrange,eq: H1_Lagrange} and recalling that the coefficients $u_j$ are independent of $\delta$, we finally obtain
\begin{equation*}
    \lambda_1(\hat{M}) \lesssim \frac{\delta^{2p+1}}{\delta^3} = \delta^{2(p-1)} \quad \text{and} \quad \lambda_1(\hat{K}) \lesssim \frac{\delta^{2p-1}}{\delta} = \delta^{2(p-1)}.
\end{equation*}
Since the largest eigenvalue is independent of $\delta$, we obtain in either case
\begin{equation*}
    \kappa(\hat{A}) \gtrsim \delta^{-2(p-1)}.
\end{equation*}
Moreover, since the operator norm of each basis function on $\Omega_l$ scales identically with respect to $\delta$ (see \cref{eq: L2_Lagrange,eq: H1_Lagrange}), and the denominator in \eqref{eq: upper_bound_lambda_min} does not allow for potential cancellation, it will always scale the same, regardless of the coefficients $u_j$ (as long as they are independent of $\delta$). Furthermore, since the numerator in \eqref{eq: upper_bound_lambda_min} is as small as it gets, we can even state that $\kappa(\hat{A}) \sim \delta^{-2(p-1)}$. Thus, in 1D, the condition number of the Jacobi preconditioned mass and stiffness matrices scales (at least) as $\delta^{-2(p-1)} \sim \eta^{-2(p-1)}$ in the Lagrange basis. \\

\noindent \textbf{B-spline basis}:
We now turn to (smooth) spline bases. In this case, the number of basis functions only supported on $\Omega_l$ depends of the local continuity across the rightmost knot $\xi_l$ in \Cref{fig: 1D_illustration}. More specifically, if $\xi_l$ occurs $m$ times in the knot vector (with $1 \leq m \leq p$), \Cref{thm: spline_continuity} states that there exist $m$ basis functions emerging from $\xi_l$ with various degrees of continuity ranging from $C^{p-m}$ to $C^{p-1}$. Denoting $k=p-m$ the local smoothness across $\xi_l$, the Taylor form of the B-splines (restricted to $\Omega_l$) is
\begin{equation*}
    \restr{B_i}{\Omega_l}(\xi) = \sum_{j=i}^p c_{ij}(\xi-\xi_l)^j \qquad i=k+1,\dots,p
\end{equation*}
with $c_{ii} \neq 0$ for all $i=k+1,\dots,p$. Therefore,
\begin{align}
    (B_i,B_j)_{L^2} &= \int_{\Omega_l} B_i(\xi)B_j(\xi) \dxi \sim \delta^{i+j+1} & &\text{as } \delta \to 0, \label{eq: L2_Bspline} \\
    (B_i',B_j')_{L^2} &= \int_{\Omega_l} B_i'(\xi)B_j'(\xi) \dxi \sim \delta^{i+j-1} & &\text{as } \delta \to 0. \label{eq: H1_Bspline}
\end{align}
This time, the operator norm of the basis functions scales differently, depending on $i$. Now, given arbitrary coefficients $u_j$ (independent of $\delta$), consider a function $u$ defined as
\begin{equation*}
    u(\xi) =
    \begin{cases}
        \sum_{k< j \leq p} u_j B_j(\xi). & \text{if } \xi \geq \xi_l, \\
        0 & \text{otherwise},
    \end{cases}
\end{equation*}
and note that $\restr{u}{\Omega_l}$ is completely generic. Similarly to the Lagrange basis, in the rescaled B-spline basis, this function becomes
\begin{equation*}
    \restr{u}{\Omega_l}(\xi) = \sum_{k< j \leq p} \hat{u}_j \hat{B}_j(\xi),
\end{equation*}
where $\hat{u}_j = \alpha_ju_j$ with $\alpha_j=\|B_j\|_a$. Denoting $\hat{\bm{u}}$ the coefficient vector of $\restr{u}{\Omega}$ in the rescaled B-spline basis, and denoting $r=\min \{i \colon u_i \neq 0\} \geq 1$, the numerator of the Rayleigh quotient in \eqref{eq: upper_bound_lambda_min} becomes
\begin{equation*}
    \|u\|_a^2 = \sum_{r \leq i,j \leq p} u_iu_j a(B_i,B_j)
\end{equation*}
and from the scaling relations in \cref{eq: L2_Bspline,eq: H1_Bspline}, we deduce that
\begin{align*}
    \lambda_1(\hat{M}) &\sim \frac{\sum_{r \leq i,j \leq p} \delta^{i+j+1}}{\sum_{r \leq i \leq p} \delta^{2i+1}} \to 1 & &\text{as } \delta \to 0, \\
    \lambda_1(\hat{K}) &\sim \frac{\sum_{r \leq i,j \leq p} \delta^{i+j-1}}{\sum_{r \leq i \leq p} \delta^{2i-1}} \to 1 & &\text{as } \delta \to 0,
\end{align*}
independently of $r$ and $k$. Thus, the condition number of the Jacobi preconditioned mass and stiffness matrices is independent of $\delta$ for the B-spline basis, even in the $C^0$ case. This finding is a complete surprise if one believes the intuitive arguments laid out in \cite{de2023stability} that insist on the smoothness. Instead, the choice of basis appears just as important. If one were to consider a generic linear combination of Lagrange basis functions, as we did for the B-splines, the terms appearing in $\|u\|_a^2$ would all scale \emph{identically} with respect to $\delta$, suggesting the possibility of canceling out low-order terms by choosing appropriate coefficients. This cannot happen for the B-spline basis because there is a single term of minimal order. The subtle reason lies in the increasing range of continuity of the basis functions across knots (see \Cref{fig: 1D_bases_p2_n4}). In the $C^0$ case, the $p$ problematic Lagrange basis functions are all $C^0$ across $\xi_l$ while the continuity of the Bernstein basis functions ranges from $C^0$ to $C^{p-1}$. Thus, it is the basis more than the actual space that is critical.

\begin{remark}
Although eventually the smallest eigenvalue of the Jacobi preconditioned stiffness and mass matrix scales identically, the numerator and denominator of the Rayleigh quotient scale differently. The outcome is identical if $A=a_1M+a_2K$ is a linear combination of the mass and stiffness matrices, where $a_1,a_2>0$ are constants independent of $\delta$. In this case, the operator norm becomes $\|v_h\|_a^2 = a_1 \|v_h\|_{L^2}^2+a_2 \|v_h'\|_{L^2}^2$. Therefore, for the Lagrange basis, we obtain for the same function $u$
\begin{equation*}
    \lambda_1(\hat{A}) \lesssim \frac{a_1 \delta^{2p+1}+a_2 \delta^{2p-1}}{a_1 \delta^3 + a_2 \delta} \sim \delta^{2(p-1)} \qquad \text{as } \delta \to 0.
\end{equation*}
Instead, for B-spline basis, we obtain
\begin{equation*}
    \lambda_1(\hat{A}) \sim \frac{\sum_{r \leq i,j \leq p} a_1\delta^{i+j+1}+a_2\delta^{i+j-1}}{\sum_{r \leq i \leq p} a_1\delta^{2i+1}+a_2\delta^{2i-1}} \to 1 \qquad \text{as } \delta \to 0,
\end{equation*}
because there is only a single term of minimal order in the numerator and denominator. Therefore, the results are asymptotically identical for linear combinations of the mass and stiffness matrices.
\end{remark}

\subsubsection{2D analysis}
\label{se: 2D_analysis}
In the 2D case, the behavior of the eigenvalues of the Jacobi preconditioned matrix critically depends on the cut configuration, also for smooth spline bases. Contrary to the 1D case, the volume fraction does not give a complete description, and instead the behavior of the basis functions relative to the cut geometry plays a prominent role. This also holds for generalized eigenvalue problems involving a (lumped) mass matrix, as shown in \cite{bioli2025theoretical}. Despite infinitely many trimmed configurations, there are essentially three representative ways a trimming curve may cut through the support of a basis function: 
\begin{enumerate}[noitemsep,label=Configuration (\alph*):,ref=Configuration (\alph*),leftmargin=*]
    \item in one direction only (``sliver-cut'', \Cref{fig: 2D_trimming_support_1direction}), \label{conf: sliver_cut}
    \item in two directions and around a corner (``corner-cut'', \Cref{fig: 2D_trimming_support_2direction}), \label{conf: corner_cut}
    \item in two directions but not around a corner (``middle-cut'', \Cref{fig: 2D_trimming_support_angle}). \label{conf: middle_cut}
\end{enumerate}
The first two trimming configurations are rather natural generalizations of the 1D case (referred to as ``sliver-cut'' and ``corner-cut'' in \cite{stoter2023critical}). However, the last one is specific to 2D geometries and is pathological in many regards. In this case, we must assume that $\bm{z}$ is ``sufficiently far'' from $\bm{\xi}_{l}$ (the anchor of the basis function); otherwise it reproduces \ref{conf: corner_cut}. \Cref{fig: 2D_trimming_example} shows an example where all three configurations are encountered. We will carefully analyze those three configurations for the Lagrange and B-spline bases. In both cases, multivariate basis functions are a tensor product of univariate ones such that
\begin{equation*}
    L_{\bm{i}}(\bm{\xi})=L_{i_1i_2}(\xi_1,\xi_2) = L_{1,i_1}(\xi_1)L_{2,i_2}(\xi_2) \quad \text{and} \quad B_{\bm{i}}(\bm{\xi})=B_{i_1i_2}(\xi_1,\xi_2) = B_{1,i_1}(\xi_1)B_{2,i_2}(\xi_2),
\end{equation*}
for a multi-index $\bm{i}=(i_1,i_2)$. In the next sections, we analyze the Lagrange and B-spline bases separately, using a local numbering of the basis functions to lighten the notation. Throughout the analysis, we use standard operations on multi-indices, where, given two multi-indices $\bm{i},\bm{j} \in \mathbb{N}^2$, $\bm{i}+\bm{j}$ and $\bm{i} \leq \bm{j}$ denotes componentwise addition and inequality, respectively, and $|\bm{i}|=i_1+i_2$. Furthermore, $\bm{p}=(p_1,p_2)$ are the degrees and $\bm{k}=(k_1,k_2)$ are the (spline) continuities along each direction.

\begin{figure}[H]
    \centering
    \begin{subfigure}[t]{0.32\textwidth}
        \centering
        \includegraphics[width=\textwidth]{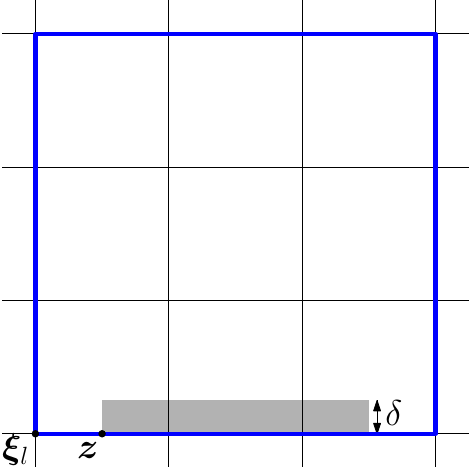}
        \caption{Trimming curve cutting along one direction only (``sliver-cut'').}
        \label{fig: 2D_trimming_support_1direction}
    \end{subfigure}
    \hfill
    \begin{subfigure}[t]{0.32\textwidth}
        \centering
        \includegraphics[width=\textwidth]{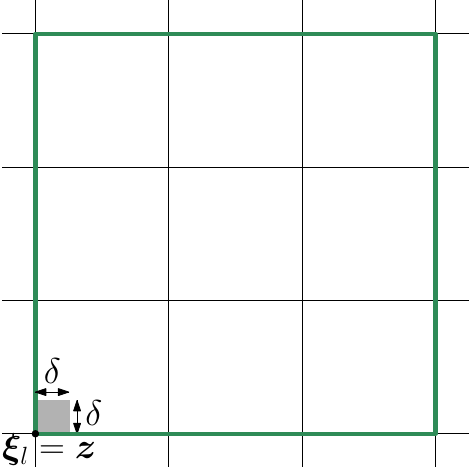}
        \caption{Trimming curve cutting in two directions and around a corner of the support (``corner-cut'').}
        \label{fig: 2D_trimming_support_2direction}
    \end{subfigure}
    \hfill
    \begin{subfigure}[t]{0.32\textwidth}
        \centering
        \includegraphics[width=\textwidth]{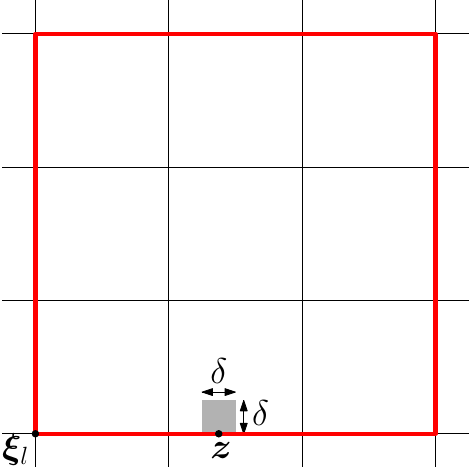}
        \caption{Trimming curve cutting in two directions but not around a corner of the support (``middle-cut'').}
        \label{fig: 2D_trimming_support_angle}
    \end{subfigure}
    \caption{Idealized trimming configurations in 2D. The point $\bm{\xi}_{l}=(\xi_{l_1},\xi_{l_2})$ denotes the anchor of a basis function while $\bm{z}=(z_1,z_2)$ parameterizes the position of the trimmed feature. Figure adapted from \cite{bioli2025theoretical}.}
    \label{fig: 2D_trimming_allconfigurations}
\end{figure}

\begin{figure}[H]
    \centering
    \includegraphics[width=0.5\textwidth]{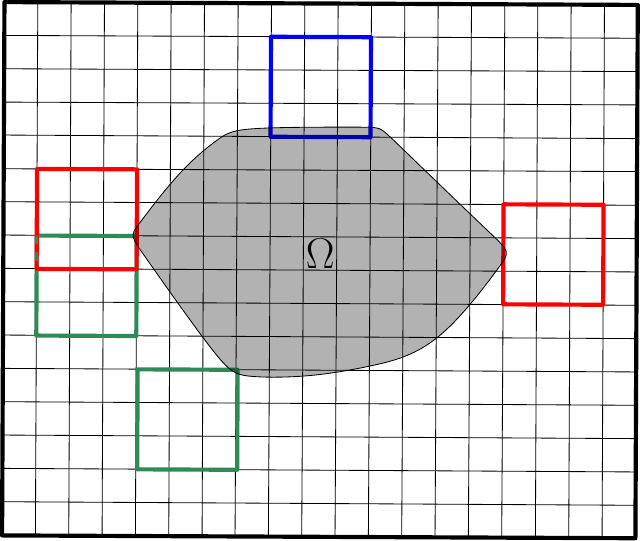}
    \caption{Support of several badly trimmed basis functions. Figure adapted from \cite{bioli2025theoretical}.}
    \label{fig: 2D_trimming_example}
\end{figure}

\noindent\textbf{Lagrange basis}: \\
\noindent \textbf{\ref{conf: sliver_cut}}:
Similarly to the 1D case, we will explicitly construct a function $u$ that causes the Rayleigh quotient in \eqref{eq: upper_bound_lambda_min} to vanish. A natural extension of the 1D case consists in choosing
\begin{equation*}
    u(\bm{\xi}) =
    \begin{cases}
        (\xi_2-\xi_{l_2})^{p_2} & \text{if } \xi_2 \geq \xi_{l_2}, \\
        0 & \text{otherwise}.
    \end{cases}
\end{equation*}
Repeating the 1D computations, we immediately obtain
\begin{align*}
    \|u\|_{L^2}^2 &= \int_{\Omega_l} (\xi_2-\xi_{l_2})^{2p_2} \ \dbxi \sim \delta^{2p_2+1} & &\text{as } \delta \to 0, \\
    \|\nabla u\|_{L^2}^2 &= \int_{\Omega_l} p_2^2(\xi_2-\xi_{l_2})^{2(p_2-1)} \ \dbxi \sim \delta^{2p_2-1} & &\text{as } \delta \to 0,
\end{align*}
while for the basis functions, we get
\begin{align*}
    \|L_{\bm{i}}\|_{L^2}^2 &= \int_{\Omega_l} L_{\bm{i}}^2 \ \dbxi \sim \delta^{3} & &\text{as } \delta \to 0, \\
    \|\nabla L_{\bm{i}}\|_{L^2}^2 &= \int_{\Omega_l} \|\nabla L_{\bm{i}}\|_2^2 \ \dbxi \sim \delta & &\text{as } \delta \to 0.
\end{align*}
Therefore, the results in this case perfectly mirror the 1D case:
\begin{equation*}
    \lambda_1(\hat{M}) \lesssim \frac{\delta^{2p_2+1}}{\delta^3} = \delta^{2(p_2-1)} \quad \text{and} \quad \lambda_1(\hat{K}) \lesssim \frac{\delta^{2p_2-1}}{\delta} = \delta^{2(p_2-1)},
\end{equation*}
and we deduce that
\begin{equation*}
    \kappa(\hat{A}) \gtrsim \delta^{-2(p_2-1)} \sim \eta^{-2(p_2-1)}.
\end{equation*}

\noindent \textbf{\ref{conf: corner_cut}}: This configuration is treated analogously by choosing instead
\begin{equation*}
    u(\bm{\xi}) =
    \begin{cases}
        (\xi_1-\xi_{l_1})^{p_1}(\xi_2-\xi_{l_2})^{p_2} & \text{if } \bm{\xi} \geq \bm{\xi}_{l}, \\
        0 & \text{otherwise}.
    \end{cases}
\end{equation*}
Since $(\xi_i-\xi_{l_i})^{p_i}$ is a linear combination of $\{L_{i,j}\}_{j=1}^{p_i}$ for $i=1,2$, it follows that $\restr{u}{\Omega_l}$ is merely a linear combination of $\{L_{\bm{i}}\}_{\bm{1} \leq \bm{i} \leq \bm{p}}$ and we conclude that $\restr{u}{\Omega} \in V_h$. Recalling that $\Omega_l = [\xi_{l_1},\xi_{l_1}+\delta) \times [\xi_{l_2},\xi_{l_2}+\delta)$, the evaluation of the numerator in \eqref{eq: upper_bound_lambda_min} yields for the mass and stiffness matrices
\begin{align*}
    \|u\|_{L^2}^2 &= \int_{\Omega_l} (\xi_1-\xi_{l_1})^{2p_1}(\xi_2-\xi_{l_2})^{2p_2} \ \dbxi \sim \delta^{2(|\bm{p}|+1)} & &\text{as } \delta \to 0, \\
    \|\nabla u\|_{L^2}^2 &= \int_{\Omega_l} p_1^2(\xi_1-\xi_{l_1})^{2(p_1-1)}(\xi_2-\xi_{l_2})^{2p_2} + p_2^2(\xi_1-\xi_{l_1})^{2p_1}(\xi_2-\xi_{l_2})^{2(p_2-1)} \ \dbxi \sim \delta^{2|\bm{p}|} & &\text{as } \delta \to 0,
\end{align*}
while for the basis functions, we obtain
\begin{align*}
    \|L_{\bm{i}}\|_{L^2}^2 &= \int_{\Omega_l} L_{\bm{i}}^2 \ \dbxi \sim \delta^6 & &\text{as } \delta \to 0, \\
    \|\nabla L_{\bm{i}}\|_{L^2}^2 &= \int_{\Omega_l} \|\nabla L_{\bm{i}}\|_2^2 \ \dbxi \sim \delta^4 & &\text{as } \delta \to 0.
\end{align*}
Thus, in both cases, the condition number of the Jacobi preconditioned matrix scales as
\begin{equation*}
    \kappa(\hat{A}) \gtrsim \delta^{-2(|\bm{p}|-2)} \sim \eta^{-(|\bm{p}|-2)}.
\end{equation*}
Although the scaling is different with respect to $\delta$, it is identical with respect to $\eta \sim \delta^2$ for uniform discretization parameters. This will not be the case in our last trimming configuration. \\

\noindent \textbf{\ref{conf: middle_cut}}: In principle, the same function $u$ could also be chosen in this last configuration to show that the Rayleigh quotient vanishes. However, in this case, there exist better choices for which the Rayleigh quotient decays much faster. With the notation in \Cref{fig: 2D_trimming_support_angle}, we construct a function $u$ such that
\begin{equation*}
    \restr{u}{\Omega_l}(\bm{\xi}) = (\xi_1-z_1)^{p_1}(\xi_2-z_2)^{p_2}.
\end{equation*}
Note that $z_2=\xi_{l_2}$ and thus the integration with respect to $\xi_2$ remains unchanged. However, $u_1(\xi_1):=(\xi_1-z_1)^{p_1}$ differs from the last case. In particular, since $u_1(\xi_{l_1}) \neq 0$, its expansion in the Lagrange basis must necessarily involve $L_{1,0}$ (contrary to all previous cases). Thus, there exist coefficients $u_{1,i}$ (with $u_{1,0} \neq 0$) and $u_{2,j}$ such that
\begin{equation*}
    \restr{u}{\Omega_l}(\bm{\xi})  = \left(\sum_{i_1=0}^{p_1} u_{1,i_1}L_{1,i_1}(\xi_1)\right)\left(\sum_{i_2=1}^{p_2}u_{2,i_2} L_{2,i_2}(\xi_2)\right).
\end{equation*}
The analysis of this case is very subtle since the placement of the centerpoint $\bm{z}$ within an element also influences the scaling. For understanding why, one must have in mind the expression of the Lagrange basis functions in \eqref{eq: lagrange_polynomials}. Denoting $\{\hat{\xi}_{1,i}\}_{i=0}^{p_1}$ the interpolation points along the $x$ direction within the trimmed element (with $\hat{\xi}_{1,0}=\xi_{l_1}$ and $\hat{\xi}_{1,p_1}=\xi_{l_1+1}$), the univariate Lagrange basis functions along the horizontal direction are
\begin{equation*}
    L_{1,i}(\xi_1) = \prod_{\substack{j=0 \\ j \neq i}}^{p_1} \frac{\xi_1-\hat{\xi}_{1,j}}{\hat{\xi}_{1,i}-\hat{\xi}_{1,j}} \qquad i=0,\dots,p_1.
\end{equation*}
Hence, the coefficients $u_{1,i}$ are given by $u_{1,i}=u_1(\hat{\xi}_{1,i})$. We must now distinguish cases, depending on whether $z_1$ falls in the neighborhood of an interpolation point. Let $\mathcal{B}_{\xi}(\delta)$ denote the ball of radius $\delta$ centered at $\xi$.

\begin{itemize}[noitemsep]
    \item If $z_1 \notin \mathcal{B}_{\hat{\xi}_{1,i}}(\delta)$ for any $i$, then all coefficients are independent of $\delta$ and for all $j=0,\dots,p_1$,
    \begin{equation*}
        \|L_{1,j}\|_{L^2}^2 \sim \delta \quad \text{and} \quad \|L_{1,j}'\|_{L^2}^2 \sim \delta \quad \text{as } \delta \to 0.
    \end{equation*}
    Therefore,
    \begin{equation*}
        \|L_{\bm{i}}\|_{L^2}^2 \sim \delta^4 \quad \text{and} \quad \|\nabla L_{\bm{i}}\|_{L^2}^2 \sim \delta^2 \quad \text{as } \delta \to 0.
    \end{equation*}
    
    \item If $z_1 \in \mathcal{B}_{\hat{\xi}_{1,i}}(\delta)$ for some $i$, then $u_{1,i} \lesssim \delta^{p_1}$ while the other coefficients are independent of $\delta$. Since all basis functions except $L_{1,i}(\xi_1)$ include the product $(\xi_1-\hat{\xi}_{1,i})$, for all $j \neq i$,
    \begin{equation*}
        \|L_{1,j}\|_{L^2}^2 \sim \delta^3 \quad \text{and} \quad \|L_{1,j}'\|_{L^2}^2 \sim \delta \quad \text{as } \delta \to 0.
    \end{equation*}
    However, for $j=i$, $\|L_{1,i}\|_{L^2}^2 \sim \delta$ and $\|L_{1,i}'\|_{L^2}^2 \sim \delta$. While for the other direction,
    \begin{equation*}
        \|L_{2,j}\|_{L^2}^2 \sim \delta^3 \quad \text{and} \quad \|L_{2,j}'\|_{L^2}^2 \sim \delta \quad \text{as } \delta \to 0.
    \end{equation*}
    Consequently, for the mass,
    \begin{align*}
        \sum_{(0,1) \leq \bm{i} \leq \bm{p}} u_{\bm{i}}^2 \|L_{\bm{i}}\|_{L^2}^2 &= \left(\sum_{i_1=0}^{p_1} u_{1,i_1}^2 \|L_{1,i_1}\|_{L^2}^2\right)\left(\sum_{i_2=1}^{p_2} u_{2,i_2}^2 \|L_{2,i_2}\|_{L^2}^2\right) \\
        &\sim (\delta^{2p_1+1} + \delta^3)\delta^3 \sim \delta^6 \quad \text{as } \delta \to 0.
    \end{align*}
    While for the stiffness,
    \begin{align*}
        \sum_{(0,1) \leq \bm{i} \leq \bm{p}} u_{\bm{i}}^2 \|\nabla L_{\bm{i}}\|_{L^2}^2 &= \sum_{i_1=0}^{p_1} \sum_{i_2=1}^{p_2} u_{1,i_1}^2 u_{2,i_2}^2 (\|L_{1,i_1}'\|_{L^2}^2\|L_{2,i_2}\|_{L^2}^2+\|L_{1,i_1}\|_{L^2}^2\|L_{2,i_2}'\|_{L^2}^2) \\
        &= \left(\sum_{i_1=0}^{p_1} u_{1,i_1}^2 \|L_{1,i_1}'\|_{L^2}^2\right)\left(\sum_{i_2=1}^{p_2} u_{2,i_2}^2 \|L_{2,i_2}\|_{L^2}^2\right)+\left(\sum_{i_1=0}^{p_1} u_{1,i_1}^2 \|L_{1,i_1}\|_{L^2}^2\right)\left(\sum_{i_2=1}^{p_2} u_{2,i_2}^2 \|L_{2,i_2}'\|_{L^2}^2\right) \\
        &\sim (\delta^{2p_1+1}+\delta)\delta^3 + (\delta^{2p_1+1}+\delta^3)\delta \\
        &\sim \delta^4 \quad \text{as } \delta \to 0.
    \end{align*}
\end{itemize}
Thus, the coefficient $u_{1,i}$ only contributes high-order terms (and vanishes entirely if $\hat{\xi}_{1,i}=z_1$). However, the function $u$ itself is independent of the interpolation points, and similarly to the previous configuration, we obtain in either case
\begin{equation*}
    \|u\|_{L^2}^2 \sim \delta^{2(|\bm{p}|+1)} \quad \text{and} \quad \|\nabla u\|_{L^2}^2 \sim \delta^{2|\bm{p}|} \quad \text{as } \delta \to 0.
\end{equation*}
Finally, we obtain the following scaling for the condition number of the preconditioned mass and stiffness matrices:
\begin{itemize}
    \item if $z_1 \notin \mathcal{B}_{\hat{\xi}_{1,i}}(\delta)$, then 
    \begin{equation*}
        \kappa(\hat{A}) \gtrsim \delta^{-2(|\bm{p}-1)} \sim \eta^{-(|\bm{p}|-1)},
    \end{equation*}
    \item if $z_1 \in \mathcal{B}_{\hat{\xi}_{1,i}}(\delta)$ for some $i$, then
    \begin{equation*}
        \kappa(\hat{A}) \gtrsim \delta^{-2(|\bm{p}|-2)} \sim \eta^{-(|\bm{p}|-2)}.
    \end{equation*}
\end{itemize}
Thus, the second case is slightly more favorable. We now repeat the analysis for the B-spline basis. \\

\noindent\textbf{B-spline basis}: \\
\noindent \textbf{\ref{conf: sliver_cut}}:
The proof arguments are completely analogous to the 1D configuration. Considering a generic function $u$ only supported on the trimmed elements, we notice that the Rayleigh quotient is always independent of $\delta$. Thus, the ``sliver-cut'' configuration does not cause any ill-conditioning for the B-spline basis. \\

\noindent \textbf{\ref{conf: corner_cut}}:
This configuration is a tensor product generalization of the 1D case. We detail it to make sure that the same arguments apply. Considering a generic function $u$ supported on the trimmed corner in \Cref{fig: 2D_trimming_support_2direction}
\begin{equation*}
    u(\bm{\xi}) =
    \begin{cases}
        \sum_{\bm{k} < \bm{i} \leq \bm{p}} u_{\bm{i}} B_{\bm{i}}(\bm{\xi}) & \text{if } \bm{\xi} \geq \bm{\xi}_{l}, \\
        0 & \text{otherwise}.
    \end{cases}
\end{equation*}
In the rescaled B-spline basis, the restriction of this function becomes
\begin{equation*}
    \restr{u}{\Omega_l}(\bm{\xi}) = \sum_{\bm{k} < \bm{i} \leq \bm{p}} \hat{u}_{\bm{i}} \hat{B}_{\bm{i}}(\bm{\xi}),
\end{equation*}
where $\hat{u}_{\bm{i}} = \alpha_{\bm{i}}u_{\bm{i}}$ with $\alpha_{\bm{i}}=\|B_{\bm{i}}\|_a$ are the coefficients of $\restr{u}{\Omega}$ in the rescaled B-spline basis. Denoting $r_1= \min \{i_1 \colon u_{i_1i_2} \neq 0\} \geq 1$, $r_2= \min \{i_2 \colon u_{i_1i_2} \neq 0\} \geq 1$ and $\bm{r}=(r_1,r_2)$, the numerator in \eqref{eq: upper_bound_lambda_min} becomes
\begin{equation*}
    \|u\|_a^2 = \sum_{\bm{r} \leq \bm{i},\bm{j} \leq \bm{p}}  u_{\bm{i}}u_{\bm{j}} a(B_{\bm{i}},B_{\bm{j}}).
\end{equation*}
Moreover, from the scaling relations in \cref{eq: L2_Bspline,eq: H1_Bspline},
\begin{align*}
    (B_{\bm{i}},B_{\bm{j}})_{L^2} &= (B_{1,i_1},B_{1,j_1})_{L^2}(B_{2,i_2},B_{2,j_2})_{L^2} \sim \delta^{|\bm{i}|+|\bm{j}|+2} & &\text{as } \delta \to 0, \\
    (\nabla B_{\bm{i}}, \nabla B_{\bm{j}})_{L^2} &= (B'_{1,i_1},B'_{1,j_1})_{L^2}(B_{2,i_2},B_{2,j_2})_{L^2}+(B_{1,i_1},B_{1,j_1})_{L^2}(B'_{2,i_2},B'_{2,j_2})_{L^2} \sim \delta^{|\bm{i}|+|\bm{j}|} & &\text{as } \delta \to 0.
\end{align*}
Therefore, from the Rayleigh quotient in \eqref{eq: upper_bound_lambda_min}, we deduce that
\begin{align*}
    \lambda_1(\hat{M}) &\sim \frac{\sum_{\bm{r} \leq \bm{i},\bm{j} \leq \bm{p}} \delta^{|\bm{i}|+|\bm{j}|+2}}{\sum_{\bm{r} \leq \bm{i} \leq \bm{p}} \delta^{2|\bm{i}|+2}} \to 1 & &\text{as } \delta \to 0, \\
    \lambda_1(\hat{K}) &\sim \frac{\sum_{\bm{r} \leq \bm{i},\bm{j} \leq \bm{p}} \delta^{|\bm{i}|+|\bm{j}|}}{\sum_{\bm{r} \leq \bm{i} \leq \bm{p}} \delta^{2|\bm{i}|}} \to 1 & &\text{as } \delta \to 0,
\end{align*}
independently of $\bm{r}$ and $\bm{k}$. Indeed, although some of the terms in the sum might scale identically with respect to $\delta$, there always exists a single term of minimal order, which ensures that the Rayleigh quotient becomes independent of $\delta$ in the limit as $\delta \to 0$, regardless of the coefficients $u_{\bm{i}}$. Thus, the results so far perfectly align with the 1D case. Unfortunately, the last configuration is again an exception. \\

\noindent \textbf{\ref{conf: middle_cut}}:
Repeating the previous derivations for this configuration reveals that there exist multiple terms of minimal order, and our arguments fall through. This is a strong hint for instead seeking a counter-example. Two cases must be distinguished depending on the position of $z_1$. 

\begin{itemize}
    \item If $z_1 \notin \mathcal{B}_{\xi_{l_i}}(\delta)$, then we consider the same counter-example as for the Lagrange basis by constructing a function $u$ such that
\begin{equation*}
    \restr{u}{\Omega_l}(\bm{\xi}) = (\xi_1-z_1)^{p_1}(\xi_2-z_2)^{p_2}.
\end{equation*}
Since $z_2=\xi_{l_2}$, $u_2(\xi_2)=(\xi_2-z_2)^{p_2} \sim B_{2,p_2}$ behaves as the B-spline of maximal smoothness, which always exists independently of the local continuity. However, since $z_1 \neq \xi_{l_1}$, this is not the case for $u_1(\xi_1)=(\xi_1-z_1)^{p_1}$, whose expansion in the B-spline basis will include multiple nonzero coefficients (independent of $\delta$). Moreover, similarly to the Lagrange basis,
\begin{equation*}
    \|B_{1,j}\|_{L^2}^2 \sim \delta \quad \text{and} \quad \|B_{1,j}'\|_{L^2}^2 \sim \delta \quad \text{as } \delta \to 0.
\end{equation*}
The scaling for the numerator of the Rayleigh quotient remains unchanged, and we obtain
\begin{equation*}
    \|u\|_{L^2}^2 \sim \delta^{2(|\bm{p}|+1)} \quad \text{and} \quad \|\nabla u\|_{L^2}^2 \sim \delta^{2|\bm{p}|} \quad \text{as } \delta \to 0.
\end{equation*}
However, since $u_2(\xi_2)$ is perfectly described by the smoothest basis functions along the $y$ direction
\begin{equation*}
    \|B_{\bm{i}}\|_{L^2}^2 \sim \delta^{2(p_2+1)} \quad \text{and} \quad \|\nabla B_{\bm{i}}\|_{L^2}^2 \sim \delta^{2p_2} \quad \text{as } \delta \to 0.
\end{equation*}
Finally, we deduce that
\begin{equation*}
    \kappa(\hat{A}) \gtrsim \delta^{-2p_1} \sim \eta^{-p_1}.
\end{equation*}
\item If $z_1 \in \mathcal{B}_{\xi_{l_i}}(\delta)$, then $u$ does not produce a counter-example since it is (almost) described by a basis function from \ref{conf: corner_cut}. However, there might exist multiple other basis functions that are $O(1)$ over the cut element. The only exception is the $C^0$ case, where the univariate basis functions all scale differently with respect to $\delta$. Hence, similarly to the ``corner-cut'' configuration, linear dependencies cannot arise for $C^0$ continuity. However, smooth spline spaces become problematic. Since there are multiple functions that scale as $O(1)$, this leaves the door open for potential cancellation. Let us provide an explicit counter-example. For $k_1 \geq 1$, we build a function $u$ such that
\begin{equation*}
    \restr{u}{\Omega_l}(\bm{\xi}) = (\xi_1-z_1)^{k_1}(\xi_2-z_2)^{p_2}.
\end{equation*}
Recall that the continuity $k_1$ is related to the knot multiplicity $m_1$ through the relation $k_1=p_1-m_1$ and $1 \leq m_1 \leq p_1$. For smooth spline spaces, there are $m_1$ splines starting at $\xi_{l_1}$ with smoothness ranging from $C^{k_1}$ to $C^{p_1-1}$. Denoting this set of functions $\{B_{1,k_1+1},\dots,B_{1,p_1}\}$, they must necessarily take the form
\begin{equation*}
    B_{1,i}(\xi_1) = \sum_{j=i}^{p_1} c_{ij}(\xi_1-\xi_{l_1})^j \qquad i=k_1+1,\dots,p_1,
\end{equation*}
with $c_{ii} \neq 0$ for all $i=k_1+1,\dots,p_1$. Since the term of lowest order is $(\xi_1-\xi_{l_1})^{k_1+1}$, the expansion of $u_1(\xi_1)=(\xi_1-z_1)^{k_1}$ in the spline basis must necessarily involve other basis functions that are $O(1)$ over the trimmed element (even if $z_1=\xi_{l_1}$). Consequently, similarly to the previous case,
\begin{equation*}
    \|B_{\bm{i}}\|_{L^2}^2 \sim \delta^{2(p_2+1)} \quad \text{and} \quad \|\nabla B_{\bm{i}}\|_{L^2}^2 \sim \delta^{2p_2} \quad \text{as } \delta \to 0.
\end{equation*}
However, the scaling of the numerator becomes
\begin{equation*}
    \|u\|_{L^2}^2 \sim \delta^{2(k_1+p_2+1)} \quad \text{and} \quad \|\nabla u\|_{L^2}^2 \sim \delta^{2(k_1+p_2)} \quad \text{as } \delta \to 0.
\end{equation*}
Thus, we deduce that
\begin{equation*}
    \kappa(\hat{A}) \gtrsim \delta^{-2k_1} \sim \eta^{-k_1}.
\end{equation*}
This relation actually combines the $C^0$ and smooth cases in a single expression, and it is the first time we experience an explicit dependency on the continuity. Yet, this dependency is the opposite of the one that is commonly advertised: increased smoothness is in this case detrimental to the scaling of the condition number.
\end{itemize}

\begin{remark}
The above arguments also hold if $z_1 \notin \mathcal{B}_{\xi_{l_i}}(\delta)$ but in this case, a better counter-example exists.   
\end{remark}

The scaling relations in terms of the volume fraction $\eta = \min_{T \in \mathcal{T}_h} |T \cap \Omega|/|T|$ are summarized in \Cref{tab: scaling_jacobi} for the Lagrange and B-spline bases and uniform discretization parameters ($p_1=p_2=p$ and $k_1=k_2=k$).

\begin{table}[H]
\centering
\begin{tabular}{lll}
    \toprule
    Cut configuration & Lagrange basis & B-spline basis \\
    \midrule
    1D & $\gtrsim \eta^{-2(p-1)}$ & $\sim 1$ \\
    2D - \ref{conf: sliver_cut} (``sliver-cut'') & $ \gtrsim \eta^{-2(p-1)}$ & $\sim 1$ \\
    2D - \ref{conf: corner_cut} (``corner-cut'') & $ \gtrsim \eta^{-2(p-1)}$ & $\sim 1$ \\
    2D - \ref{conf: middle_cut} (``middle-cut'') & $
        \begin{cases}
            \gtrsim \eta^{-(2p-1)} &  \text{if } z_1 \notin \mathcal{B}_{\hat{\xi}_{1,i}}(\delta) \\
            \gtrsim \eta^{-2(p-1)} &  \text{if } z_1 \in \mathcal{B}_{\hat{\xi}_{1,i}}(\delta) 
        \end{cases}$ & $
        \begin{cases}
            \gtrsim \eta^{-p} &  \text{if } z_1 \notin \mathcal{B}_{\xi_{l_i}}(\delta) \\
            \gtrsim \eta^{-k} &  \text{if } z_1 \in \mathcal{B}_{\xi_{l_i}}(\delta) 
        \end{cases}$ \\
    \bottomrule
\end{tabular}
\caption{Scaling relations for the condition number of the Jacobi preconditioned matrix}
\label{tab: scaling_jacobi}
\end{table}

\begin{example}[Trimmed line]
\label{ex: trimmed_line}
This example confirms numerically the scaling relations for the trimmed 1D line of \Cref{fig: 1D_illustration} with Dirichlet boundary conditions on its left side and Neumann boundary conditions on its right (trimmed) side. The conditioning of the Jacobi preconditioned mass matrix is shown in \Cref{fig: 1D_Laplace_trimming_cond_M_Jacobi} for the Lagrange basis (\Cref{fig: 1D_Laplace_trimming_cond_M_Jacobi_Lagrange}), the Bernstein basis (\Cref{fig: 1D_Laplace_trimming_cond_M_Jacobi_Bspline_C0}) and the maximally smooth B-spline basis (\Cref{fig: 1D_Laplace_trimming_cond_M_Jacobi_Bspline_Cp-1}). The results perfectly align with the scaling relations reported in \Cref{tab: scaling_jacobi}. Those for the stiffness matrix were similar and are therefore omitted. In summary, for $\eta$ small enough, the condition number for spline discretizations is independent of $\eta$, even for the Bernstein basis (although the constant is significantly worse than in the maximally smooth case). Thus, the effectiveness of Jacobi preconditioning in 1D is not explained by the \emph{global} regularity of the space (as sometimes claimed), but rather by the \emph{local} regularity of the basis functions. Indeed, although the Lagrange and Bernstein bases span the same $C^0$ continuous space, the regularity of individual Bernstein basis functions across knots actually ranges from $C^0$ to $C^{p-1}$, contrary to the Lagrange basis functions, which are all $C^0$. Thus, subtle differences in the basis properties may critically affect the conditioning. Unfortunately, from dimension $2$ onward, many more factors enter the equation, and future examples will attempt to summarize the main messages of our analysis.

\begin{figure}[H]
     \centering
     \begin{subfigure}[t]{0.31\textwidth}
    \centering
    \includegraphics[width=\textwidth]{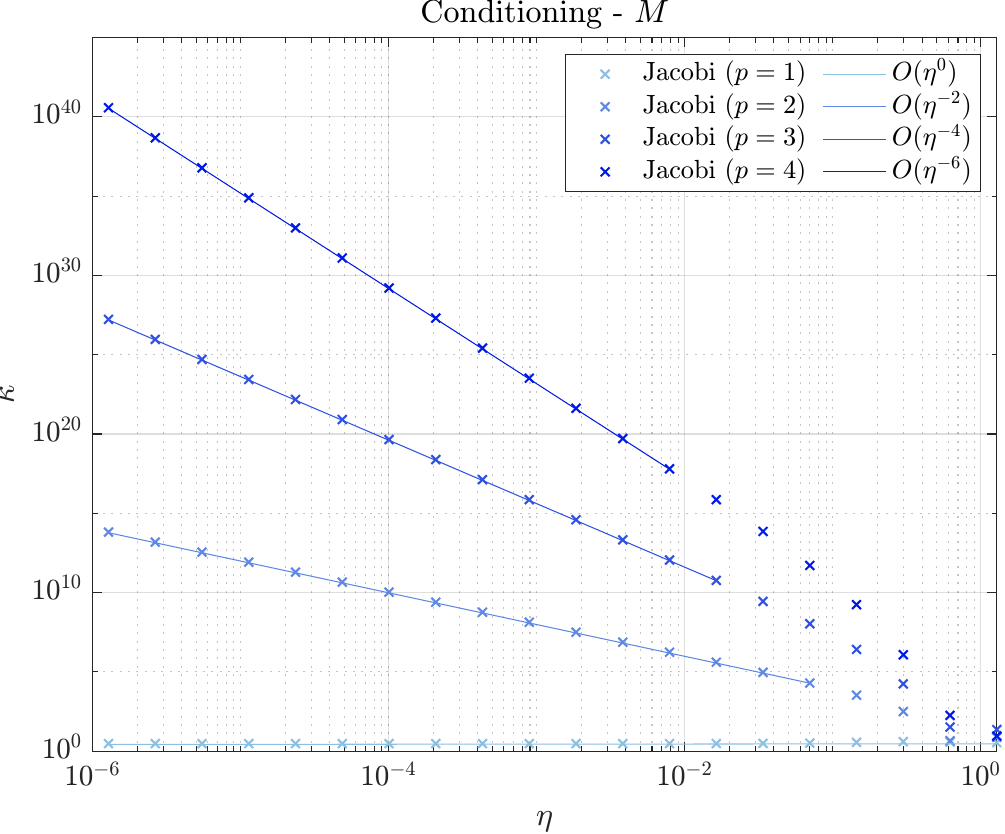}
    \caption{Lagrange basis}
    \label{fig: 1D_Laplace_trimming_cond_M_Jacobi_Lagrange}
     \end{subfigure}
     \hfill
     \begin{subfigure}[t]{0.31\textwidth}
    \centering
    \includegraphics[width=\textwidth]{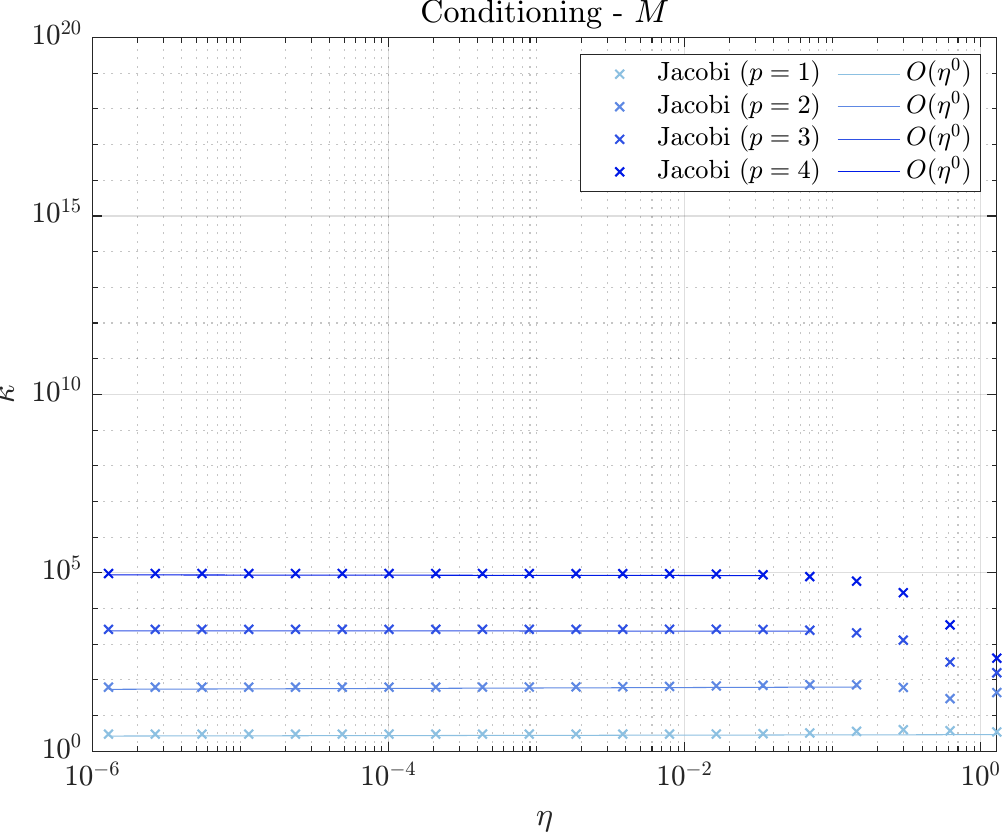}
    \caption{$C^0$ B-spline basis}
    \label{fig: 1D_Laplace_trimming_cond_M_Jacobi_Bspline_C0}
     \end{subfigure}
     \hfill
    \begin{subfigure}[t]{0.31\textwidth}
    \centering
    \includegraphics[width=\textwidth]{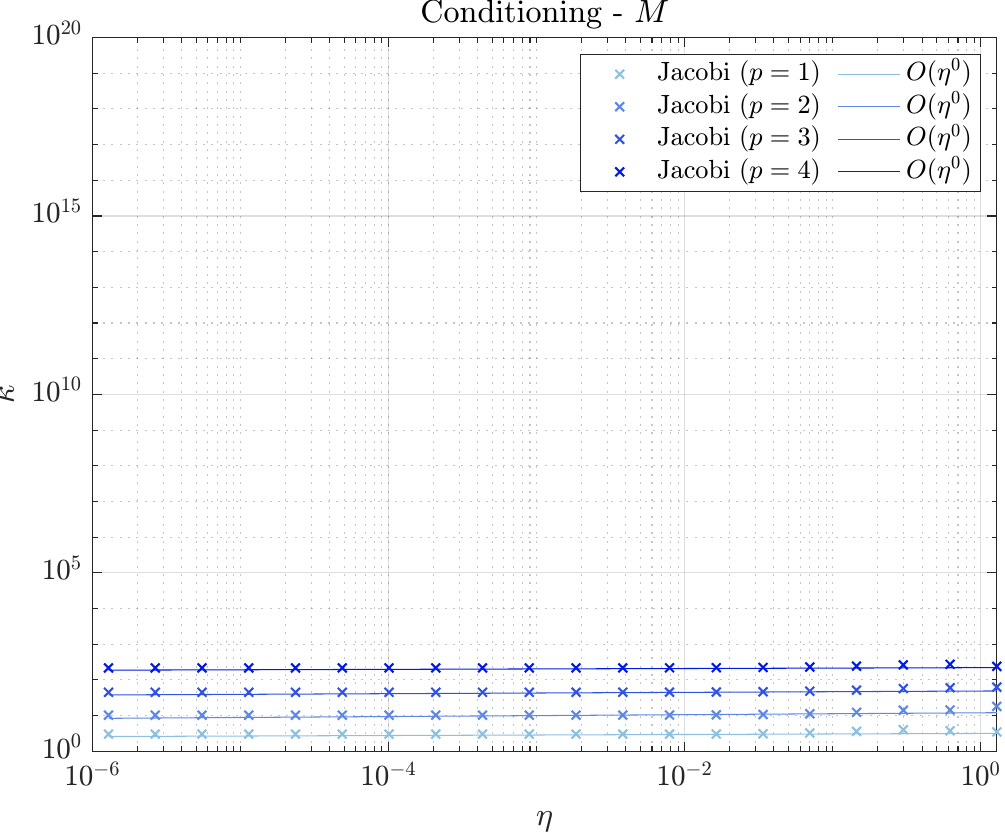}
    \caption{$C^{p-1}$ B-spline basis}
    \label{fig: 1D_Laplace_trimming_cond_M_Jacobi_Bspline_Cp-1}
     \end{subfigure}
     \hfill
    \caption{Conditioning of the Jacobi preconditioned mass matrix for the trimmed 1D line (\Cref{ex: trimmed_line})}
    \label{fig: 1D_Laplace_trimming_cond_M_Jacobi}
\end{figure}    
\end{example}

\begin{remark}
\label{rk: eigenvalue_computation}
While the eigenvalues and condition numbers are accurately computed for system matrices in the Bernstein basis, earlier work has reported signs of instability for the Lagrange basis and discretizations of degree $p \geq 3$ \cite{de2017condition,ballotta2023preconditioners}. Fortunately, there is a simple workaround. Since the stiffness and mass matrices are the Gram matrix of a bilinear form, their expressions in different bases are related through a congruence transformation,
\begin{equation}
\label{eq: change_of_basis}
    K_\dutchcal{L} = C^T K_\dutchcal{B} C \quad \text{and} \quad M_\dutchcal{L} = C^T M_\dutchcal{B} C,
\end{equation}
where $\dutchcal{L}$ and $\dutchcal{B}$ denote the Lagrange and Bernstein bases, respectively, and $C$ is the basis transformation matrix. Clearly, \eqref{eq: change_of_basis} is very convenient for a computer implementation since it avoids directly assembling the stiffness and mass in the Lagrange basis. However, the change of basis matrix $C=A^{-1}$ is the inverse of a collocation matrix $A$ containing B-spline evaluations. Thus, explicitly forming the system matrices from the relations in \eqref{eq: change_of_basis} might induce a loss of accuracy that also affects subsequent eigenvalue computations. Instead of directly computing the eigenvalues of $K_\dutchcal{L}$ and $M_\dutchcal{L}$, we compute those of the equivalent matrix pairs $(K_\dutchcal{B}, A^TA)$ and $(M_\dutchcal{B}, A^TA)$, which do not require any explicit inverse. This strategy allowed stably computing the eigenvalues with a (shift and invert) Lanczos method over a wider range of degrees.
\end{remark}

\begin{example}[Stretched square]
\label{ex: stretched_square}
This example combines \ref{conf: sliver_cut} and \ref{conf: corner_cut} in a single geometry, consisting of a square of side length $0.5+\delta$ whose bottom left corner is fixed at the origin. The conditioning of the Jacobi preconditioned mass matrix is shown in \Cref{fig: 2D_Laplace_square_cond_M_Jacobi} as $\delta \to 0$. Once again, the results perfectly align with our theoretical predictions: for the Lagrange basis, the condition number grows as $\eta^{-2(p-1)}$ while it remains constant for B-spline bases, even for the $C^0$ case (although the constant scales poorly with the degree).

\begin{figure}[H]
    \centering
    \includegraphics[width=0.5\textwidth]{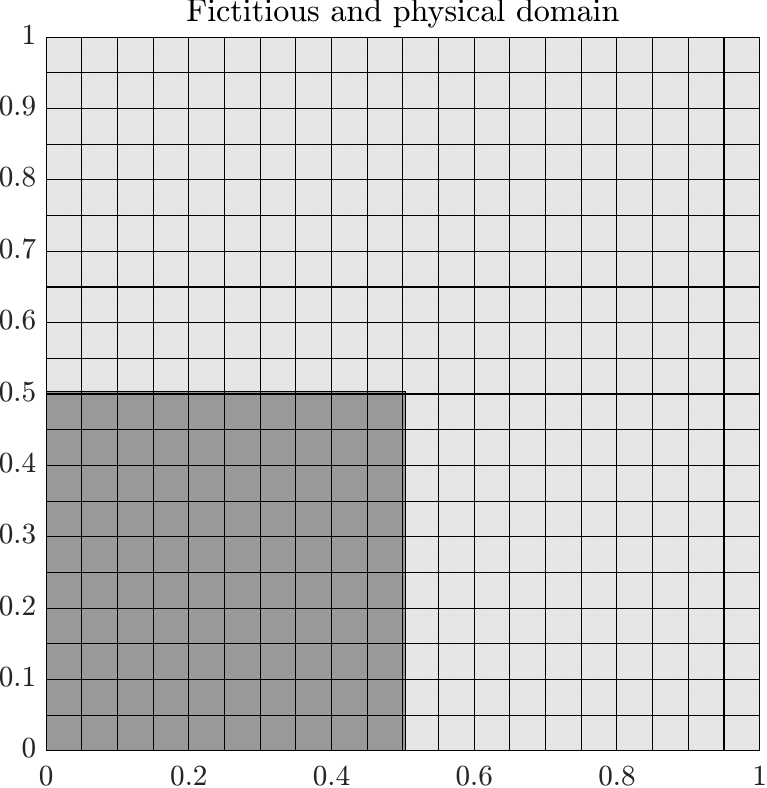}
    \caption{Stretched square (\Cref{ex: stretched_square})}
    \label{fig: 2D_Laplace_square_mesh}
\end{figure}

\begin{figure}[H]
     \centering
     \begin{subfigure}[t]{0.31\textwidth}
    \centering
    \includegraphics[width=\textwidth]{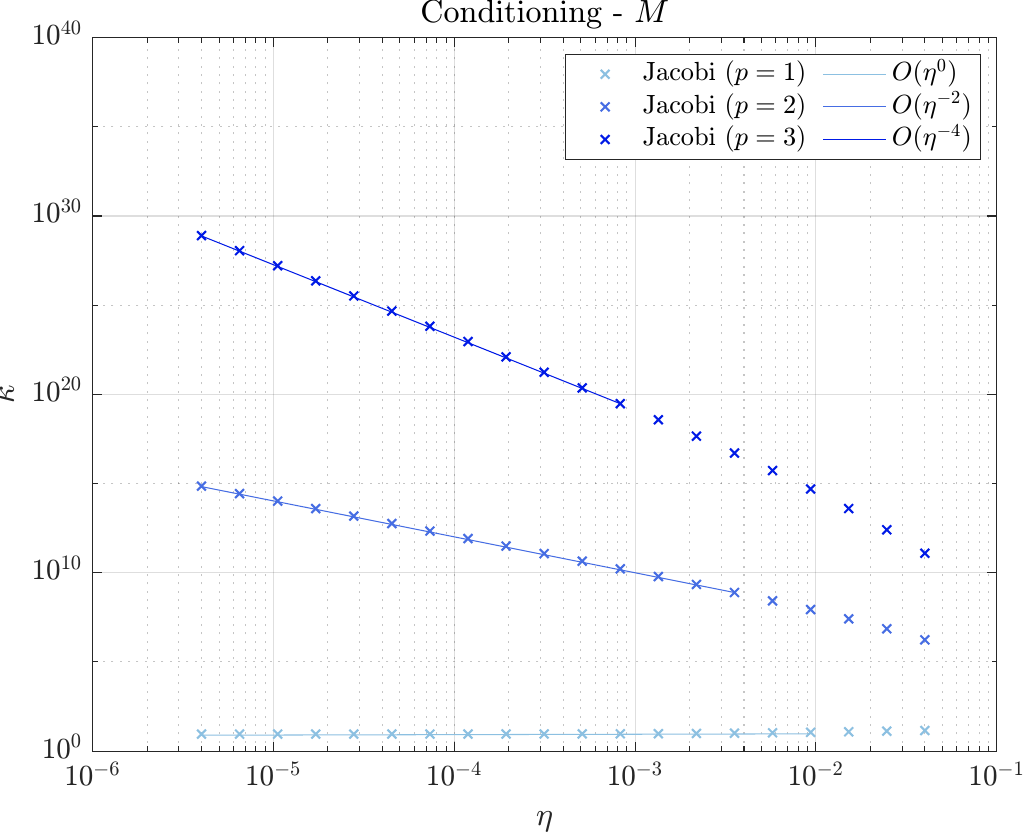}
    \caption{Lagrange basis}
    \label{fig: 2D_Laplace_square_cond_M_Jacobi_Lagrange}
     \end{subfigure}
     \hfill
     \begin{subfigure}[t]{0.31\textwidth}
    \centering
    \includegraphics[width=\textwidth]{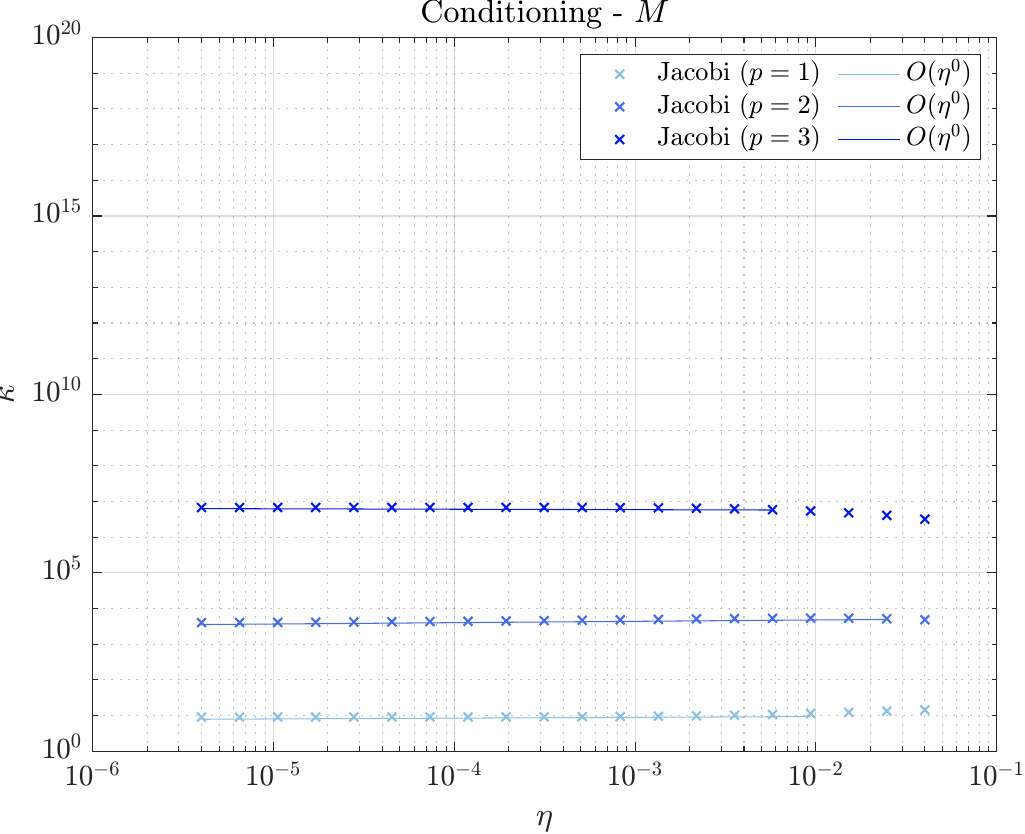}
    \caption{$C^0$ B-spline basis}
    \label{fig: 2D_Laplace_square_cond_M_Jacobi_Bspline_C0}
     \end{subfigure}
     \hfill
    \begin{subfigure}[t]{0.31\textwidth}
    \centering
    \includegraphics[width=\textwidth]{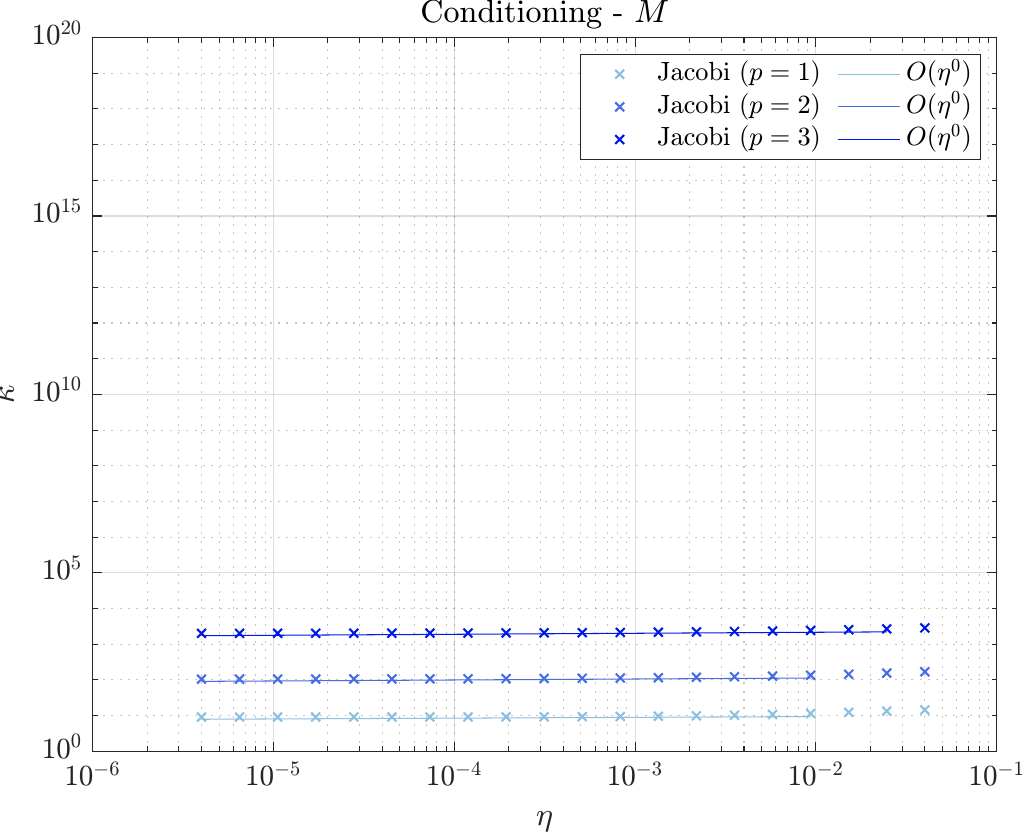}
    \caption{$C^{p-1}$ B-spline basis}
    \label{fig: 2D_Laplace_square_cond_M_Jacobi_Bspline_Cp-1}
     \end{subfigure}
     \hfill
    \caption{Conditioning of the Jacobi preconditioned mass matrix for the stretched square (\Cref{ex: stretched_square})}
    \label{fig: 2D_Laplace_square_cond_M_Jacobi}
\end{figure}    
\end{example}

\begin{example}[Ridge]
\label{ex: ridge}
Finally, to confirm the theoretical findings for \ref{conf: middle_cut}, we consider a house-like trimmed geometry that is shifted along the horizontal to produce different (not necessarily symmetric) trimming configurations. In the first case (\Cref{fig: 2D_Laplace_ridge_corner_mesh}), the house's ridge is perfectly aligned with a vertical grid line, thereby reproducing a corner-cut configuration. In the second case (\Cref{fig: 2D_Laplace_ridge_center_mesh}), the base of the house is shifted horizontally to produce a centered middle-cut configuration. Finally, in the third case (\Cref{fig: 2D_Laplace_ridge_decenter_mesh}), the base is further shifted to the right to slightly decenter the tip. In each case, pure Neumann boundary conditions are prescribed all along the boundary. The conditioning of the Jacobi preconditioned mass matrix is reported in \Cref{fig: 2D_Laplace_ridge_cond_M_Jacobi_Lagrange,fig: 2D_Laplace_ridge_cond_M_Jacobi_Bspline_C0,fig: 2D_Laplace_ridge_cond_M_Jacobi_Bspline_Cp-1} for the Lagrange, Bernstein, and maximally smooth B-spline bases, respectively. For the Lagrange basis, the scaling keeps changing from one configuration to the next, as expected theoretically. In particular, when the cut is perfectly centered (\Cref{fig: 2D_Laplace_ridge_center_cond_M_Jacobi_Lagrange}), $z_1$ coincides with an interpolation point for even degrees, which explains the milder growth of the condition number for $p=2$. However, when the cut is decentered, its tip never coincides with an interpolation point, and the condition number grows as $\eta^{-(2p-1)}$ independently of the degree parity. Instead, for spline bases, the tip coincides with a knot in \Cref{fig: 2D_Laplace_ridge_corner_mesh} and the condition number grows at the expected rate of $\eta^{-k}$. In this case, $C^0$ spline bases are asymptotically better than maximally smooth ones, although the constants are horrific. The differences are much less pronounced for the centered and decentered middle-cut configurations. For those cases, since the tip never coincides with a knot, the condition number always scales as $\eta^{-p}$. Similar trends were observed for the stiffness matrix (when ignoring zero eigenvalues associated with rigid-body motions). The scaling relations are further confirmed in \Cref{se: additional_examples} for a non-uniform grid.

\begin{figure}[H]
     \centering
     \begin{subfigure}[t]{0.31\textwidth}
    \centering
    \includegraphics[width=\textwidth]{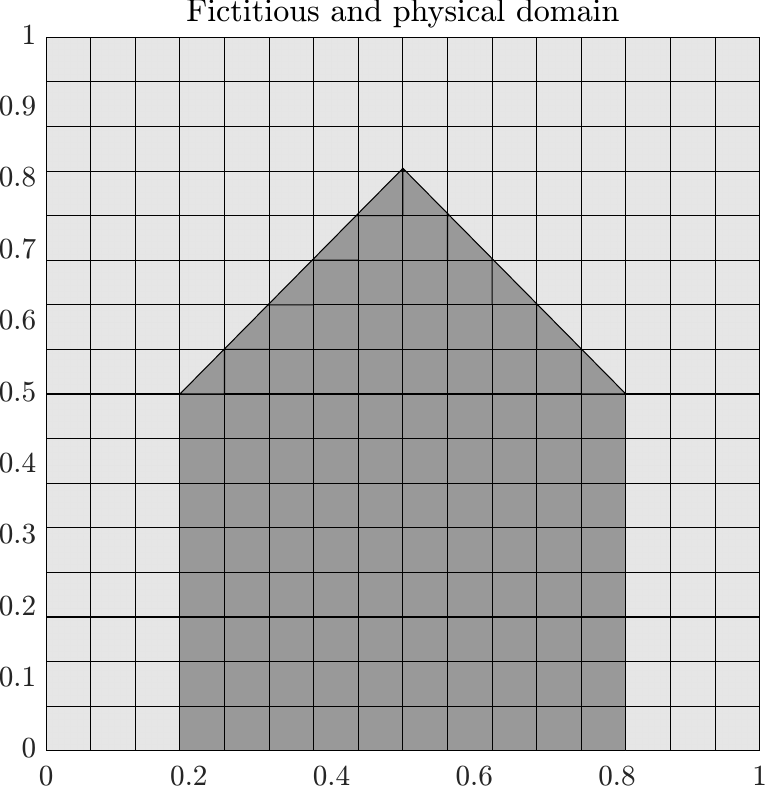}
    \caption{Corner-cut}
    \label{fig: 2D_Laplace_ridge_corner_mesh}
     \end{subfigure}
     \hfill
     \begin{subfigure}[t]{0.31\textwidth}
    \centering
    \includegraphics[width=\textwidth]{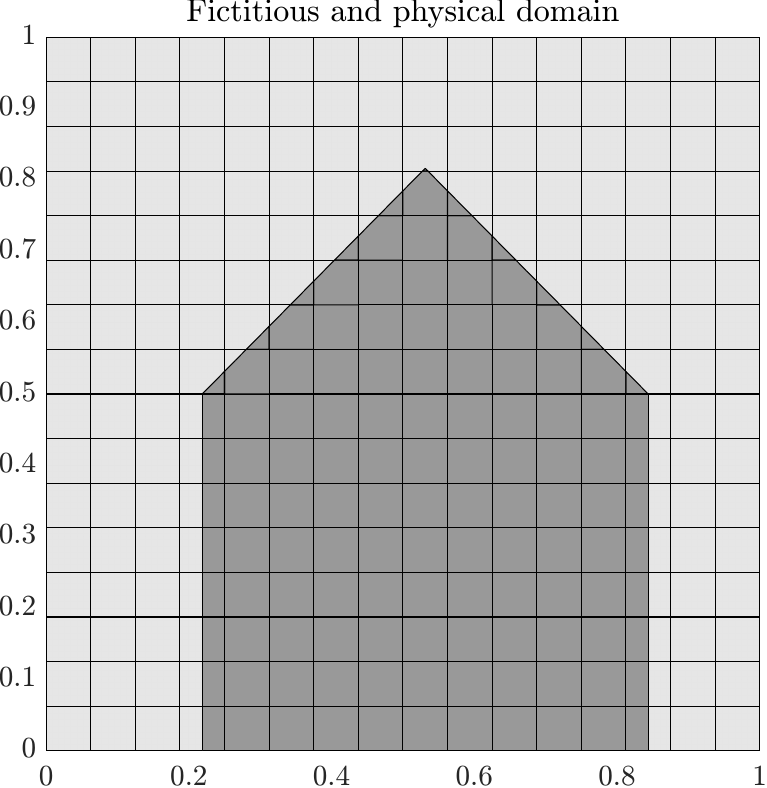}
    \caption{Centered middle-cut}
    \label{fig: 2D_Laplace_ridge_center_mesh}
     \end{subfigure}
     \hfill
    \begin{subfigure}[t]{0.31\textwidth}
    \centering
    \includegraphics[width=\textwidth]{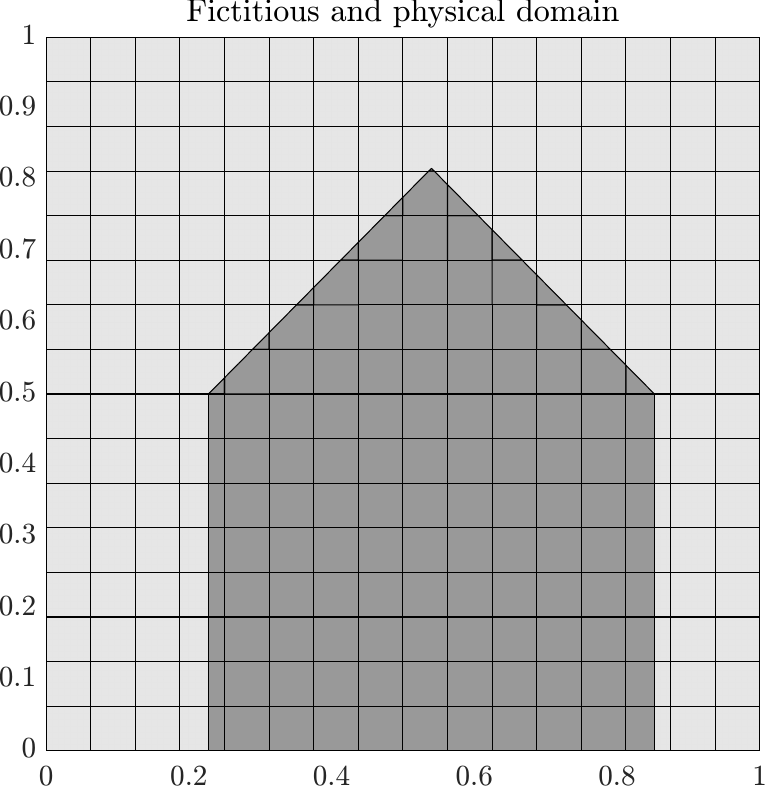}
    \caption{Decentered middle-cut}
    \label{fig: 2D_Laplace_ridge_decenter_mesh}
     \end{subfigure}
     \hfill
    \caption{House-like trimmed geometry with different cut configurations (\Cref{ex: ridge})}
    \label{fig: 2D_Laplace_ridge_mesh}
\end{figure}

\begin{figure}[H]
     \centering
     \begin{subfigure}[t]{0.31\textwidth}
    \centering
    \includegraphics[width=\textwidth]{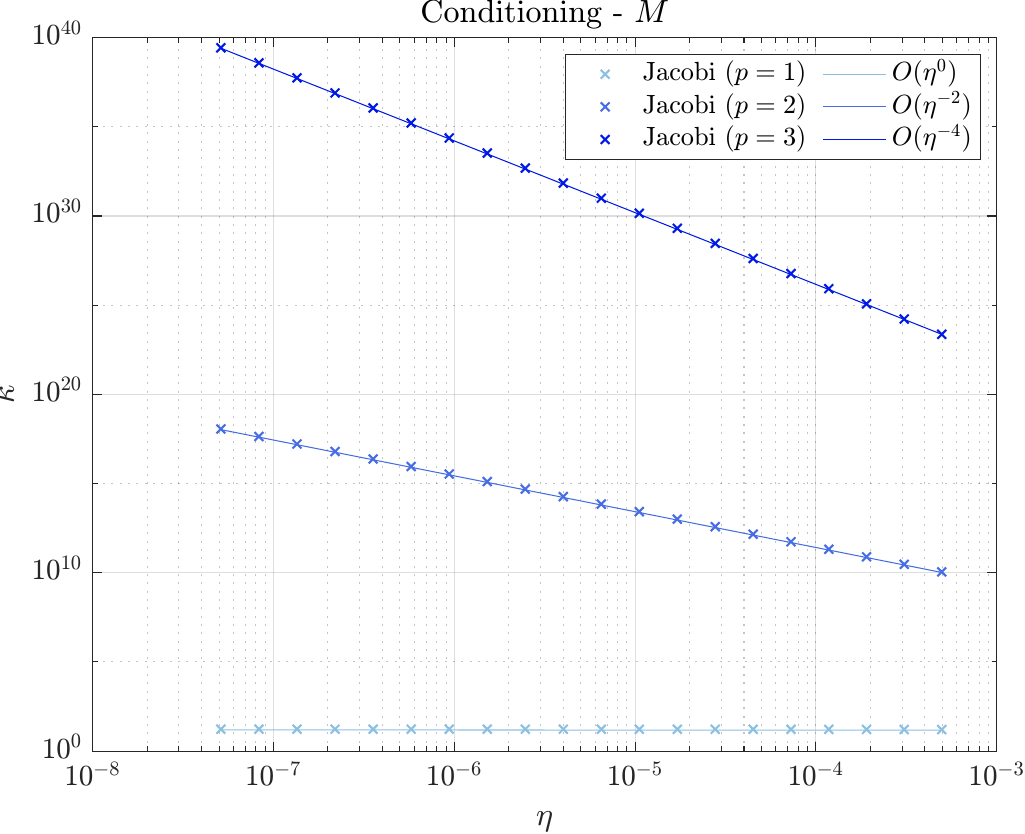}
    \caption{Corner-cut}
    \label{fig: 2D_Laplace_ridge_corner_cond_M_Jacobi_Lagrange}
     \end{subfigure}
     \hfill
     \begin{subfigure}[t]{0.31\textwidth}
    \centering
    \includegraphics[width=\textwidth]{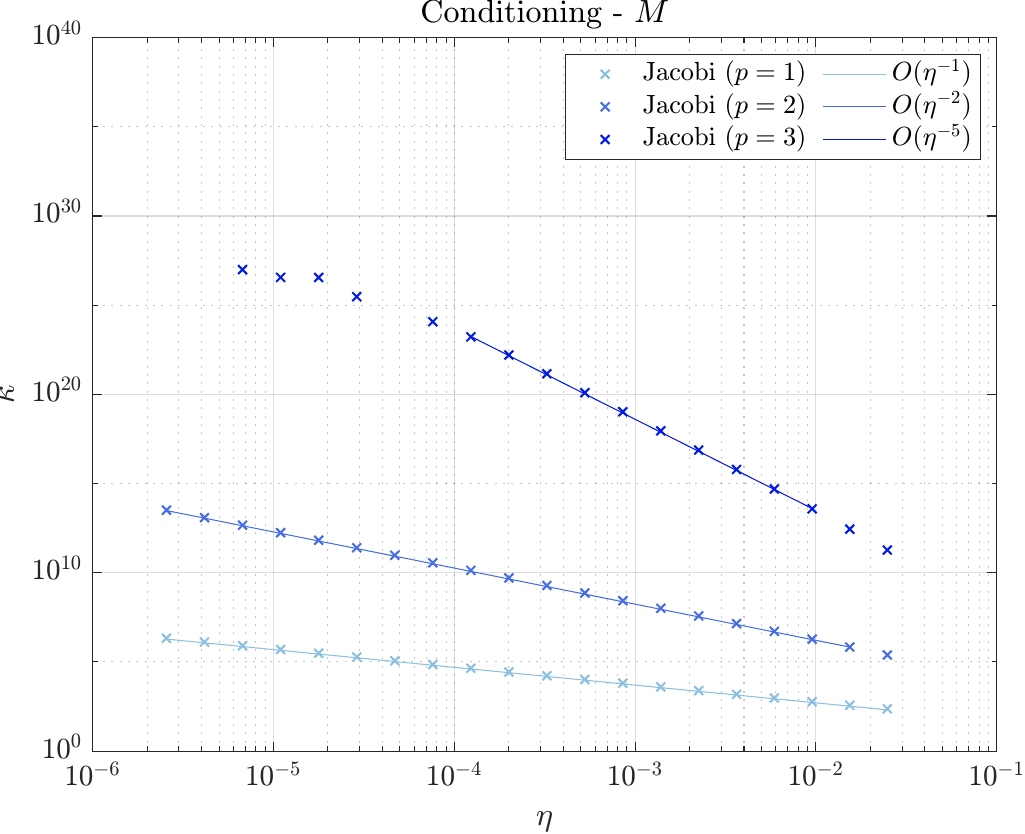}
    \caption{Centered middle-cut}
    \label{fig: 2D_Laplace_ridge_center_cond_M_Jacobi_Lagrange}
     \end{subfigure}
     \hfill
    \begin{subfigure}[t]{0.31\textwidth}
    \centering
    \includegraphics[width=\textwidth]{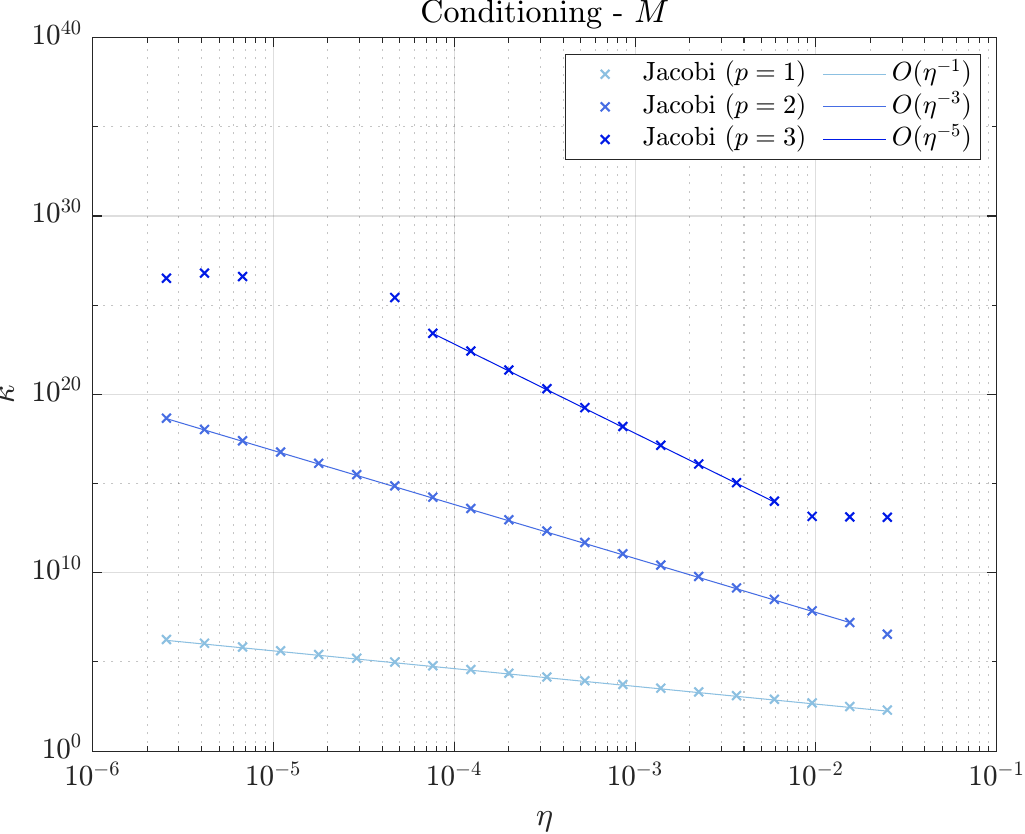}
    \caption{Decentered middle-cut}
    \label{fig: 2D_Laplace_ridge_decenter_cond_M_Jacobi_Lagrange}
     \end{subfigure}
     \hfill
    \caption{Lagrange basis (\Cref{ex: ridge})}
    \label{fig: 2D_Laplace_ridge_cond_M_Jacobi_Lagrange}
\end{figure}

\begin{figure}[H]
     \centering
     \begin{subfigure}[t]{0.31\textwidth}
    \centering
    \includegraphics[width=\textwidth]{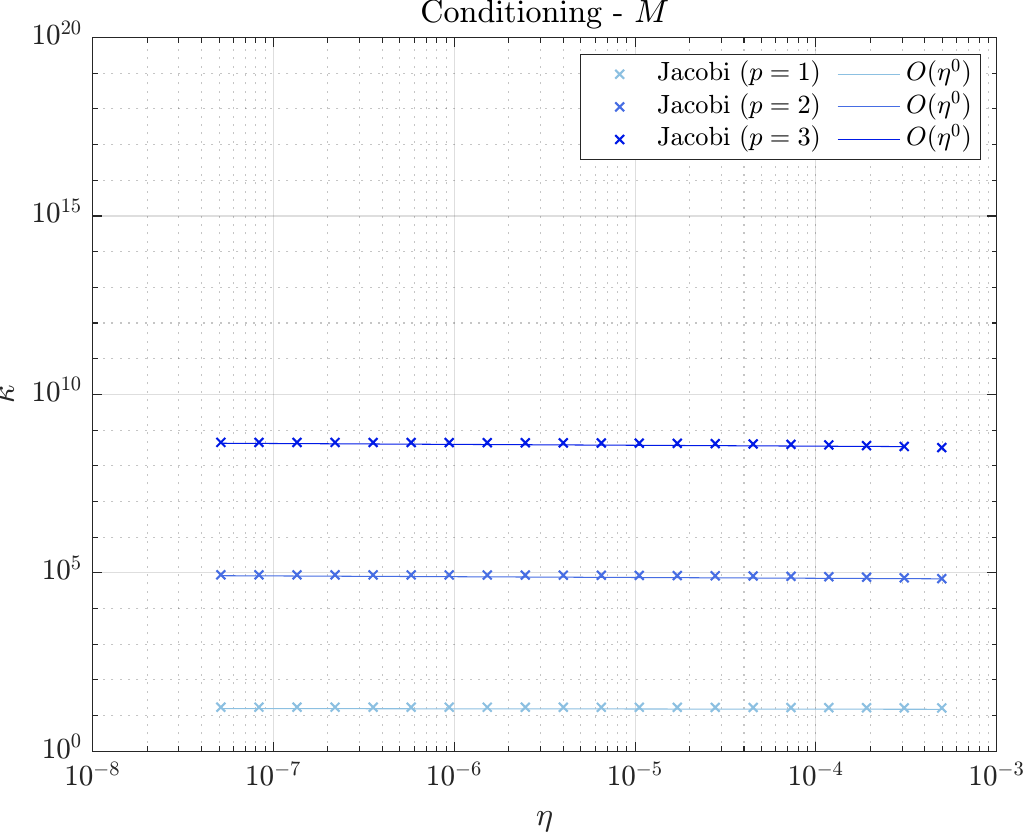}
    \caption{Corner-cut}
    \label{fig: 2D_Laplace_ridge_corner_cond_M_Jacobi_Bspline_C0}
     \end{subfigure}
     \hfill
     \begin{subfigure}[t]{0.31\textwidth}
    \centering
    \includegraphics[width=\textwidth]{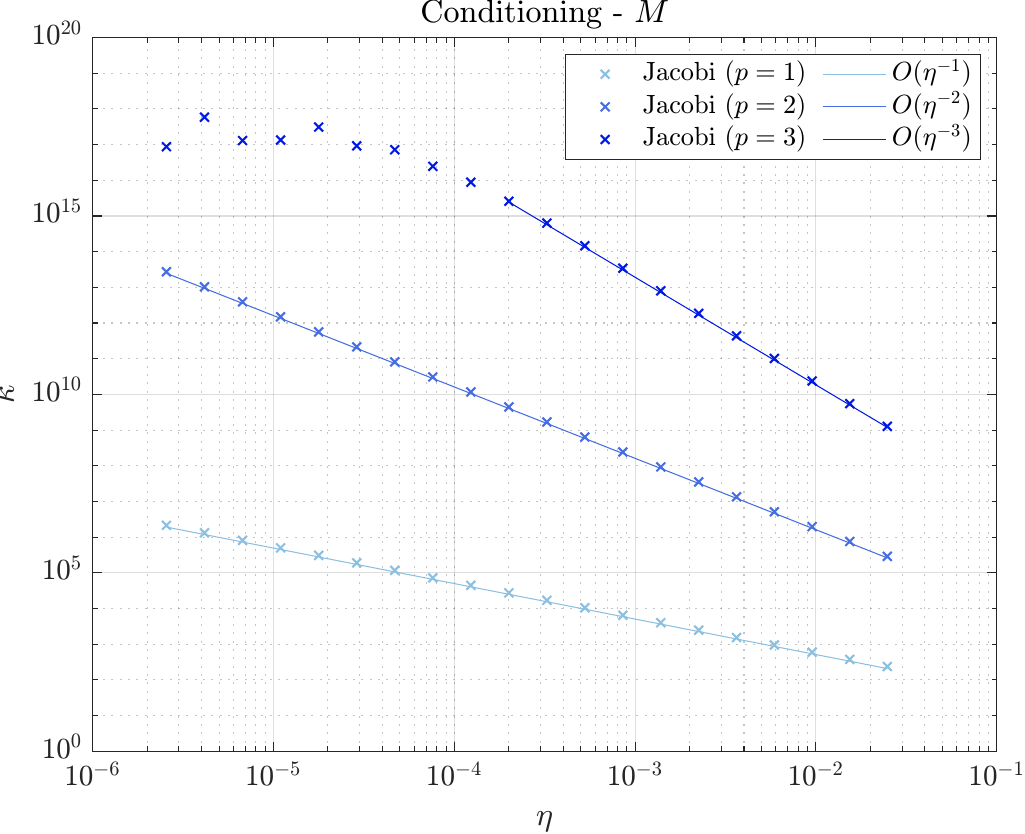}
    \caption{Centered middle-cut}
    \label{fig: 2D_Laplace_ridge_center_cond_M_Jacobi_Bspline_C0}
     \end{subfigure}
     \hfill
    \begin{subfigure}[t]{0.31\textwidth}
    \centering
    \includegraphics[width=\textwidth]{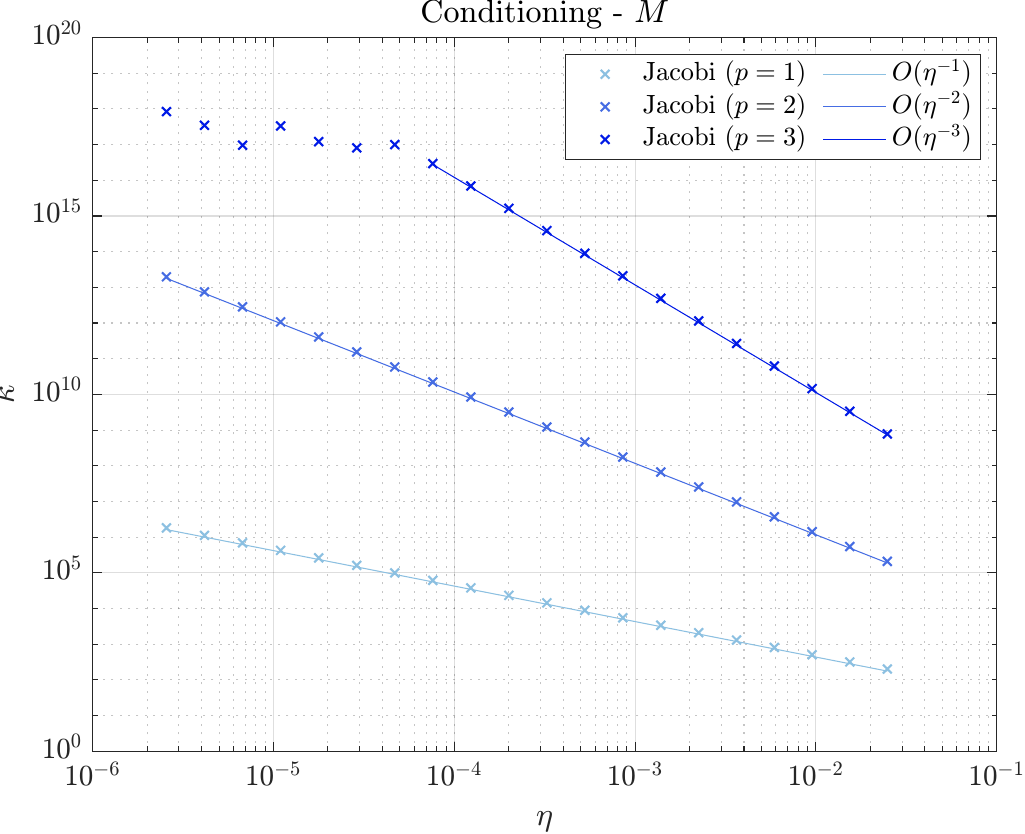}
    \caption{Decentered middle-cut}
    \label{fig: 2D_Laplace_ridge_decenter_cond_M_Jacobi_Bspline_C0}
     \end{subfigure}
     \hfill
    \caption{$C^{0}$ B-spline basis (\Cref{ex: ridge})}
    \label{fig: 2D_Laplace_ridge_cond_M_Jacobi_Bspline_C0}
\end{figure}

\begin{figure}[H]
     \centering
     \begin{subfigure}[t]{0.31\textwidth}
    \centering
    \includegraphics[width=\textwidth]{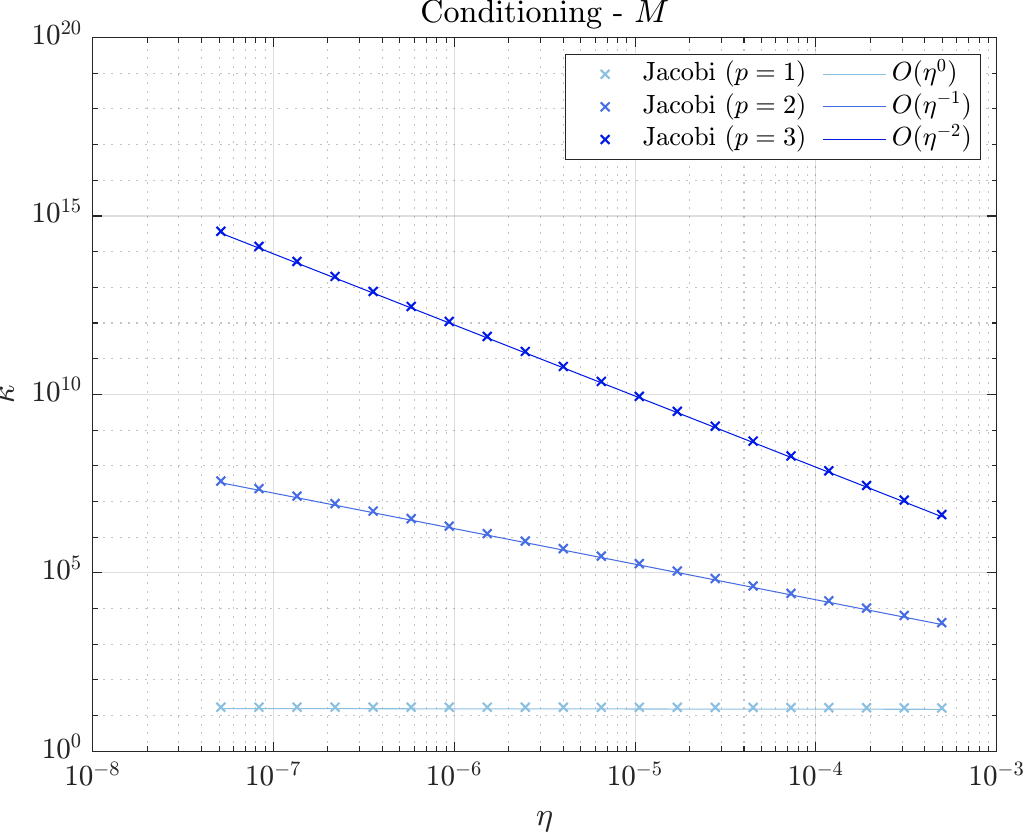}
    \caption{Corner-cut}
    \label{fig: 2D_Laplace_ridge_corner_cond_M_Jacobi_Bspline_Cp-1}
     \end{subfigure}
     \hfill
     \begin{subfigure}[t]{0.31\textwidth}
    \centering
    \includegraphics[width=\textwidth]{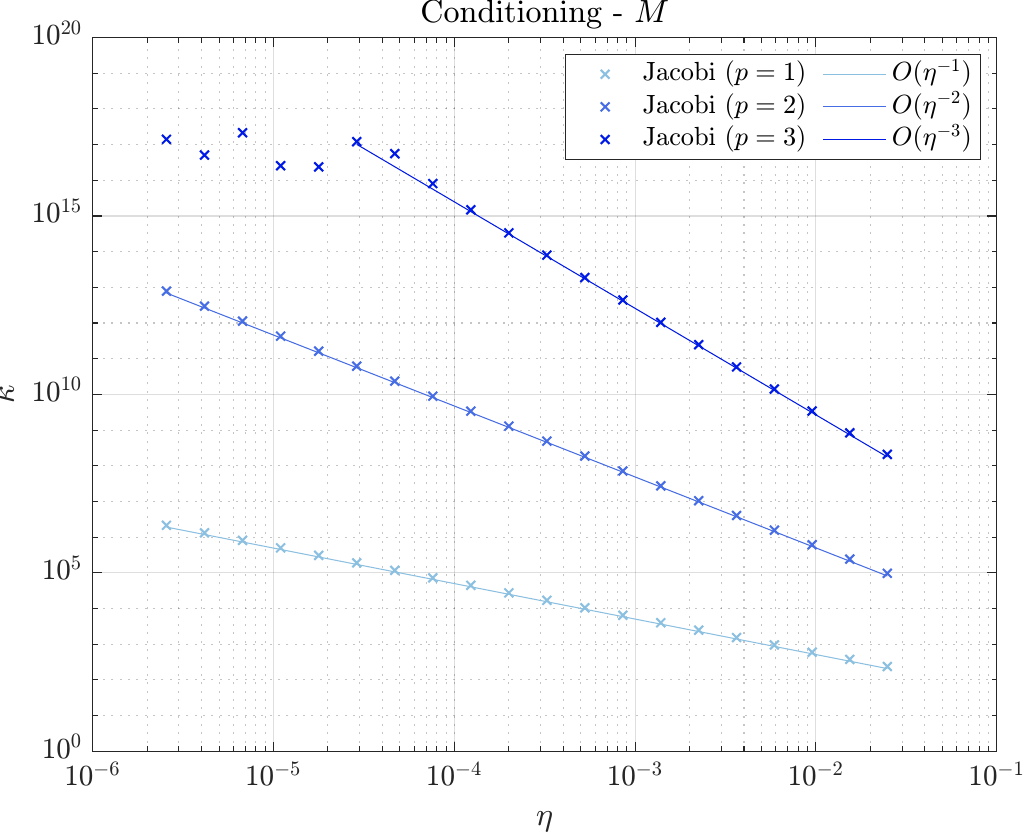}
    \caption{Centered middle-cut}
    \label{fig: 2D_Laplace_ridge_center_cond_M_Jacobi_Bspline_Cp-1}
     \end{subfigure}
     \hfill
    \begin{subfigure}[t]{0.31\textwidth}
    \centering
    \includegraphics[width=\textwidth]{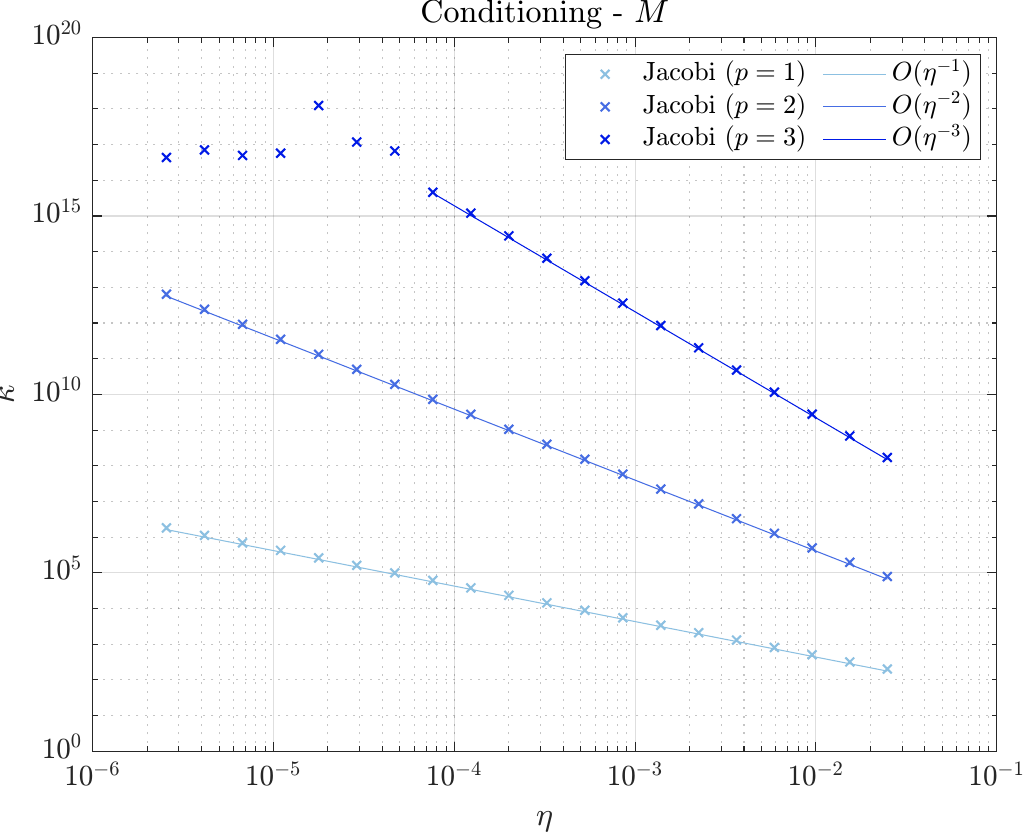}
    \caption{Decentered middle-cut}
    \label{fig: 2D_Laplace_ridge_decenter_cond_M_Jacobi_Bspline_Cp-1}
     \end{subfigure}
     \hfill
    \caption{$C^{p-1}$ B-spline basis (\Cref{ex: ridge})}
    \label{fig: 2D_Laplace_ridge_cond_M_Jacobi_Bspline_Cp-1}
\end{figure}
\end{example}

These examples confirm that diagonal scaling alone is insufficient and must be combined with other forms of preconditioning. The next couple of sections review some of the solutions that have been proposed.

\subsection{Symmetric Incomplete Permuted Inverse Cholesky (SIPIC) preconditioner}
\label{se: sipic}
In 2017, de Prenter et al.\ \cite{de2017condition} presented the Symmetric Incomplete Permuted Inverse Cholesky (SIPIC) preconditioner, an approximate inverse preconditioner for SPD matrices. This technique combines diagonal scaling with local orthonormalizations of nearly linearly dependent basis functions. The strategy for detecting near linear dependencies is based on the magnitude of the off-diagonal entries in the diagonally preconditioned matrix. More specifically, given a threshold $\zeta \in [0,1]$, a pair of basis functions $\{\hat{\varphi}_i,\hat{\varphi}_j\}$ is marked as nearly linearly dependent if
\begin{equation*}
    |a(\hat{\varphi}_i,\hat{\varphi}_j)| > \zeta.
\end{equation*}
From the pairs of indices of nearly linearly dependent functions
\begin{equation*}
    \mathcal{I}=\{(i,j) \colon i>j, \ |a(\hat{\varphi}_i,\hat{\varphi}_j)| > \zeta\},
\end{equation*}
we form a graph $G$ whose vertex set is $V(G)=\{1,\dots,n\}$ and whose edge set is $E(G)=\mathcal{I}$. The connected components of $G$ form non-intersecting sets of indices. The basis functions corresponding to indices within each set are then locally orthonormalized in the energy norm with Gram-Schmidt \cite{golub2013matrix} to produce a new basis $\hat{\Phi}^{(1)}$. However, local orthonormalizations may potentially introduce near linear dependencies elsewhere. Let $\mathcal{I}^{(1)}$ denote the updated set of index pairs obtained by detecting near linear dependencies in $\hat{\Phi}^{(1)}$. If $\mathcal{I}^{(1)} = \emptyset$, the procedure ends there. Otherwise, local orthonormalizations are carried out, and the procedure is repeated until (hopefully) converging to a basis $\Phi^{(k)}$ satisfying $\mathcal{I}^{(k)} = \emptyset$. In \cite{de2017condition}, the authors report convergence in at most two iterations and compelling numerical results. Unfortunately though, the strategy for detecting near linear dependencies is mathematically flawed. Consider the following minimal counter-example for a set of vectors $\{\bm{u}_1,\bm{u}_2,\bm{u}_3\} \subseteq \mathbb{R}^3$ stored along the columns of $U=[\bm{u}_1,\bm{u}_2,\bm{u}_3]$ where
\begin{equation*}
    \bm{u}_1=
    \begin{pmatrix}
        1 \\
        0 \\
        0
    \end{pmatrix},
    \bm{u}_2=
    \begin{pmatrix}
        0 \\
        1 \\
        0
    \end{pmatrix},
    \bm{u}_3=\frac{1}{\sqrt{2}}
    \begin{pmatrix}
        1 \\
        1 \\
        0
    \end{pmatrix}.
\end{equation*}
This set is obviously linearly dependent since $\bm{u}_1+\bm{u}_2=\sqrt{2}\bm{u}_3$. However, it will never be detected from its Gram matrix
\begin{equation*}
    U^TU=
    \begin{pmatrix}
        1 & 0 & \frac{1}{\sqrt{2}} \\
        0 & 1 & \frac{1}{\sqrt{2}} \\
        \frac{1}{\sqrt{2}} & \frac{1}{\sqrt{2}} & 1
    \end{pmatrix}
\end{equation*}
whose off-diagonal entries $\frac{1}{\sqrt{2}} \approx 0.7071$ are nowhere near $1$ in magnitude. Indeed, (near) linear dependencies in a set of $m \geq 3$ functions may arise without necessarily \emph{pairwise} dependencies, which are only necessary and sufficient for $m=2$. Although the shortcomings of SIPIC were already acknowledged in \cite[Remark 3.1]{de2019preconditioning}, explicit counter-examples were never provided, neither in this reference nor in follow-up work. Yet, this problem is not merely theoretical and may arise in very simple settings, as shown in the next counter-examples. Following the description of SIPIC in \cite{de2017condition}, we chose $\zeta = 0.9$ as default threshold in all our examples. Moreover, as explained in \Cref{se: ill_conditioning_origin}, the weak imposition of Dirichlet boundary conditions is not the main cause of ill-conditioning. Therefore, unless specified otherwise, Neumann boundary conditions are prescribed on all boundaries, and we simply ignore the zero eigenvalues of the stiffness tied to rigid-body modes.

\begin{example}[Rotated square with a hole]
\label{ex: square_with_hole}
This counter-example is a minor variation of an example that inspired the design of SIPIC itself \cite{de2017condition}: a rotated square with a hole in the middle\footnote{The authors thank Frits de Prenter who inspired this example.}. The quadratic Lagrange basis constructed over the background mesh in \Cref{fig: 2D_Laplace_square_with_hole_mesh} perfectly mimics the original example in \cite[Fig. 2]{de2017condition}. In fact, the only difference is that instead of rotating the square around its center, we slightly enlarge the radius $r=\sqrt{5}h-\delta$ of the hole by reducing $\delta$ and therefore the size of the smallest trimmed element, as shown in \Cref{fig: 2D_Laplace_square_with_hole_mesh_detail}. Unfortunately, the condition number of the SIPIC preconditioned matrix now tends to increase (\Cref{fig: 2D_Laplace_square_with_hole_cond_Lagrange_p2}), although it still reduces the quadratic growth of the condition number for the Jacobi preconditioner to a linear one. The reason is that it successfully eliminates the smallest eigenvalue but not the second smallest one, which still converges to zero, albeit at a milder linear rate. Unfortunately, future counter-examples will show that this is not always the case and SIPIC may not improve the rate at all.

\begin{figure}[H]
     \centering
     \begin{subfigure}[t]{0.48\textwidth}
    \centering
    \includegraphics[width=\textwidth]{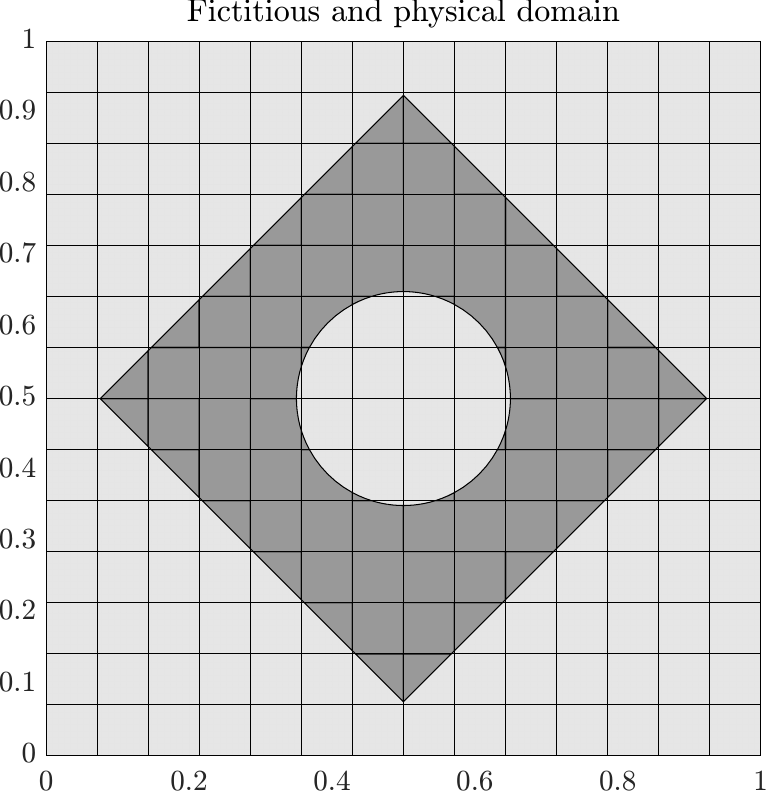}
    \caption{Geometry and background mesh}
    \label{fig: 2D_Laplace_square_with_hole_mesh}
     \end{subfigure}
     \hfill
     \begin{subfigure}[t]{0.48\textwidth}
    \centering
    \includegraphics[width=\textwidth]{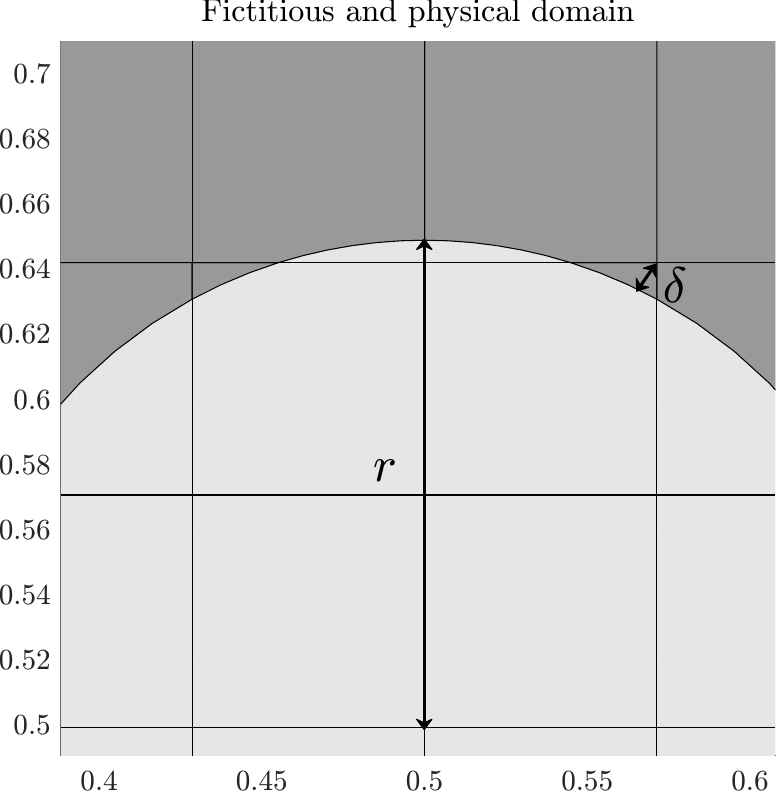}
    \caption{Detail of the hole}
    \label{fig: 2D_Laplace_square_with_hole_mesh_detail}
     \end{subfigure}
     \hfill
    \caption{Rotated square with a hole (\Cref{ex: square_with_hole})}
    \label{fig: 2D_Laplace_square_with_hole}
\end{figure}

\begin{figure}[H]
     \centering
     \begin{subfigure}[t]{0.48\textwidth}
    \centering
    \includegraphics[width=\textwidth]{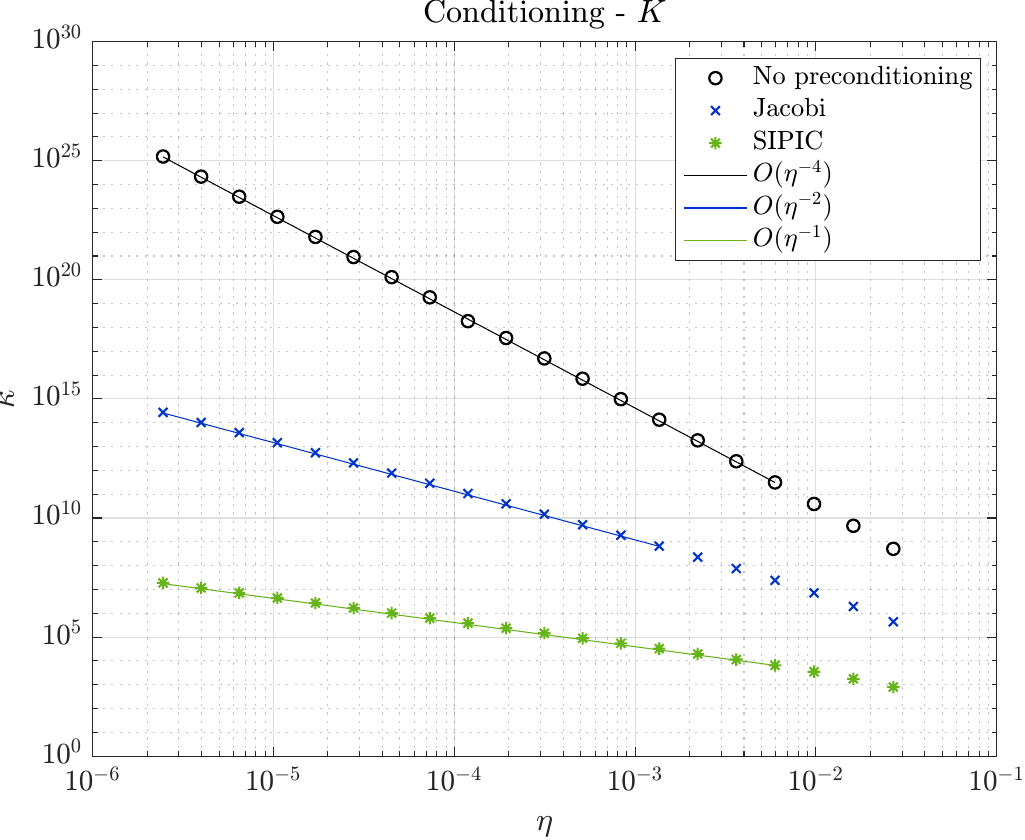}
    \caption{Stiffness matrix}
    \label{fig: 2D_Laplace_square_with_hole_cond_K_Lagrange_p2}
     \end{subfigure}
     \hfill
     \begin{subfigure}[t]{0.48\textwidth}
    \centering
    \includegraphics[width=\textwidth]{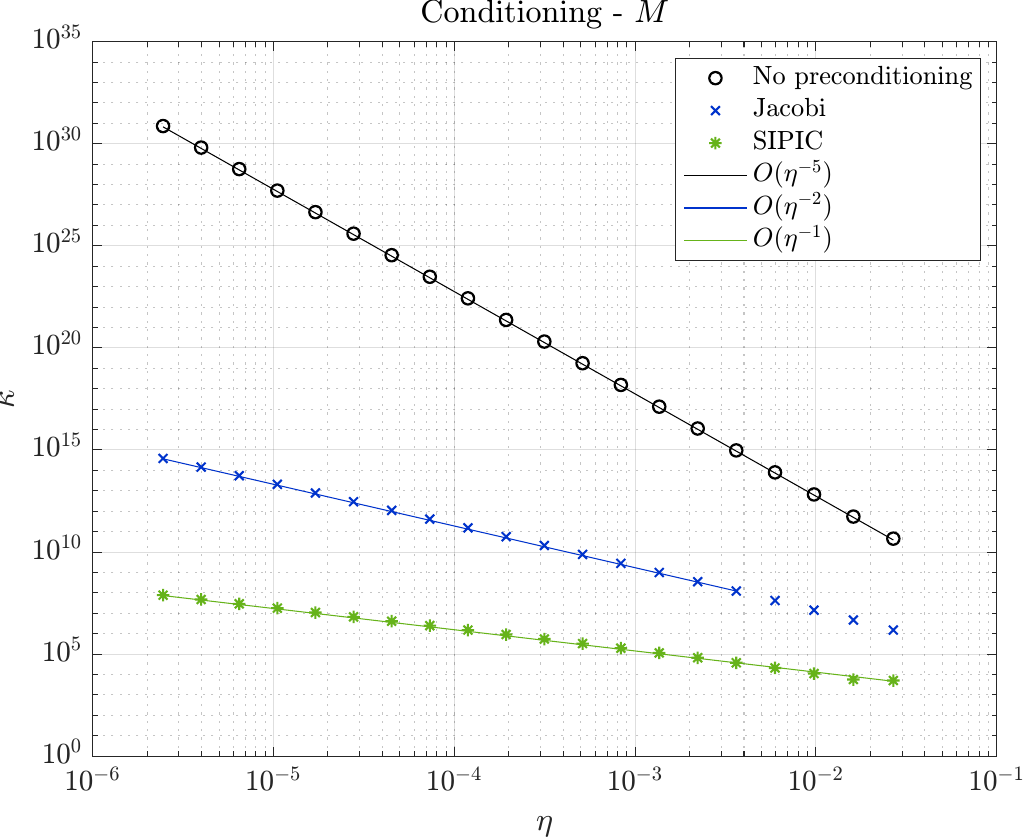}
    \caption{Mass matrix}
    \label{fig: 2D_Laplace_square_with_hole_cond_M_Lagrange_p2}
     \end{subfigure}
     \hfill
    \caption{Conditioning of the stiffness and mass matrices for the quadratic Lagrange basis (\Cref{ex: square_with_hole})}
    \label{fig: 2D_Laplace_square_with_hole_cond_Lagrange_p2}
\end{figure}    
\end{example}

Alas, SIPIC may also break down for (smooth) spline bases, as shown in the next counter-example.

\begin{example}[Two ridges]
\label{ex: 2ridges}
We consider a slight modification of the house-like geometry of \Cref{ex: ridge} by splitting the roof's ridge into two ridges separated by a distance $h$ in \Cref{fig: 2D_Laplace_2tips_mesh_C0} and $3h$ in \Cref{fig: 2D_Laplace_2tips_mesh_Cp-1}, while the ridges' tips are at a distance $\delta$ from the nearest horizontal grid line. In the first case, we build a cubic Bernstein basis over the background mesh while in the second case we consider a maximally smooth cubic B-spline basis. Finally, Dirichlet boundary conditions are prescribed along the base and Neumann boundary conditions elsewhere. In both the $C^0$ (\Cref{fig: 2D_Laplace_2tips_cond_K_Bspline_p3_C0}) and $C^{p-1}$ (\Cref{fig: 2D_Laplace_2tips_cond_K_Bspline_p3_Cp-1}) cases, the condition number of the SIPIC preconditioned stiffness matrix increases at the same rate as the Jacobi preconditioner, although with a smaller constant. This rate matches the one predicted in \Cref{tab: scaling_jacobi}. Similar results are obtained for the mass matrix and are omitted. Although these counter-examples were obviously intended to highlight SIPIC's flaws, they are not completely far-fetched and could arise for an unfortunate placement of the underlying grid.

\begin{figure}[H]
     \centering
     \begin{subfigure}[t]{0.48\textwidth}
    \centering
    \includegraphics[width=\textwidth]{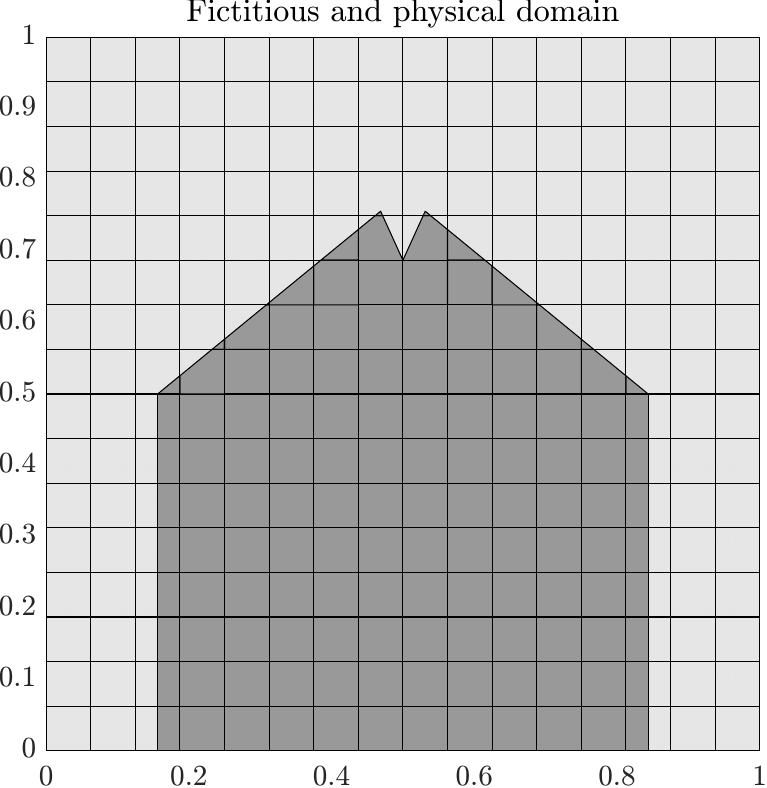}
    \caption{Cubic $C^0$ continuity}
    \label{fig: 2D_Laplace_2tips_mesh_C0}
     \end{subfigure}
     \hfill
     \begin{subfigure}[t]{0.48\textwidth}
    \centering
    \includegraphics[width=\textwidth]{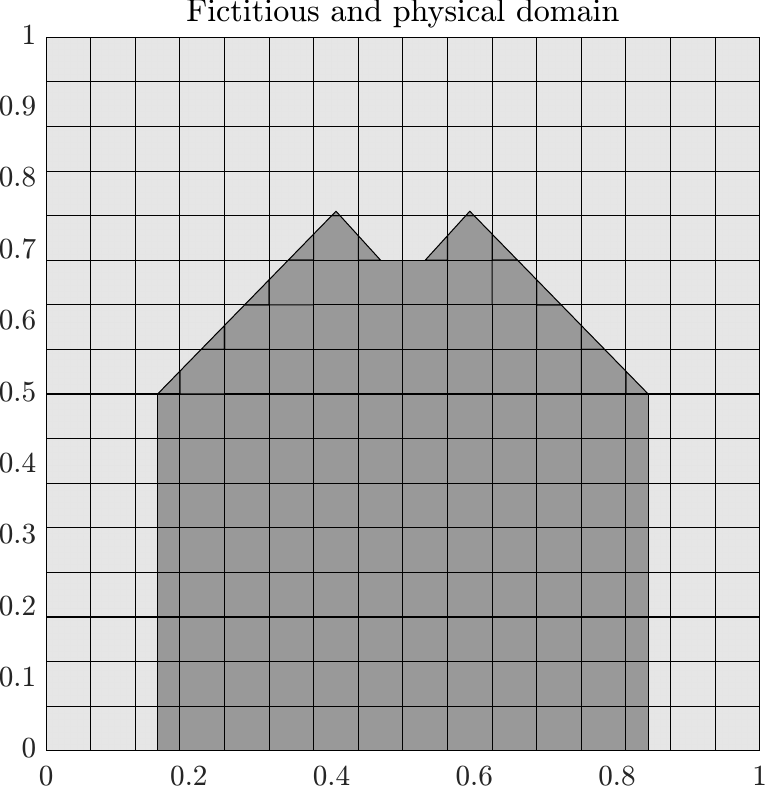}
    \caption{Cubic $C^2$ continuity}
    \label{fig: 2D_Laplace_2tips_mesh_Cp-1}
     \end{subfigure}
     \hfill
    \caption{Geometries for counter-examples (\Cref{ex: 2ridges})}
    \label{fig: 2D_Laplace_2tips_mesh}
\end{figure}

\begin{figure}[H]
     \centering
     \begin{subfigure}[t]{0.48\textwidth}
    \centering
    \includegraphics[width=\textwidth]{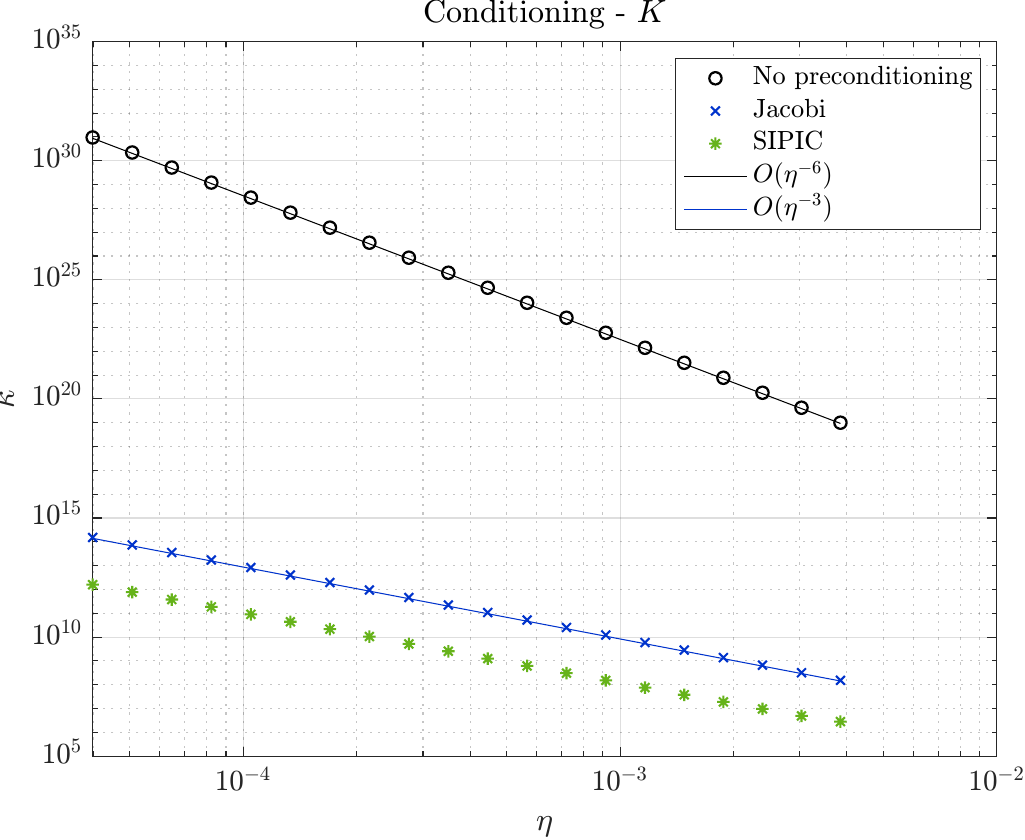}
    \caption{Cubic $C^0$ continuity}
    \label{fig: 2D_Laplace_2tips_cond_K_Bspline_p3_C0}
     \end{subfigure}
     \hfill
     \begin{subfigure}[t]{0.48\textwidth}
    \centering
    \includegraphics[width=\textwidth]{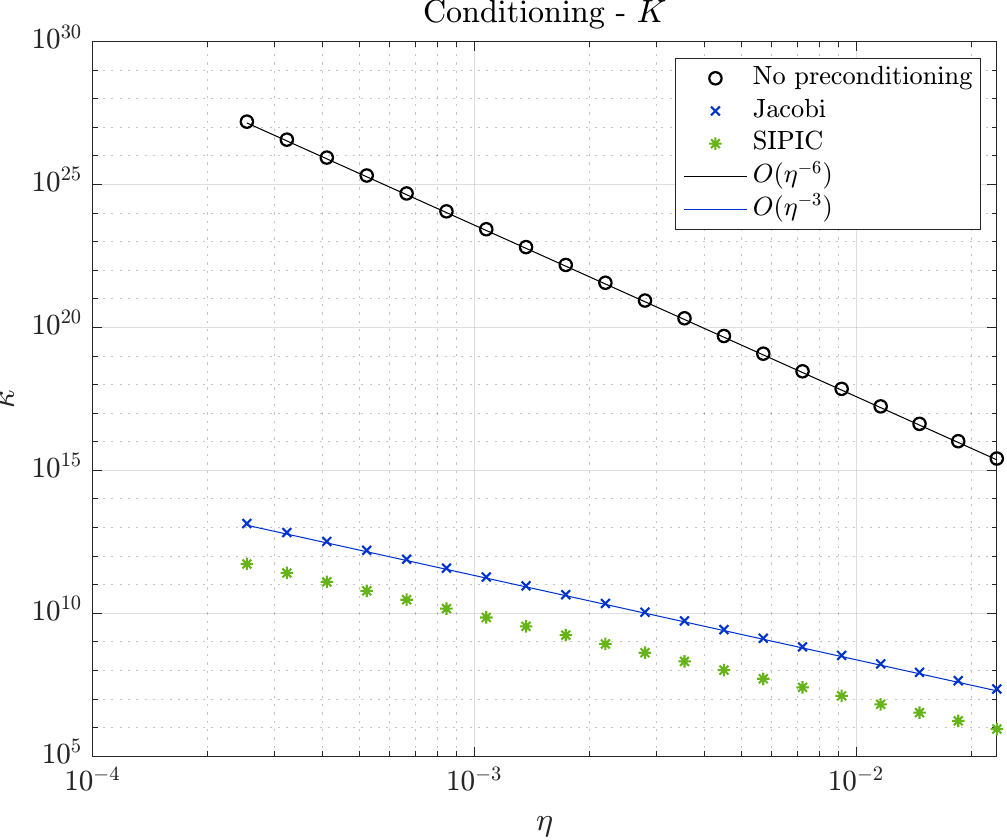}
    \caption{Cubic $C^2$ continuity}
    \label{fig: 2D_Laplace_2tips_cond_K_Bspline_p3_Cp-1}
     \end{subfigure}
     \hfill
    \caption{Conditioning of the stiffness matrix for the B-spline basis (\Cref{ex: 2ridges})}
    \label{fig: 2D_Laplace_2tips_cond_Bspline_p3_C0}
\end{figure}    
\end{example}

\subsection{Schwarz preconditioners}
A couple of years after SIPIC, de Prenter et al.\ proposed a different approximate inverse preconditioner in the form of an additive \cite{de2019preconditioning} or multiplicative \cite{de2020multigrid} Schwarz preconditioner. Given $N$ (potentially overlapping) index blocks $\{\mathcal{K}_i\}_{i=1}^N$, the preconditioning matrices are defined as
\begin{align*}
    S &= \sum_{i=1}^N S_i = \sum_{i=1}^N P_i(P_i^TAP_i)^{-1}P_i^T & &\text{(Additive Schwarz)},\\
    S &= \sum_{i=1}^N S_i \prod_{j=1}^{i-1}(I-AS_j) & &\text{(Multiplicative Schwarz)},
\end{align*}
where $S_i=P_i(P_i^TAP_i)^{-1}P_i^T$ and $P_i=[\bm{e}_{\mathcal{K}_i(1)},\dots,\bm{e}_{\mathcal{K}_i(m)}]$ contains the columns of the identity matrix for the indices of the $i$th index block $\mathcal{K}_i$. Contrary to other preconditioning techniques, the Schwarz method handles more general matrices provided the submatrices $A_i=P_i^TAP_i$ are invertible. This is obviously the case if $A$ is SPD. Moreover, in order to ensure the invertibility of the preconditioning matrix $S$, index blocks containing single indices are formed for basis functions that are not supported on any cut element such that $S_i=P_i(P_i^TAP_i)^{-1}P_i^T$ merely inverts the corresponding diagonal entry in $A$. Nevertheless, due to the ill-conditioning of the submatrices $A_i$ entering the definition of $S$, one should avoid explicitly forming the preconditioning matrix and instead define it implicitly through matrix-vector multiplications. Those operations only require solving small linear systems with the submatrices $A_i$, which is more stable than forming their inverse. Alternatively, one may approximately form the inverses of $A_i$ by truncating their eigendecompositions; see \cite{jomo2019robust}. To mitigate this last issue, we suggest combining the Jacobi and Schwarz preconditioners by forming the Schwarz preconditioner from the Jacobi preconditioned matrix. This strategy drastically helps stabilize the computations since the submatrices $A_i$ are now those of the diagonally scaled matrix, which is already much better conditioned (see \Cref{se: diagonal_scaling}).

However, the quality of the preconditioner critically depends on the strategy for selecting the index blocks, and, unfortunately, the community has not agreed on the strategy to adopt \cite{de2023stability}. Some of the schemes proposed in the literature are summarized below. 
\begin{itemize}[noitemsep]
    \item In \cite{de2019preconditioning}, the authors suggested forming index blocks from the indices of the basis functions intersecting on cut elements. The subset of cut elements is defined as
    \begin{equation*}
        \mathcal{T}_C = \{T \in \mathcal{T}_h \colon |T \cap \Omega | < |T| \} \subseteq \mathcal{T}_h.
    \end{equation*}
    An index block $\mathcal{K}_i$ is then devised for each cut element $T_i \in \mathcal{T}_C$ such that
    \begin{equation}
    \label{eq: block_sel_cut_elem}
        \mathcal{K}_i = \{j \colon T_i \subseteq \supp(\varphi_j) \}.
    \end{equation}
    Several variants of this approach were also suggested for multi-level $h,p$-refined discretizations \cite{jomo2019robust,jomo2021hierarchical}. In \cite{jomo2019robust}, index blocks were only devised for small cut elements and further truncated to selected degrees of freedom, whereas in \cite{jomo2021hierarchical} a patchwise strategy was suggested by regrouping elements around certain nodes.
    \item Instead, the authors in \cite{de2020multigrid} formed index blocks for each trimmed basis function based on support containment or intersection. For instance, in \cite{de2020multigrid}, the index block associated with $\varphi_i$ is defined as
    \begin{equation}
    \label{eq: block_sel_supp_cont}
        \mathcal{K}_i = \{j \colon \supp_{\Omega}(\varphi_j) \subseteq \supp_{\Omega}(\varphi_i)\}
    \end{equation}
    where $\supp_{\Omega}(\varphi)=\supp(\varphi_i) \cap \Omega$ denotes the active support. A related but more stringent criterion is based on intersecting supports \cite{de2023stability}:
    \begin{equation}
    \label{eq: block_sel_supp_int}
        \mathcal{K}_i = \{j \colon \supp_{\Omega}(\varphi_j) \cap \supp_{\Omega}(\varphi_i) \neq \emptyset\}.
    \end{equation}
\end{itemize}
Unfortunately, the next counter-example shows that none of these strategies are perfectly robust.

\begin{example}[Three ridges]
\label{ex: 3ridges}
We now generalize \Cref{ex: 2ridges} to three ridges by duplicating the configuration in \Cref{fig: 2D_Laplace_2tips_mesh_Cp-1}\footnote{The authors thank Frits de Prenter who inspired this example.}. A smooth cubic B-spline basis is constructed over the background mesh shown in \Cref{fig: 2D_Laplace_3tips_mesh_Cp-1}. In addition to SIPIC, \Cref{fig: 2D_Laplace_3tips_cond_Bspline_p3_Cp-1} shows the condition number for the additive Schwarz preconditioner for the block selection strategies based on cut elements \eqref{eq: block_sel_cut_elem} and intersecting supports \eqref{eq: block_sel_supp_int}. While the block selection strategy based on cut elements performs just as poorly as SIPIC, the one based on intersecting supports reduces the condition number to a linear growth. Thus, while advanced block selection strategies certainly help, none of them are perfectly robust to this date. Although we have not tested the multiplicative Schwarz technique, it generally cannot be expected to succeed when the additive one fails. It is also more expensive to apply than the additive version \cite{de2020multigrid}.

\begin{figure}[H]
     \centering
     \begin{subfigure}[t]{0.48\textwidth}
    \centering
    \includegraphics[width=\textwidth]{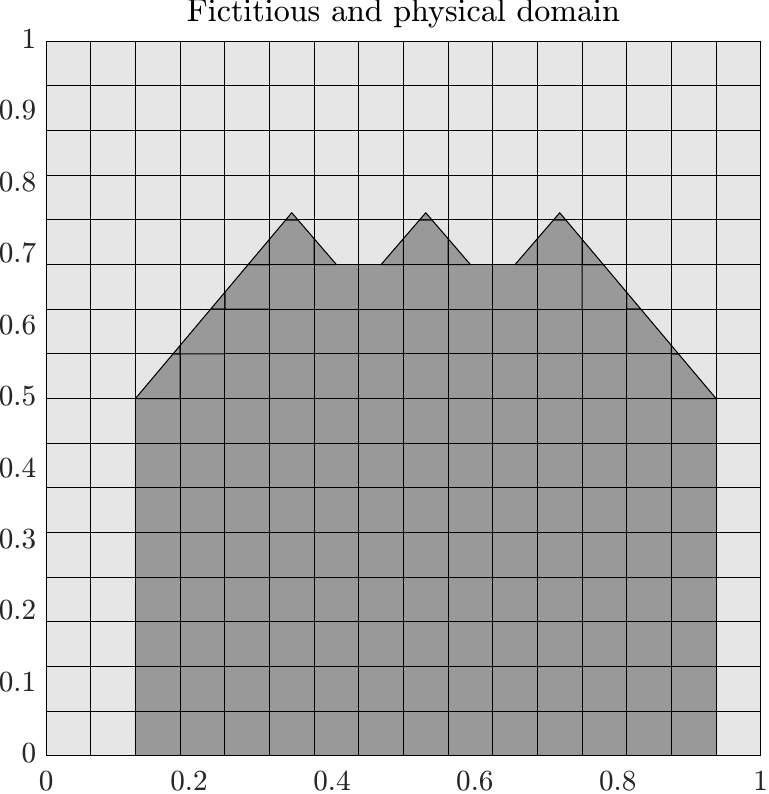}
    \caption{Geometry and background mesh}
    \label{fig: 2D_Laplace_3tips_mesh_Cp-1}
     \end{subfigure}
     \hfill
     \begin{subfigure}[t]{0.48\textwidth}
    \centering
    \includegraphics[width=\textwidth]{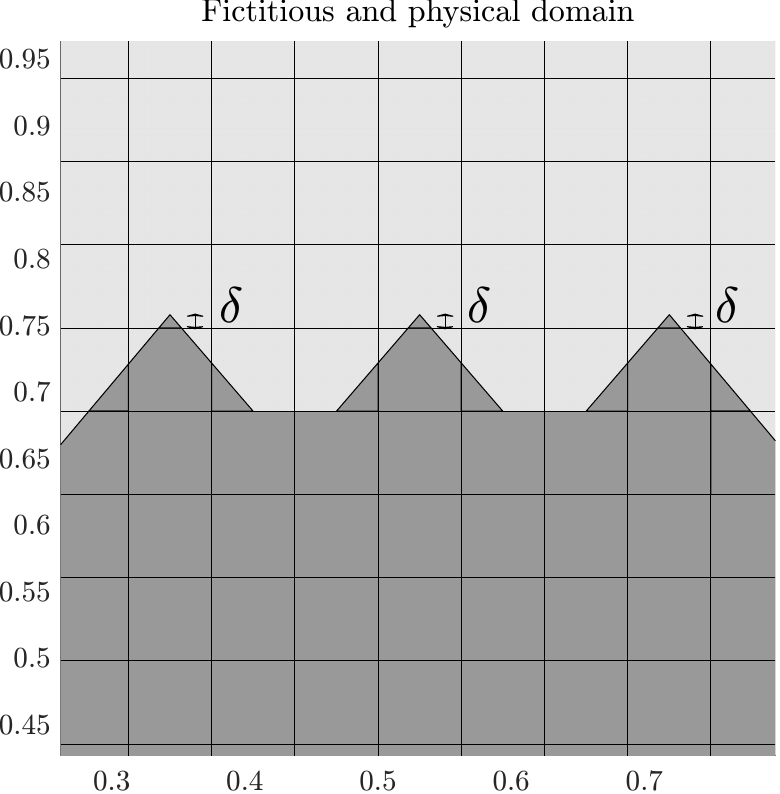}
    \caption{Detail of the ridges}
    \label{fig: 2D_Laplace_3tips_mesh_Cp-1_detail}
     \end{subfigure}
     \hfill
    \caption{Three ridges (\Cref{ex: 3ridges})}
    \label{fig: 2D_Laplace_3tips_mesh}
\end{figure} 

\begin{figure}[H]
     \centering
     \begin{subfigure}[t]{0.48\textwidth}
    \centering
    \includegraphics[width=\textwidth]{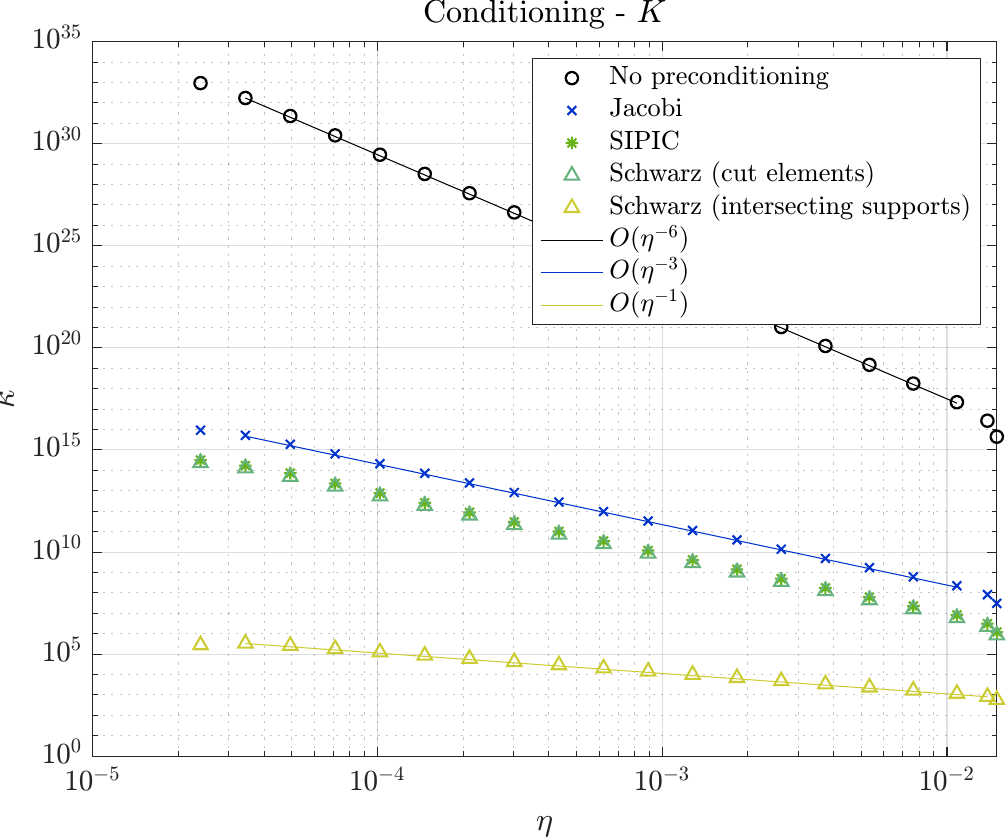}
    \caption{Stiffness matrix}
    \label{fig: 2D_Laplace_3tips_cond_K_Bspline_p3_Cp-1}
     \end{subfigure}
     \hfill
     \begin{subfigure}[t]{0.48\textwidth}
    \centering
    \includegraphics[width=\textwidth]{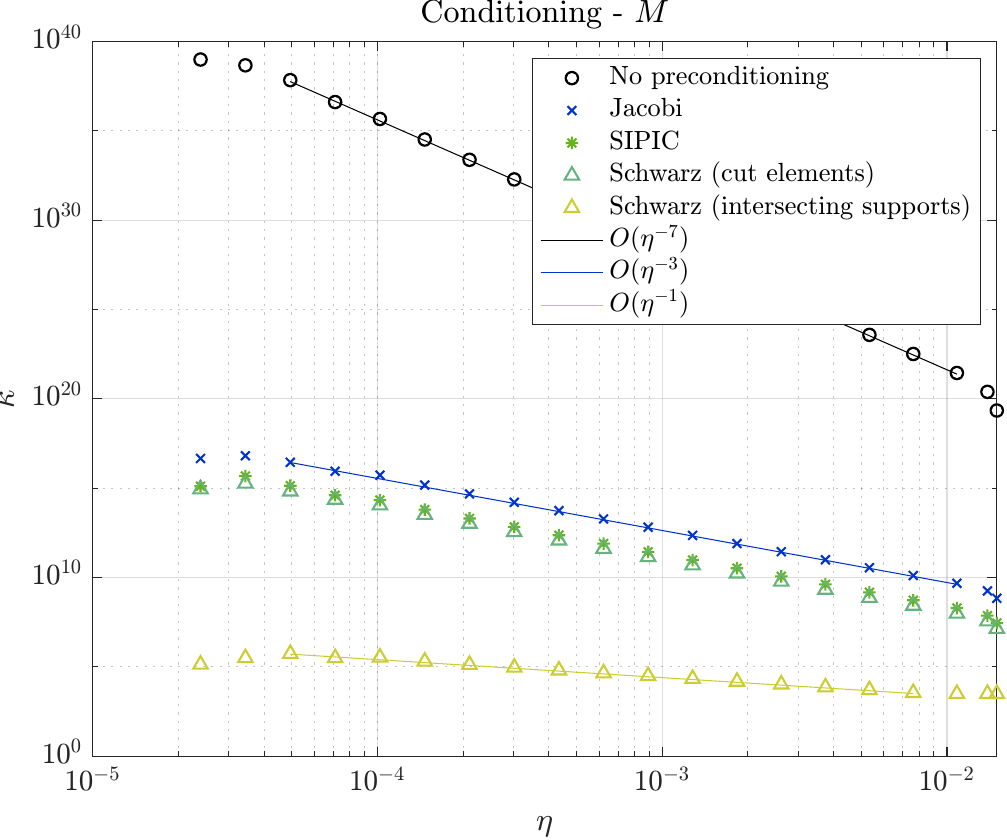}
    \caption{Mass matrix}
    \label{fig: 2D_Laplace_3tips_cond_M_Bspline_p3_Cp-1}
     \end{subfigure}
     \hfill
    \caption{Conditioning of the stiffness and mass matrices for the cubic $C^2$ B-spline basis (\Cref{ex: 3ridges})}
    \label{fig: 2D_Laplace_3tips_cond_Bspline_p3_Cp-1}
\end{figure}    
\end{example}

Our readers must not draw general conclusions from those counter-examples. In the vast majority of cases, both the SIPIC and Schwarz preconditioners perform satisfactorily. Our counter-examples are not meant to discredit them but only to highlight their flaws. In fact, once those flaws have been identified, one can easily construct counter-examples for any given mesh size, smoothness, and spline order. Although artificial, they are not as fanciful as one would expect, especially not when they are based on those same geometries that served to design the preconditioners. Our findings further cast doubt on the true effectiveness of multigrid techniques \cite{de2020multigrid,jomo2021hierarchical}, where Schwarz preconditioners have been used as smoothers. Consequently, a perfectly robust strategy is still sought. The next section revisits an old idea that has proven effective in many different settings.

\section{Deflation-based preconditioner}
\label{se: deflation}
As we have seen, the ill-conditioning and subsequent slowdown of iterative solvers is caused by a relatively small number of near zero eigenvalues. In numerical linear algebra, \emph{deflation} refers to a broad collection of methods for removing unwanted eigenvalues. The original idea dates as far back as the 1940s, when Hotelling \cite{hotelling1943some} introduced the concept of deflation ``by substraction''. Deflation may indeed take various forms and applies to eigenvalue problems and linear systems alike (see e.g., \citep{parlett1998symmetric, saad2011numerical} for an overview of deflation techniques for eigenvalue problems and \cite{frank2001construction,tang2009comparison} for linear systems). Deflation techniques were recently proposed in IGA for outlier removal \cite{voet2026mass} or eigenvalue stabilization \cite{eisentrager2024eigenvalue,burchner2026generalized} but not directly for preconditioning. Although eigenvalue stabilization also improves the conditioning, it directly alters the system matrix and cannot be considered as a form of preconditioning per se. In this section, we employ deflation as a preconditioner to eliminate the small eigenvalues resulting from badly trimmed elements. Since the IGA community might not be familiar with deflation techniques for linear systems, short proofs are included for completeness.

\subsection{Introduction}
In this section, we recall some general properties of deflation-based preconditioning for a linear system $A\bm{x}=\bm{b}$ with an SPD matrix $A$. Given a full-rank matrix $Z \in \mathbb{R}^{n \times r}$ with $r<n$, we define the projector
\begin{equation}
\label{eq: projector}
    P = I-AZ(Z^TAZ)^{-1}Z^T.
\end{equation}
Since $A$ is SPD and $Z$ is full-rank, the so-called \emph{coarse matrix} $E=Z^TAZ$ is also SPD and therefore invertible. Instead of solving the original system, we solve the projected system
\begin{equation}
\label{eq: projected_system}
PA\tilde{\bm{x}}=P\bm{b}.
\end{equation}
Note that the projected matrix $PA$ is symmetric but singular. Some straightforward properties are summarized in the next lemma. Hereafter, we use the Loewner order on symmetric matrices and write $A \succeq B$ (resp. $A \succ B$) to indicate that $A-B$ is positive semidefinite (resp. positive definite) and let $R(Z)$ denote the range of $Z$.

\begin{lemma}
\label{lem: basic_properties}
Let $A \succ 0$. Then,
\begin{enumerate}[label=(\arabic*), itemsep=0pt, parsep=0pt]
    \item $PA=AP^T$, \label{prop: conjugacy}
    \item $P^TZ=0$, \label{prop: PtZ}
    \item $PAZ=0$, \label{prop: PAZ}
    \item $P^2=P$, \label{prop: projector}
    \item $PA \succeq 0$, \label{prop: SPSP}
    \item $\ker(PA)=R(Z)$. \label{prop: ker_PA}
\end{enumerate}
\end{lemma}
\begin{proof}
Properties \ref{prop: conjugacy}, \ref{prop: PtZ}, \ref{prop: PAZ} and \ref{prop: projector} are immediate (see e.g., \cite{vermolen2004deflation,tang2008two} for detailed proofs). The latter also confirms that $P$ is a projection. The symmetry of $PA$ directly follows from \ref{prop: conjugacy} and the symmetry of $A$. Its positive semidefiniteness is a consequence of \ref{prop: projector}. Indeed, since $A \succ 0$, then $PAP^T \succeq 0$ \cite{horn2012matrix} and consequently $PA=(PA)^T=(P^2A)^T=PAP^T \succeq 0$. For proving \ref{prop: ker_PA}, we first deduce from \ref{prop: PAZ} that $R(Z) \subseteq \ker(PA)$ and from the definition of $P$ in \eqref{eq: projector} and the invertibility of $A$, $\ker(PA) \subseteq R(Z)$. Therefore, $\ker(PA)=R(Z)$.
\end{proof}

As already mentioned, the convergence of the Conjugate Gradient (CG) method is plagued by small near zero eigenvalues. The goal of deflation is to shift all those eigenvalues to zero (in exact arithmetic). As shown in the next lemma, this is perfectly achievable if the columns of $Z$ contain the associated eigenvectors. Here again, we assume that the eigenvalues are numbered in ascending algebraic order.

\begin{lemma}
\label{lem: eigenvectors}
Let $\{(\lambda_k,\bm{v}_k)\}_{k=1}^n$ denote the eigenpairs of an SPD matrix $A$. Then, if $Z=[\bm{v}_1,\dots,\bm{v}_r]$ contains the $r$ smallest orthonormal eigenvectors of $A$, then
\begin{equation*}
    PA\bm{v}_k =
    \begin{cases}
        0 & \text{for } k=1,\dots,r, \\
        \lambda_k \bm{v}_k & \text{for } k=r+1,\dots,n.
    \end{cases}
\end{equation*}
\end{lemma}
\begin{proof}

If $Z=V=[\bm{v}_1,\dots,\bm{v}_r]$ contains the $r$ smallest orthonormal eigenvectors of $A$, $AV=VD$, where $D=\diag(\lambda_1, \dots, \lambda_r)$ is the diagonal matrix containing the $r$ smallest eigenvalues. Therefore,
\begin{equation*}
    P = I - AV(V^TAV)^{-1}V^T = I - VDD^{-1}V^T = I - VV^T
\end{equation*}
reduces to the orthogonal projector in the orthogonal complement of the smallest eigenspace. Consequently,
\begin{equation*}
    PA\bm{v}_k = \lambda_k P \bm{v}_k =
    \begin{cases}
        0 & \text{for } k=1,\dots,r, \\
        \lambda_k \bm{v}_k & \text{for } k=r+1,\dots,n.
    \end{cases}    
\end{equation*}
\end{proof}

Unfortunately, choosing eigenvectors is not a viable solution, especially not in our setting where they are so difficult to accurately compute numerically (see \Cref{rk: eigenvalue_computation}). Fortunately, choosing eigenvectors is also not necessary: we must only ensure that the columns of $Z$ (nearly) span the smallest eigenspace (see \Cref{lem: basic_properties}, \ref{prop: ker_PA}). We will later propose an explicit construction for the deflation matrix $Z$ purely based on geometrical grounds, but for the time being, we tend to other concerns. In particular, solving a singular system may seem surprising at first glance. In fact, CG remains applicable to positive semidefinite systems, provided the right-hand side is consistent (which is certainly true for \eqref{eq: projected_system}). Indeed, the method still converges to a particular solution, and the zero eigenvalues do not play any role in the convergence of CG \cite{kaasschieter1988preconditioned}. The solution of the original system $A\bm{x}=\bm{b}$ is recovered by first writing
\begin{equation*}
    \bm{x} = (I-P^T)\bm{x} + P^T\bm{x}
\end{equation*}
and reworking the expression. Firstly, since $A\bm{x}=\bm{b}$,
\begin{equation*}
    (I-P^T)\bm{x} = Z(Z^TAZ)^{-1}Z^TA\bm{x} = Z(Z^TAZ)^{-1}Z^T\bm{b}.
\end{equation*}
Secondly, since $\bm{x}$ in particular satisfies \eqref{eq: projected_system} and $\ker(PA)=R(Z)$ (see \Cref{lem: basic_properties}, \ref{prop: ker_PA}),
\begin{equation*}
    \tilde{\bm{x}} = \bm{x} + Z\bm{y}
\end{equation*}
for some $\bm{y} \in \mathbb{R}^r$ and thanks to \Cref{lem: basic_properties}, \ref{prop: PtZ}, $P^T\tilde{\bm{x}}=P^T\bm{x}$. Therefore, putting the pieces back together,
\begin{equation}
\label{eq: solution}
    \bm{x}= (I-P^T)\bm{x} + P^T\bm{x} = Z(Z^TAZ)^{-1}Z^T\bm{b} + P^T\tilde{\bm{x}}
\end{equation}
recovers the solution of the original system. However, since $PA$ is singular, the efficiency of preconditioning must be assessed through different means. The quantity that naturally emerges from the convergence analysis is the \emph{effective condition number} of $PA$, defined as
\begin{equation*}
    \keff(PA) = \frac{\lambda_n(PA)}{\lambda_{r+1}(PA)}.
\end{equation*}
The next lemma shows that the effective condition number of the deflated matrix cannot be any worse than the condition number of $A$ for any full-rank matrix $Z$.

\begin{lemma}[{\cite[Theorem 2.1]{vuik2005deflation}}]
\label{lem: cond_defl}
\begin{equation*}
    \keff(PA) \leq \kappa(A).
\end{equation*}  
\end{lemma}
\begin{proof}
First note that $PA = A-AZ(Z^TAZ)^{-1}Z^TA$ is a negative semidefinite rank $r$ correction of $A$. Therefore, thanks to a generalization of the Cauchy eigenvalue interlacing theorem \cite{thompson1976behavior},
\begin{equation*}
    \lambda_k(A) \leq \lambda_{r+k}(PA) \leq \lambda_{r+k}(A) \quad k=1,\dots,n-r.
\end{equation*}
In particular, $\lambda_{r+1}(PA) \geq \lambda_1(A)$ and $\lambda_n(PA) \leq \lambda_n(A)$. Consequently,
\begin{equation*}
    \keff(PA) = \frac{\lambda_n(PA)}{\lambda_{r+1}(PA)} \leq \frac{\lambda_n(A)}{\lambda_1(A)} = \kappa(A).
\end{equation*}
\end{proof}

Although the condition number cannot get any worse, if the deflation matrix $Z$ is not cleverly chosen, $\lambda_{r+1}(PA)$ may not be much larger than $\lambda_1(A)$. An explicit construction of the deflation matrix is provided at the end of this section. Before that, we discuss how to combine deflation with an arbitrary SPD preconditioner.

\subsection{Deflated Preconditioned Conjugate Gradient}
\label{se: dpcg}
In practice, deflation preconditioners are often combined with black-box preconditioners such as incomplete Cholesky or Jacobi \cite{vuik1999efficient,tang2008two}. In this section, we explain how to incorporate a generic SPD preconditioner $H$, which will eventually lead to the Deflated Preconditioned Conjugate Gradient (DPCG) method. Let $H=LL^T$ be the Cholesky factorization of $H$, where $L$ is lower triangular. The deflated version of the preconditioned conjugate gradient method applies the deflation strategy to the preconditioned system $\hat{A}\hat{\bm{x}}=\hat{\bm{b}}$, where $\hat{A}=L^{-1}AL^{-T}$, $\hat{\bm{x}}=L^T\bm{x}$ and $\hat{\bm{b}}=L^{-1}\bm{b}$. Thus,
\begin{align*}
    \hat{P} &= I - \hat{A}\hat{Z}(\hat{Z}^T\hat{A}\hat{Z})^{-1}\hat{Z}^T \\
      &= I - L^{-1}AZ(Z^TAZ)^{-1}Z^TL
\end{align*}
where $Z=L^{-T}\hat{Z}$. After some simplification, the deflated preconditioned system $\hat{P}\hat{A}\hat{\bm{x}}=\hat{P}\hat{\bm{b}}$ is equivalent to
\begin{equation}
\label{eq: deflated_preconditioned_system}
    L^{-1}PAL^{-T}\hat{\bm{x}}=L^{-1}P\bm{b}
\end{equation}
where, analogously to the non-preconditioned case, we defined
\begin{equation*}
    P = I - AZ(Z^TAZ)^{-1}Z^T.
\end{equation*}
\Cref{lem: cond_defl_prec} below is the preconditioned counterpart of \Cref{lem: cond_defl}.

\begin{lemma}[{\cite[Theorem 2.2]{vuik2005deflation}}]
\label{lem: cond_defl_prec}
\begin{equation*}
    \keff(H^{-1}PA) \leq \kappa(H^{-1}A).
\end{equation*}  
\end{lemma}
\begin{proof}
The proof immediately follows from applying \Cref{lem: cond_defl} to the deflated preconditioned system \eqref{eq: deflated_preconditioned_system} and realizing that $\hat{P}\hat{A}$ and $\hat{A}$ are similar to $H^{-1}PA$ and $H^{-1}A$, respectively.    
\end{proof}

Applying the standard CG method to \eqref{eq: deflated_preconditioned_system} finally leads to the deflated preconditioned conjugate gradient method summarized in \Cref{algo: dpcg}. Its derivation is detailed in \Cref{app: dpcg}. In the sequel, we will often refer to deflation as a form of preconditioning, although those are strictly speaking different components.

\begin{algorithm}[H]
\begin{algorithmic}[1]
\caption{Deflated Preconditioned Conjugate Gradient (DPCG) \cite[Algorithm 1]{vermolen2004deflation}}
\label{algo: dpcg}
\State Set $\bm{r}_0=P(\bm{b}-A\bm{x}_0)$ and $\bm{p}_0=H^{-1}\bm{r}_0$.
\For{$j=0,1,\dots$ until convergence}
    \State $\alpha_j = (\bm{r}_j,H^{-1}\bm{r}_j)/(\bm{p}_j,PA\bm{p}_j)$
    \State $\bm{x}_{j+1}=\bm{x}_j+\alpha_j\bm{p}_j$
    \State $\bm{r}_{j+1}=\bm{r}_j-\alpha_jPA\bm{p}_j$
    \State $\beta_j=(\bm{r}_{j+1},H^{-1}\bm{r}_{j+1})/(\bm{r}_j,H^{-1}\bm{r}_j)$
    \State $\bm{p}_{j+1}=H^{-1}\bm{r}_{j+1}+\beta_j \bm{p}_j$
\EndFor
\State $\bm{x}_{j+1} = Z(Z^TAZ)^{-1}Z^T\bm{b} + P^T\bm{x}_{j+1}$
\end{algorithmic}
\end{algorithm}

As usual, \Cref{algo: dpcg} terminates when the (relative) residual falls below a certain tolerance. This stopping criterion stems from classical bounds on the relative error. For CG, which minimizes the error in the $A$-norm, the following upper bound is appropriate: 
\begin{equation}
\label{eq: classical_bound}
    \frac{\|\bm{x}-\bm{x}_j\|_A}{\|\bm{x}\|_A} \leq \sqrt{\kappa(A)} \frac{\|\bm{r}_j\|_2}{\|\bm{b}\|_2}.
\end{equation}
However, the relative residual does not always provide good control over the relative error for severely ill-conditioned systems. Applying \eqref{eq: classical_bound} instead to the preconditioned system $\hat{A}\hat{\bm{x}}=\hat{\bm{b}}$ immediately yields
\begin{equation}
\label{eq: preconditioned_bound}
    \frac{\|\bm{x}-\bm{x}_j\|_A}{\|\bm{x}\|_A} \leq \sqrt{\kappa(H^{-1}A)} \frac{\|\bm{r}_j\|_{H^{-1}}}{\|\bm{b}\|_{H^{-1}}}.
\end{equation}
Since $\kappa(H^{-1}A) \ll \kappa(A)$, the relative preconditioned residual in \eqref{eq: preconditioned_bound} should provide much better control over the relative error. Due to the semidefiniteness of $PA$, accounting for deflation is less straightforward, and the proof of the following bound is deferred to \Cref{app: dpcg}:
\begin{equation}
\label{eq: dp_bound}
    \frac{\|\bm{x}-\bm{x}_j\|_{PA}}{\|\bm{x}\|_{PA}} \leq \sqrt{\keff(H^{-1}PA)} \frac{\|\bm{r}_j\|_{H^{-1}}}{\|P\bm{b}\|_{H^{-1}}}.
\end{equation}
The upper bound \eqref{eq: dp_bound} also suggests monitoring the preconditioned residual. In our numerical experiments, we have additionally computed the smallest nonzero eigenvalue of $H^{-1}PA$ and terminated the algorithm when
\begin{equation}
\label{eq: stopping_condition}
    \|\bm{r}_j\|_{H^{-1}} \leq \epsilon \sqrt{\lambda_{r+1}(H^{-1}PA)}\|P\bm{b}\|_{H^{-1}}
\end{equation}
for a specified tolerance $\epsilon$. In practice, a coarse estimate of $\lambda_{r+1}(H^{-1}PA)$ with e.g., Kaasschieter's method \cite{kaasschieter1988practical} is sufficient.

\subsection{Construction of the deflation matrix}
\label{se: Z_matrix_construction}
Finally, we come to the construction of the deflation (or projection) vectors. This step is crucial to the success of deflation-based preconditioners and is generally nontrivial (see e.g., \cite{vuik1999efficient,vermolen2004deflation} for diffusion problems in porous media with large contrasts in permeability and \cite{tang2008two} for bubbly flow problems). Fortunately, the construction is relatively straightforward in our case. Indeed, since the ill-conditioning is caused by basis functions that are only supported on trimmed elements, we can effectively construct a basis for the smallest eigenspace just from the geometry. The basis vectors will then form the columns of the deflation matrix. For this purpose, we reconsider the set of cut elements
\begin{equation*}
    \mathcal{T}_C = \{T \in \mathcal{T}_h \colon |T \cap \Omega | <  |T| \}
\end{equation*}
and its complement $\mathcal{T}_C' = \mathcal{T}_h \setminus \mathcal{T}_C$, the set of uncut elements. We then define the cut and uncut regions of the computational domain
\begin{equation*}
    \Omega_C = \bigcup_{T \in \mathcal{T}_C} T \quad \text{and} \quad \Omega_C' = \bigcup_{T \in \mathcal{T}_C'} T.
\end{equation*}
These two regions partition the basis into disjoint sets of ``strongly'' and ``weakly'' supported basis functions. The set of weakly supported basis functions
\begin{equation*}
    \Phi_C = \{\varphi \in \Phi \colon \supp_{\Omega}(\varphi) \subseteq \Omega_C\}
\end{equation*}
consists of functions that are exclusively supported in the cut region, and its complement $\Phi_C' = \Phi \setminus \Phi_C$ is the set of strongly supported basis functions. The computational domain $\Omega_h = \Omega_C \cup \Omega_C'$ is exemplarily shown in \Cref{fig: 2D_Laplace_rotating_square_active_mesh_s7} for a rotated square geometry. Its partitioning into cut and uncut regions is illustrated in \Cref{fig: 2D_Laplace_rotating_square_mesh_partitioning_s7} and serves to construct the deflation matrix $Z$. Denoting $I_C=\{i_1,\dots,i_r\}$ the set of indices for weakly supported basis functions, the deflation matrix is simply defined as
\begin{equation*}
    Z = [\bm{e}_{i_1},\dots,\bm{e}_{i_r}],
\end{equation*}
where $\bm{e}_i$ denotes the $i$th canonical basis vector of $\mathbb{R}^n$. Therefore, $Z$ is trivially full-rank, and all the ill-conditioning is now concentrated in the small coarse matrix $E=Z^TAZ$, which is factorized once and for all with a Cholesky decomposition. Thus, applying the deflation preconditioner (or projector) $P$ essentially amounts to solving small triangular systems in addition to a few sparse matrix-vector multiplications. 

Note that the rescaling does not change the support of basis functions and therefore the labeling of strongly and weakly supported basis functions. Therefore, we may repeat this construction for the Jacobi preconditioned system $\hat{A}\hat{\bm{x}}=\hat{\bm{b}}$. However, it may not be necessary in this case to consider all weakly supported basis functions. Indeed, among them, some are effectively cured by diagonal scaling alone. Hence, we present in the next section an algorithm for reducing the rank of the deflation matrix $Z$ by retaining only the truly pathological basis functions.

\begin{figure}[H]
     \centering
     \begin{subfigure}[t]{0.48\textwidth}
    \centering
    \includegraphics[width=\textwidth]{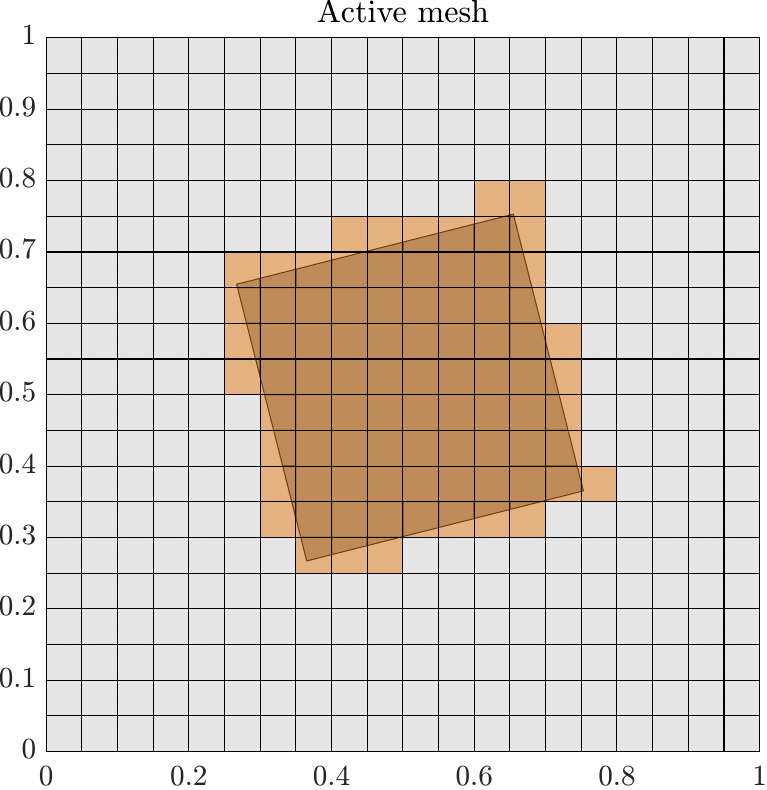}
    \caption{Computational domain $\Omega_h$ (orange) formed by all active mesh elements}
    \label{fig: 2D_Laplace_rotating_square_active_mesh_s7}
     \end{subfigure}
     \hfill
     \begin{subfigure}[t]{0.48\textwidth}
    \centering
    \includegraphics[width=\textwidth]{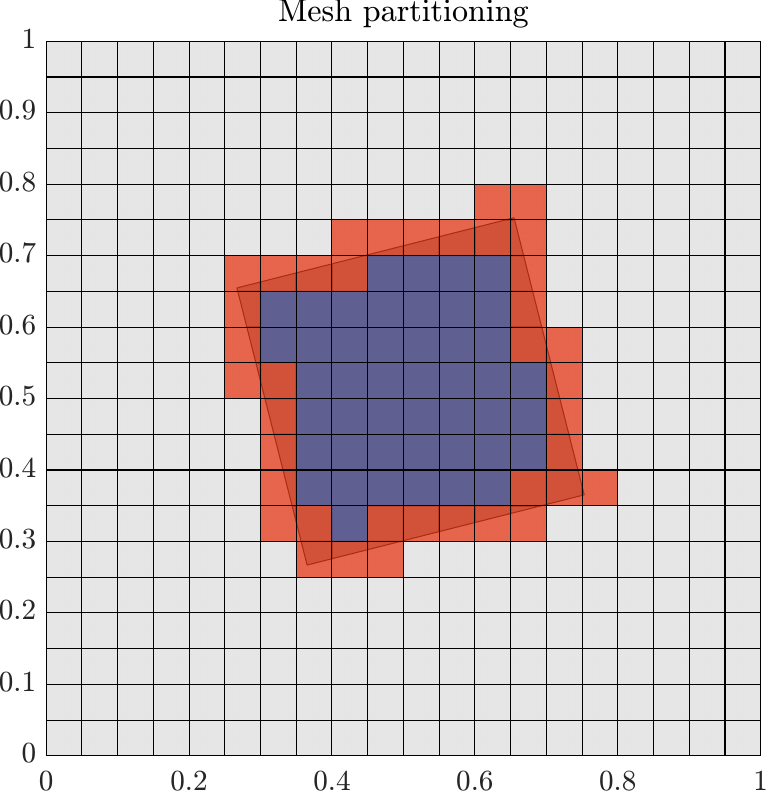}
    \caption{Partitioning of the computational domain $\Omega_h$ in cut $\Omega_C$ (red) and uncut $\Omega_C'$ (blue) regions}
    \label{fig: 2D_Laplace_rotating_square_mesh_partitioning_s7}
     \end{subfigure}
     \hfill
    \caption{Trimmed domain}
    \label{fig: 2D_Laplace_rotating_square_mesh}
\end{figure}

\begin{remark}
Since $Z$ is merely formed by columns of the identity, the deflated preconditioned matrix $\hat{P}\hat{A}$ has zero rows and columns for indices in $I_C$. One may then easily show that the deflation procedure reduces to solving $\hat{A}\hat{\bm{x}}=\hat{\bm{b}}$ with a Schur complement technique (also known as \emph{static condensation} in finite element analysis \cite{hughes2012finite,bathe2006finite}). However, the more general presentation adopted here leaves the door open to better choices of $Z$ with fewer columns. We will not dwell on it here and postpone this endeavor to future work.
\end{remark}

\subsection{Rank-reduction technique}
The deflation vectors we have considered so far correspond to functions that are only supported on trimmed elements. However, many of them are cured by diagonal scaling alone and are in fact harmless. In this section, we devise a heuristic strategy for distinguishing these basis functions from truly pathological ones that may become nearly linearly dependent. As already noted in \cite{de2017condition}, near linear dependencies often occur among basis functions with the same active support; i.e., functions such that
\begin{equation*}
    \supp_{\Omega}(\varphi_i) = \supp_{\Omega}(\varphi_j).
\end{equation*}
In practice though, this relation must be understood as an approximation rather than an equality. Therefore, we regroup functions that nearly share the same active support. In doing so, we try to exclude functions that contain on average relatively large trimmed elements in their support. Firstly, for each weakly supported basis function $\varphi_i \in \Phi_C$, we consider the cut elements within its support
\begin{equation*}
    \mathcal{T}_{i} = \{T \in \mathcal{T}_C \colon T \subseteq \supp(\varphi_i)\}.
\end{equation*}
These elements cover a local cut region
\begin{equation*}
    \Omega_{i} = \bigcup_{T \in \mathcal{T}_{i}} T.
\end{equation*}
Then, two basis functions $\varphi_i, \varphi_j \in \Phi_C$ are flagged if the following conditions are satisfied:
\begin{enumerate}
    \item $|\supp_{\Omega}(\varphi_i) \cap \supp_{\Omega}(\varphi_j)| > 0$,
    \item $|\supp_{\Omega}(\varphi_i) \cup \supp_{\Omega}(\varphi_j)| \leq \tau |\Omega_i \cup \Omega_j|$,
\end{enumerate}
for some small tolerance $\tau \in (0,1]$. In other words, the functions must intersect, and the union of their supports contains on average ``small'' trimmed elements. The precise quantification of ``small'' depends on the tolerance $\tau$. The smaller $\tau$ is, the smaller are the trimmed elements (on average) within the support of overlapping basis functions. All weakly supported basis functions that satisfy both conditions are then regrouped in a reduced set $\Phi_C^r$. The procedure is exemplarily illustrated in \Cref{fig: 2D_Laplace_ridge_corner_mesh_supp} for the house-like geometry encountered in \Cref{ex: ridge}.

\begin{figure}[H]
     \centering
     \begin{subfigure}[t]{0.48\textwidth}
    \centering
    \includegraphics[width=\textwidth]{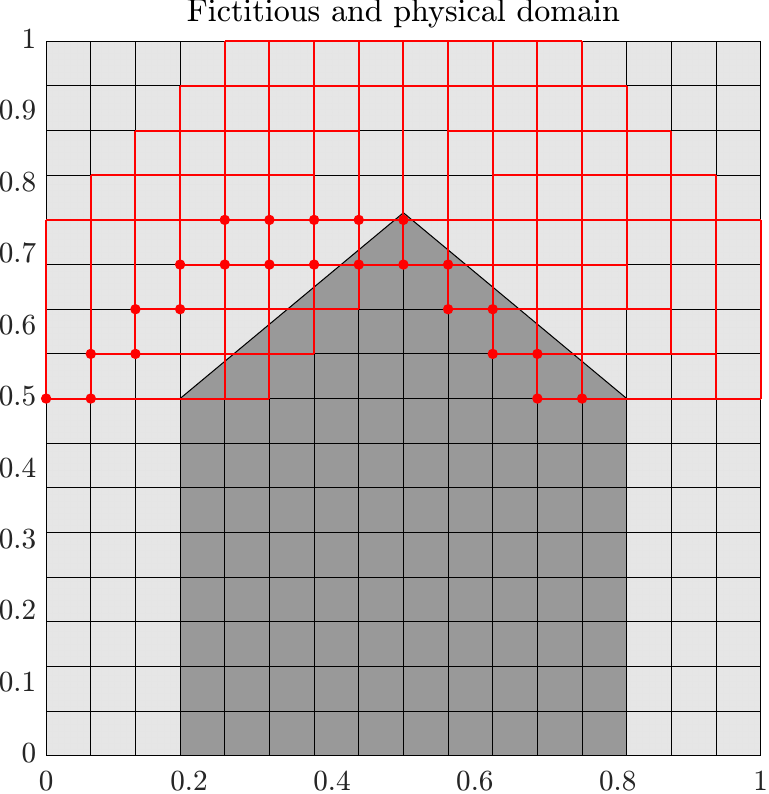}
    \caption{Full set $\Phi_C$}
    \label{fig: 2D_Laplace_ridge_corner_mesh_supp_IS}
     \end{subfigure}
     \hfill
     \begin{subfigure}[t]{0.48\textwidth}
    \centering
    \includegraphics[width=\textwidth]{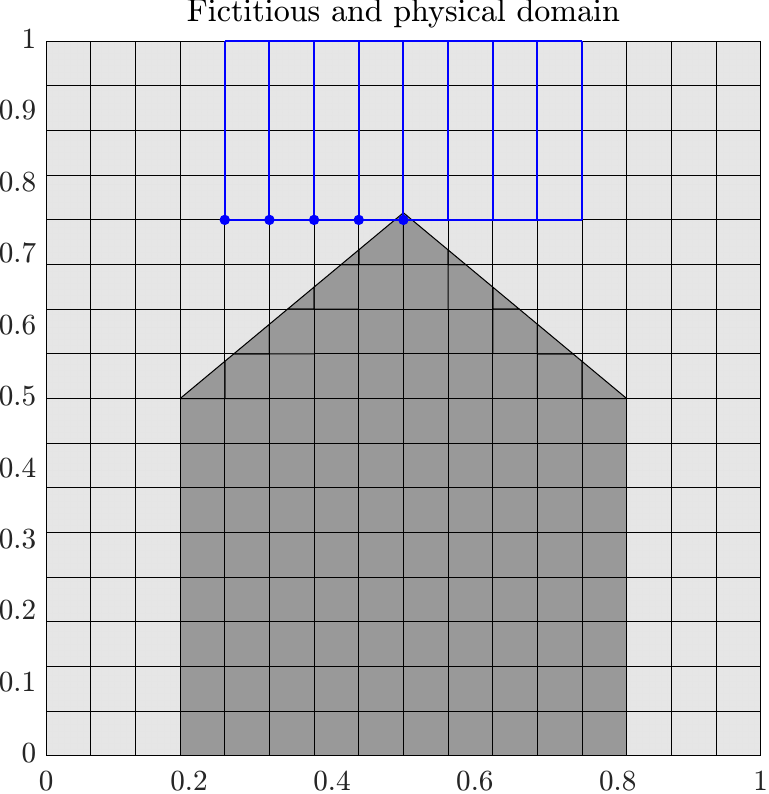}
    \caption{Reduced set $\Phi_C^r$ with $\tau = 0.1$}
    \label{fig: 2D_Laplace_ridge_corner_mesh_supp_ISr_tau_0_1}
     \end{subfigure}
     \hfill
    \caption{Support of trimmed basis functions}
    \label{fig: 2D_Laplace_ridge_corner_mesh_supp}
\end{figure}

For $\tau=0.1$, the algorithm reduces the number of problematic basis functions in \Cref{fig: 2D_Laplace_ridge_corner_mesh_supp} from $24$ to just $5$. Among those five functions, three of them have exactly the same active support, while the other two only miss one of the triangles forming the tip of the ridge. The analysis in \Cref{se: diagonal_scaling} suggests that neglecting those functions may not cause any problem in this particular setup, but in more realistic unstructured cut scenarios, it is certainly wiser to regroup them with other functions that are nearly sharing the same active support. A larger set of weakly supported basis functions is retained for larger tolerances. In our experience, $\tau=0.25$ often allowed us to eliminate at least half of all weakly supported basis functions without undermining the robustness of the preconditioner. This value was therefore adopted in all the experiments presented in the next section.

\begin{remark}
At first sight, reducing the deflation rank by only considering functions supported on small cut elements (whose volume fraction $|T \cap \Omega |/|T|$ falls below a certain tolerance) may also seem attractive \cite{jomo2019robust}. However, as we have seen in \Cref{se: diagonal_scaling}, for a fixed volume fraction, the condition number of the Jacobi preconditioned matrix scales very differently depending on the cut configuration and the polynomial order. Therefore, this criterion alone is generally insufficient. It led to significant dispersion in the results and is not recommended. The strategy outlined earlier provided much better control over the quality of the preconditioner.  
\end{remark}

While reducing the deflation rank certainly enhances the performance of the method, the size of the coarse matrix is still significantly smaller than the size of the original system, even without reduction. Moreover, since the deflation vectors are merely selected columns of the identity, the coarse matrix $E$ reduces to a small submatrix of $A$ and generally inherits some sparsity. Therefore, reduction provides an advantage but is not crucial for efficiency.

\begin{remark}
It may sometimes be possible to combine the deflation strategy with an incomplete Cholesky preconditioner instead of Jacobi \cite{vuik1999efficient}. However, the incomplete factorization without any fill-in broke down in most of our examples, and the modifications suggested in \cite{de2017condition} led to a preconditioner that did not perform any better than Jacobi. Therefore, we stuck to the standard Jacobi preconditioner in all our examples.     
\end{remark}

\section{Numerical results}
\label{se: numerical_results}
The Jacobi, SIPIC, Schwarz, and deflation preconditioners are now compared on a few selected examples, representative of a broad class of applications. In all our experiments, the physical domain $\Omega$ is embedded in a fictitious domain $\widehat{\Omega}$, discretized with a uniform B-spline or finite element mesh of mesh size $h$, degree $p$ and smoothness $0 \leq k \leq p-1$ (in the case of B-splines). Apart from the conditioning of system matrices, we also study the convergence of iterative solvers. The (deflated) preconditioned conjugate gradient algorithm is used for this purpose owing to the positive (semi-)definiteness of system matrices. The algorithm terminates as soon as the preconditioned residual meets the stopping criterion \eqref{eq: stopping_condition} for a tolerance $\epsilon$ or when the number of iterations exceeds a certain threshold, whichever happens first.

\begin{example}[$L^2$ projection]
\label{ex: l2_projection}
We first solve the $L^2$ projection problem \eqref{eq: L^2_projection} on the rotated lattice structure exemplarily shown in \Cref{fig: 2D_Laplace_rotating_lattice_geo} for an arbitrary rotation angle. This geometry was first introduced in \cite{de2019preconditioning} and we refer to the original article for its construction. The function that is projected is
\begin{equation*}
    u(x,y) = \sin(4 \pi x)\sin(4 \pi y)
\end{equation*}
and is shown in \Cref{fig: 2D_Laplace_rotating_lattice_solution}. As explained in \Cref{se: model_problems}, the $L^2$ projection problem requires solving a linear system $A\bm{u}=\bm{b}$ with the mass matrix only. We consider in this example maximally smooth B-spline discretizations of degree $2$ and $3$ on fine meshes with $70$ subdivisions in each direction, resulting in $1417$ to $1642$ degrees of freedom over the range of cases studied. We first examine the conditioning of the preconditioned matrix for various rotation angles and later compare the effectiveness of the preconditioners in speeding up the convergence of iterative solvers. Our comparison includes the newly developed deflation strategy with and without rank-reduction. In this case, our rank-reduction technique eliminates many weakly supported basis functions, commonly reducing their number by a factor $2$ to $3$. This considerable cutdown, exemplarily shown in \Cref{fig: 2D_Laplace_rotating_lattice_support_p2_n70_s9} for quadratic B-splines and a rotation angle of $\alpha=0.3307$ rad, becomes even more significant on finer meshes or in higher dimensions.

\begin{figure}[H]
     \centering
     \begin{subfigure}[t]{0.48\textwidth}
    \centering
    \includegraphics[width=\textwidth]{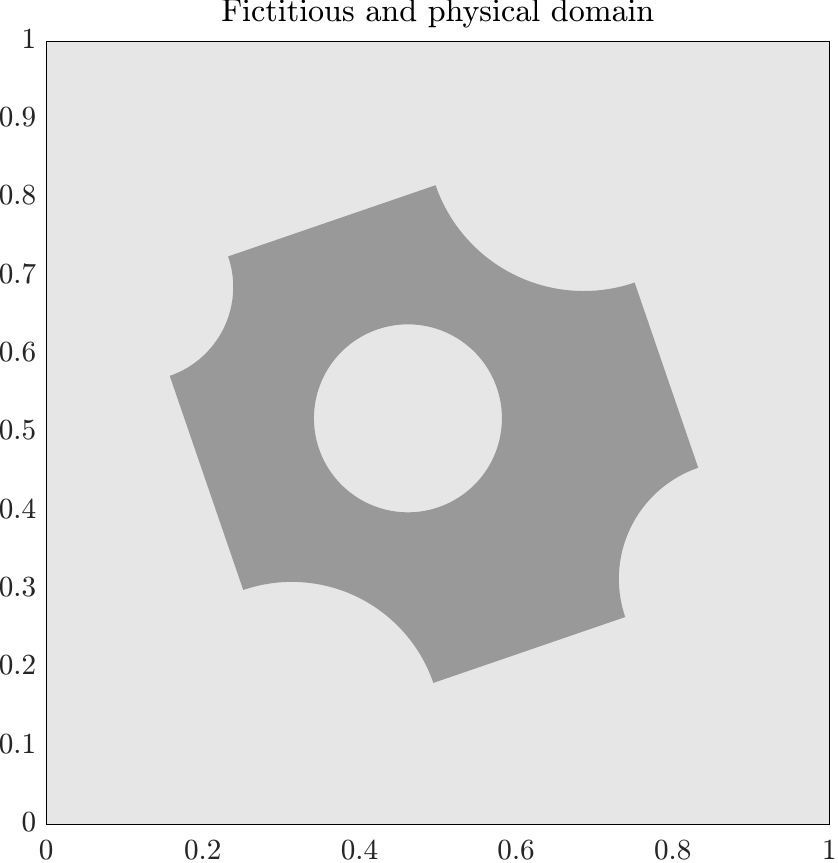}
    \caption{Fictitious domain $\widehat{\Omega}$ (light gray) and physical domain $\Omega$ (dark gray)}
    \label{fig: 2D_Laplace_rotating_lattice_geo}
     \end{subfigure}
     \hfill
     \begin{subfigure}[t]{0.48\textwidth}
    \centering
    \includegraphics[width=\textwidth]{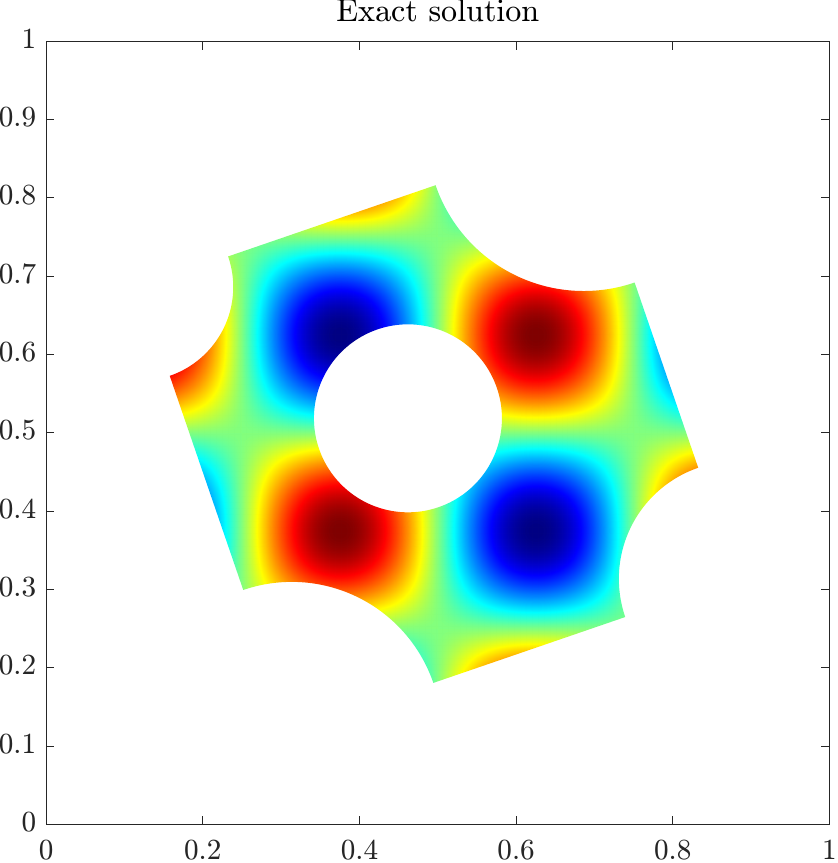}
    \caption{Solution}
    \label{fig: 2D_Laplace_rotating_lattice_solution}
     \end{subfigure}
     \hfill
    \caption{Geometry and projected function (\Cref{ex: l2_projection})}
    \label{fig: 2D_Laplace_rotating_lattice}
\end{figure}

\begin{figure}[H]
     \centering
     \begin{subfigure}[t]{0.48\textwidth}
    \centering
    \includegraphics[width=\textwidth]{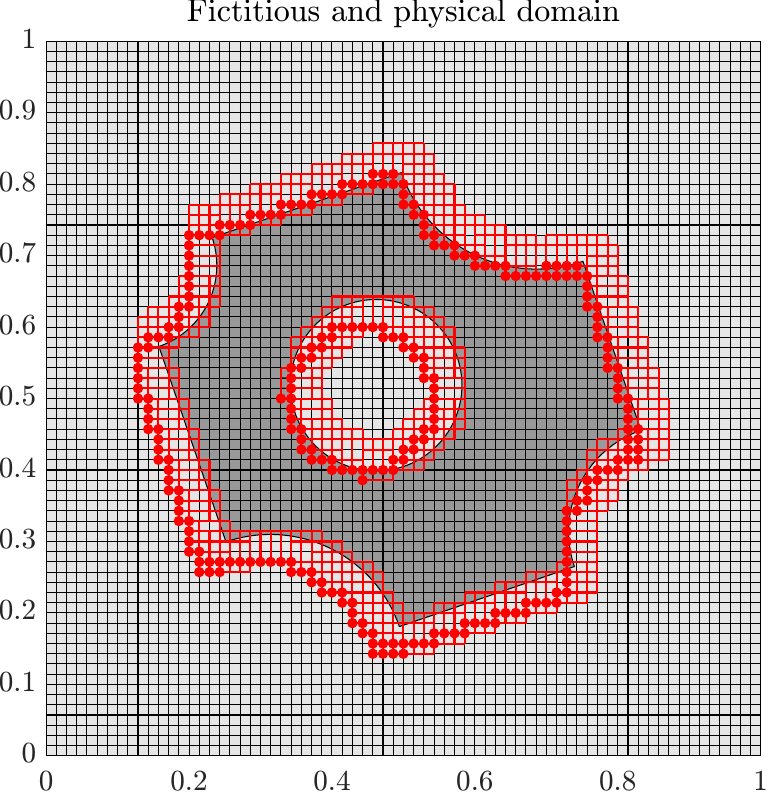}
    \caption{Full set $\Phi_C$}
    \label{fig: 2D_Laplace_rotating_lattice_supp_p2_n70_s9}
     \end{subfigure}
     \hfill
     \begin{subfigure}[t]{0.48\textwidth}
    \centering
    \includegraphics[width=\textwidth]{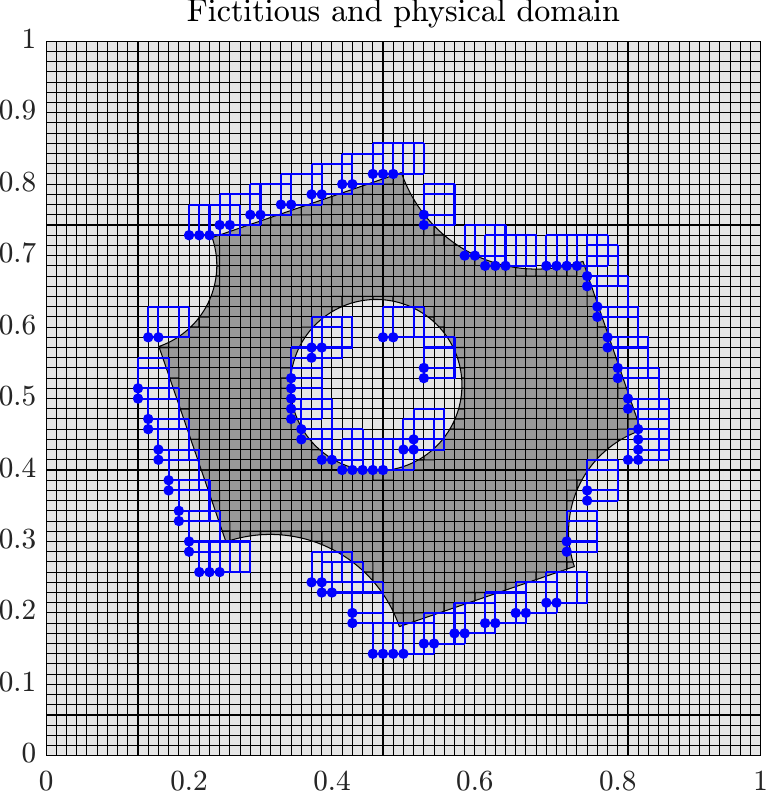}
    \caption{Reduced set $\Phi_C^r$ for $\tau=0.25$}
    \label{fig: 2D_Laplace_rotating_lattice_supp_red_p2_n70_s9_tau_0_25}
     \end{subfigure}
     \hfill
    \caption{Support of trimmed basis functions for a rotation angle of $\alpha=0.3307$ rad and $70$ subdivisions in each direction (\Cref{ex: l2_projection})}
    \label{fig: 2D_Laplace_rotating_lattice_support_p2_n70_s9}
\end{figure}

The condition number of the preconditioned systems is now computed for $100$ rotation angles uniformly spaced between $0$ and $\pi/4$. The rotation of the lattice leads to multiple cut elements and various cut configurations. The results are shown in \Cref{fig: 2D_Laplace_rotating_lattice_cond_M_Bspline_p2_Cp-1_n70,fig: 2D_Laplace_rotating_lattice_cond_M_Bspline_p3_Cp-1_n70} for degrees $p=2$ and $p=3$, respectively. The condition number of the non-preconditioned matrix grows at the expected rate of $O(\eta^{-(2p+1)})$. Although the simple Jacobi preconditioner significantly reduces the condition number, it remains highly sensitive to the cut configuration, leading to significant scatter. The SIPIC preconditioner levels it to a great extent, but not entirely. In fact, it sometimes fails to detect some quasi linear dependencies and therefore does not completely eliminate \emph{all} small eigenvalues. This is particularly clear in \Cref{fig: 2D_Laplace_rotating_lattice_eig_M_Bspline_Cp-1_n70_s9} showing the first $300$ eigenvalues of the preconditioned systems for an unfavorable rotation angle. In contrast, the condition number of the deflation and Schwarz preconditioned matrix (for the block selection strategy \eqref{eq: block_sel_cut_elem}) is completely independent of the cut configuration. Overall, our rank-reduction algorithm also performs remarkably well by eliminating most of the weakly supported basis functions while (nearly) preserving the same condition number. For $p=2$, the condition numbers are nearly identical, but for $p=3$, the Schwarz preconditioner apparently gains a slight advantage. In fact, the Schwarz preconditioner not only takes care of the trimming but also (partly) handles other sources of ill-conditioning, including the polynomial degree $p$. This property is evident from \Cref{fig: 2D_Laplace_rotating_lattice_eig_M_Bspline_p3_Cp-1_n70_s9_flat}, where a fair number of eigenvalues are detached from the rest of the spectrum, except for Schwarz. The associated eigenfunctions appeared globally supported, thereby corroborating the $p$-dependent nature of the eigenvalues. The deflation-based preconditioner was never designed with the intention of handling this source of ill-conditioning. Neither was Schwarz, but coincidentally, it still helps. Yet, its iteration count still increases significantly, and it too cannot be considered $p$-robust. In fact, none of the strategies we are aware of fully resolve the $p$-dependency, not even within multigrid methods \cite{de2020multigrid,jomo2021hierarchical}. This situation sharply contrasts with boundary-fitted discretizations of single-patch or multi-patch geometries where $p$-robust preconditioners are known \cite{gao2013kronecker,gao2014fast,loli2022easy}.

\begin{figure}[H]
     \centering
     \begin{subfigure}[t]{0.48\textwidth}
    \centering
    \includegraphics[width=\textwidth]{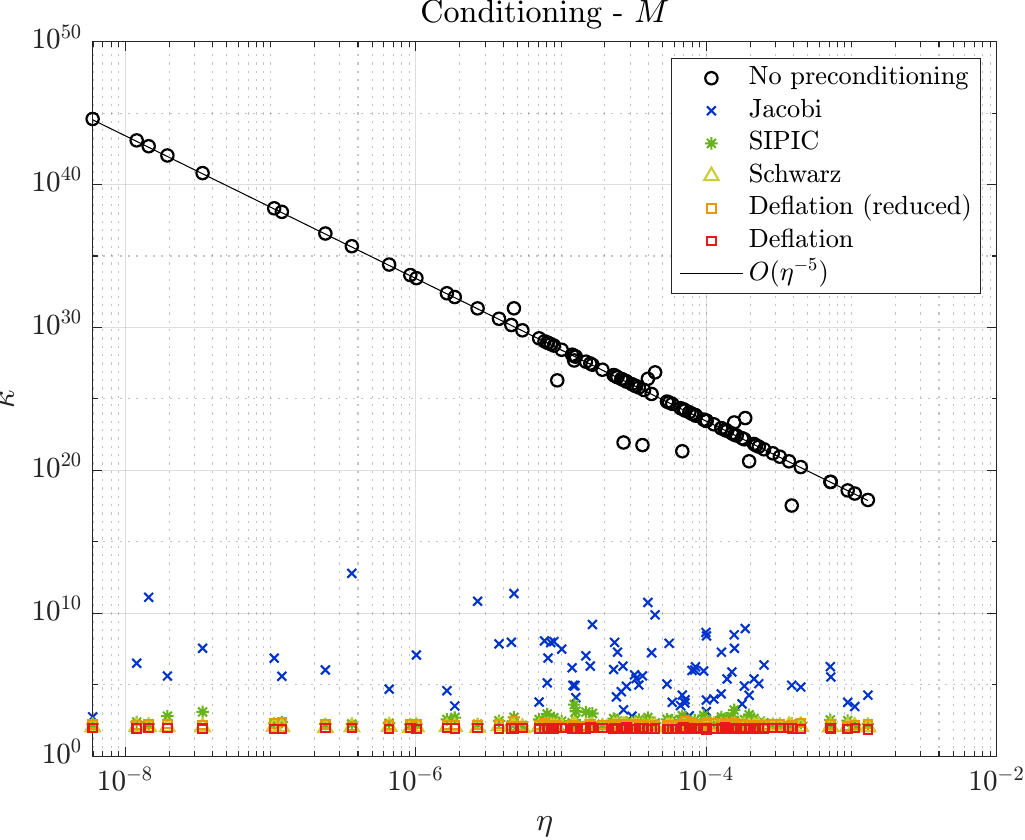}
    \caption{$p=2$}
    \label{fig: 2D_Laplace_rotating_lattice_cond_M_Bspline_p2_Cp-1_n70}
     \end{subfigure}
     \hfill
     \begin{subfigure}[t]{0.48\textwidth}
    \centering
    \includegraphics[width=\textwidth]{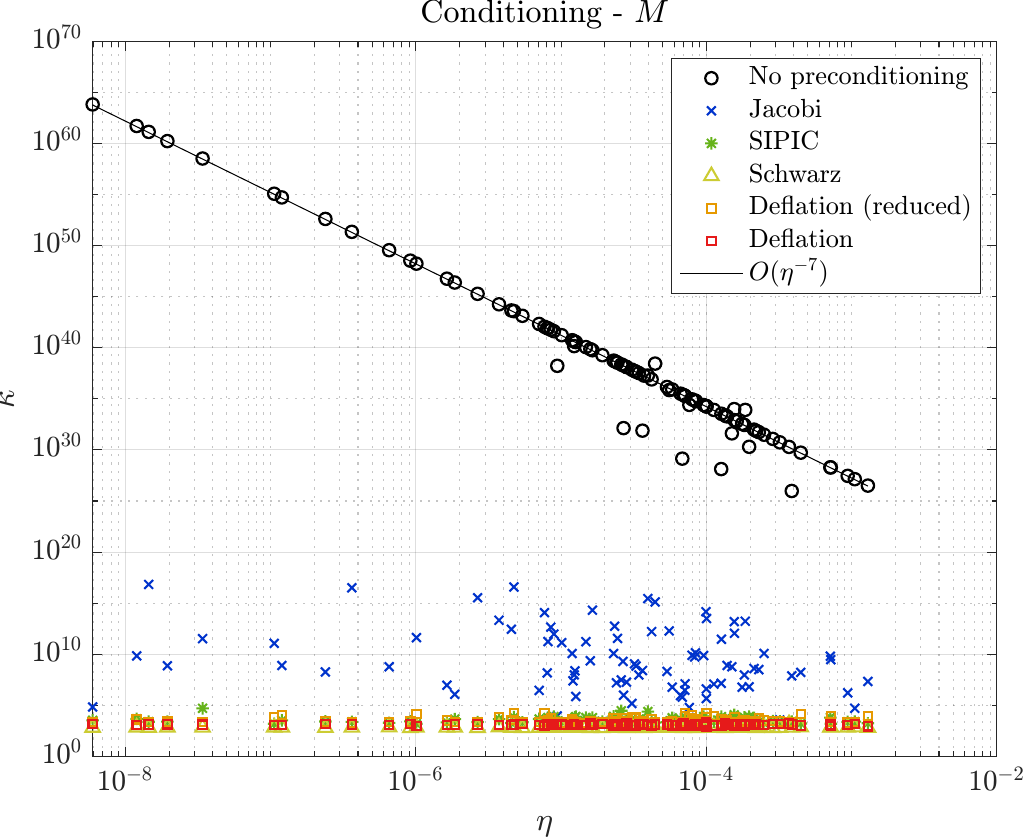}
    \caption{$p=3$}
    \label{fig: 2D_Laplace_rotating_lattice_cond_M_Bspline_p3_Cp-1_n70}
     \end{subfigure}
     \hfill
    \caption{Condition numbers for $100$ uniformly spaced rotation angles between $0$ and $\pi/4$ (\Cref{ex: l2_projection})}
    \label{fig: 2D_Laplace_rotating_lattice_cond_M_Bspline_Cp-1_n70}
\end{figure}

\begin{figure}[H]
     \centering
     \begin{subfigure}[t]{0.48\textwidth}
    \centering
    \includegraphics[width=\textwidth]{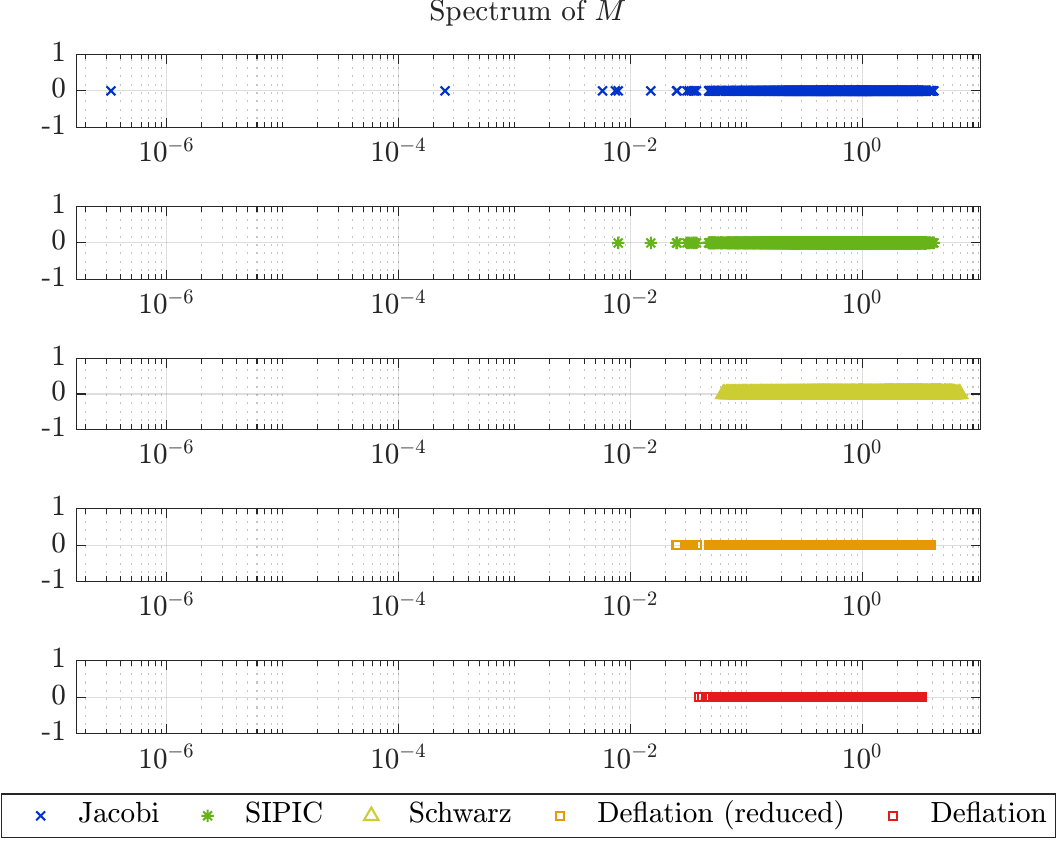}
    \caption{$p=2$}
    \label{fig: 2D_Laplace_rotating_lattice_eig_M_Bspline_p2_Cp-1_n70_s9_flat}
     \end{subfigure}
     \hfill
     \begin{subfigure}[t]{0.48\textwidth}
    \centering
    \includegraphics[width=\textwidth]{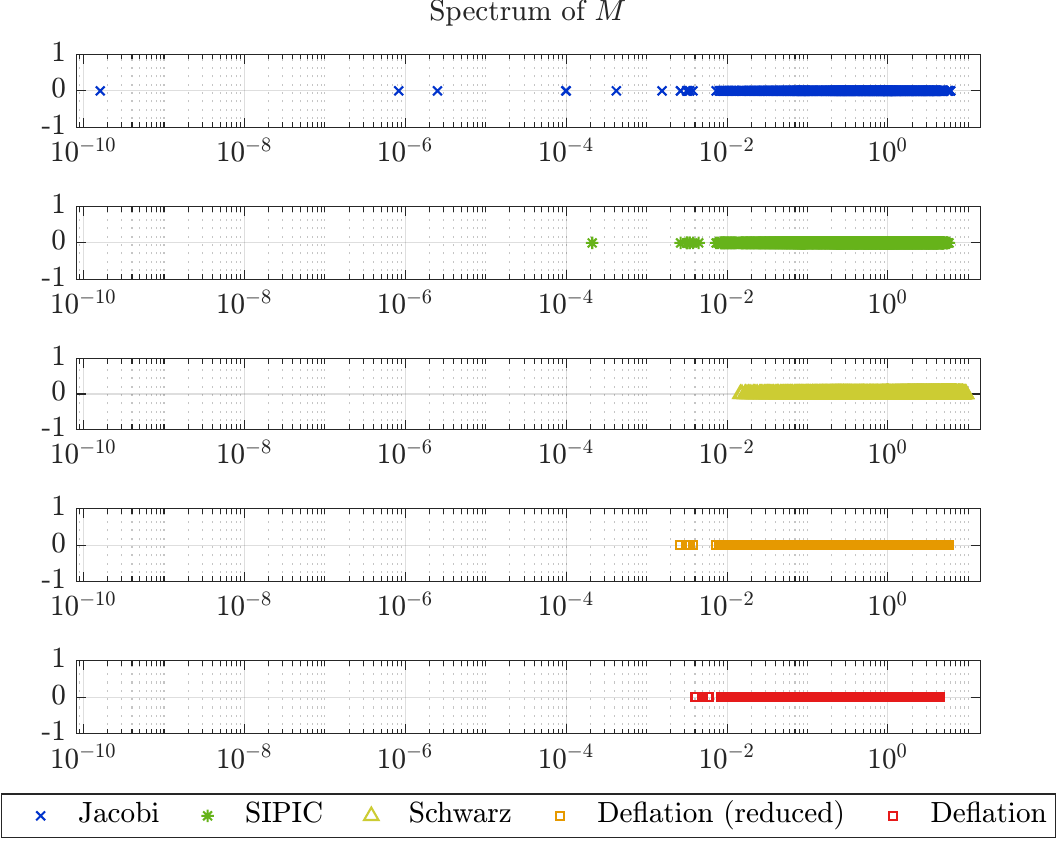}
    \caption{$p=3$}
    \label{fig: 2D_Laplace_rotating_lattice_eig_M_Bspline_p3_Cp-1_n70_s9_flat}
     \end{subfigure}
     \hfill
    \caption{First $300$ eigenvalues of the preconditioned spectrum for a rotation angle of $\alpha=0.3307$ rad (\Cref{ex: l2_projection})}
    \label{fig: 2D_Laplace_rotating_lattice_eig_M_Bspline_Cp-1_n70_s9}
\end{figure}

Now that we have studied the conditioning of the preconditioned matrix, we turn to the convergence of CG, which is monitored by computing the ``exact'' solution of the $L^2$ projection with a direct method. \Cref{fig: 2D_Laplace_rotating_lattice_conv_err_M_Bspline_Cp-1_n70} shows the convergence of the relative error in the $A$-norm, which in this case directly relates to the relative $L^2$ error
\begin{equation*}
    \frac{\|\bm{u}-\bm{u}_k\|_A}{\|\bm{u}\|_A} = \frac{\|u_h-u_{h,k}\|_{L^2}}{\|u_h\|_{L^2}},
\end{equation*}
where $u_h$ is the $L^2$ projection of $u$ and $u_{h,k} \in V_h$ is the function associated to the iterate $\bm{u}_k$. As explained in \Cref{se: dpcg}, the relative preconditioned residual, shown in \Cref{fig: 2D_Laplace_rotating_lattice_conv_prec_res_M_Bspline_Cp-1_n70}, serves to estimate the relative error. Here, the algorithm terminates when our stopping criterion \eqref{eq: stopping_condition} is met with a tolerance of $\epsilon=10^{-9}$ or when the number of iterations exceeds $1000$, whichever happens first. To avoid cluttering the figures, the errors and residuals are only shown for $20$ rotation angles uniformly spaced between $0$ and $\pi/4$. The two quantities depict analogous convergence patterns, and the trends they reveal are very clear. For $p=2$, the deflation-based preconditioner often yields the smallest iteration count, with Schwarz closely trailing behind, followed by rank-reduced deflation, SIPIC, and finally Jacobi. As expected from our condition number estimates in \Cref{se: diagonal_scaling} and \Cref{fig: 2D_Laplace_rotating_lattice_cond_M_Bspline_Cp-1_n70}, the convergence history of the Jacobi preconditioner is highly sensitive to the trimming configuration. Stagnation of the error precedes the ``discovery'' of small eigenvalues in the CG process, identified by spikes in the preconditioned residual. The same trend was already observed more than two decades ago for diffusion problems in layered rock formations \cite{vuik1999efficient,vermolen2004deflation}. However, the widely different behavior of the error or residual for increased polynomial orders signals a transition in the dominant source of ill-conditioning and aligns with earlier hints from the spectral analysis. Nevertheless, those results are not without a few surprises. Indeed, despite very similar condition numbers, the iteration counts are quite different. Most notably, a clear discrepancy appears between the iteration counts of the deflation-based preconditioners with and without rank-reduction. Although their condition numbers are almost identical, the iteration counts of the rank-reduced version are on the high end of the Schwarz preconditioner. Instead, for $p=3$, the Schwarz preconditioner converges the fastest. However, this is expected from its smaller condition number in \Cref{fig: 2D_Laplace_rotating_lattice_cond_M_Bspline_p3_Cp-1_n70}. Obviously, those advantages must be weighted together with the workload in applying the preconditioner. Applying the Schwarz preconditioner generally requires solving many small linear systems, while applying deflation-based preconditioners entails solving a single but (usually) larger system. It is impossible to draw general conclusions on the relative performance of those methods: depending on the implementation, the block selection strategy, and the geometry, either one or the other may be cheaper. 

\begin{figure}[H]
     \centering
     \begin{subfigure}[t]{0.48\textwidth}
    \centering
    \includegraphics[width=\textwidth]{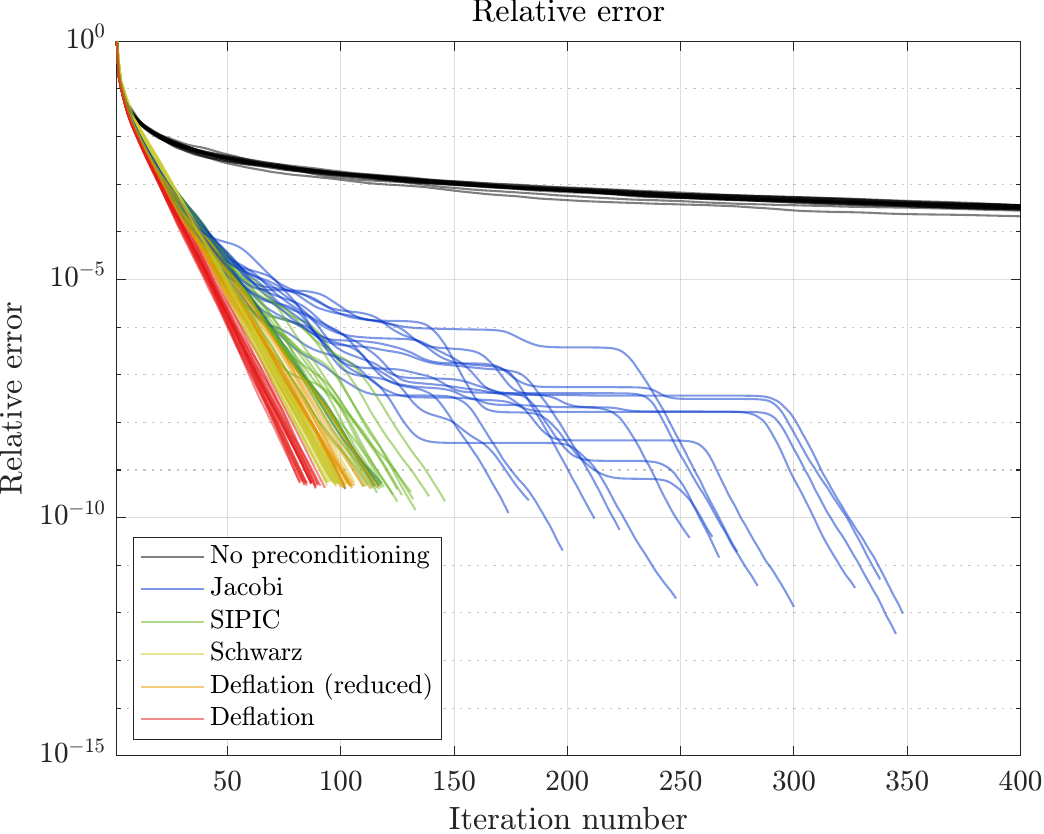}
    \caption{$p=2$}
    \label{fig: 2D_Laplace_rotating_lattice_conv_err_M_Bspline_p2_Cp-1_n70}
     \end{subfigure}
     \hfill
     \begin{subfigure}[t]{0.48\textwidth}
    \centering
    \includegraphics[width=\textwidth]{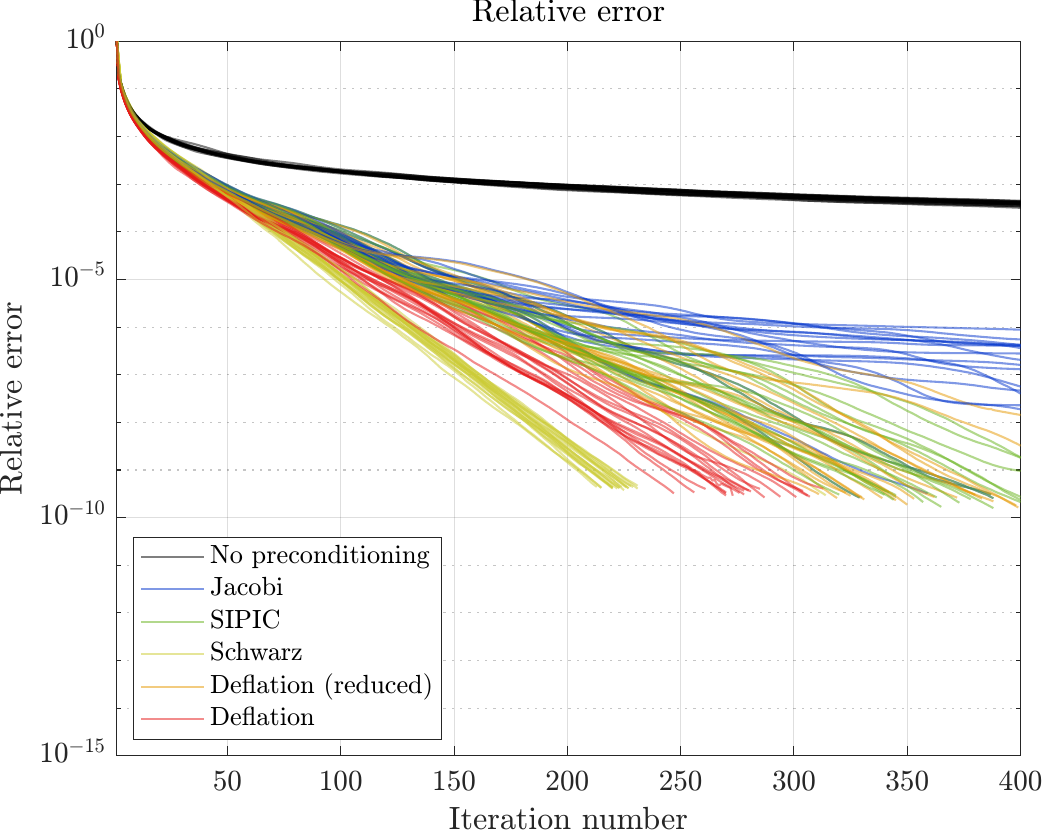}
    \caption{$p=3$}
    \label{fig: 2D_Laplace_rotating_lattice_conv_err_M_Bspline_p3_Cp-1_n70}
     \end{subfigure}
     \hfill
    \caption{Convergence of the relative error (\Cref{ex: l2_projection})}
    \label{fig: 2D_Laplace_rotating_lattice_conv_err_M_Bspline_Cp-1_n70}
\end{figure}

\begin{figure}[H]
     \centering
     \begin{subfigure}[t]{0.48\textwidth}
    \centering
    \includegraphics[width=\textwidth]{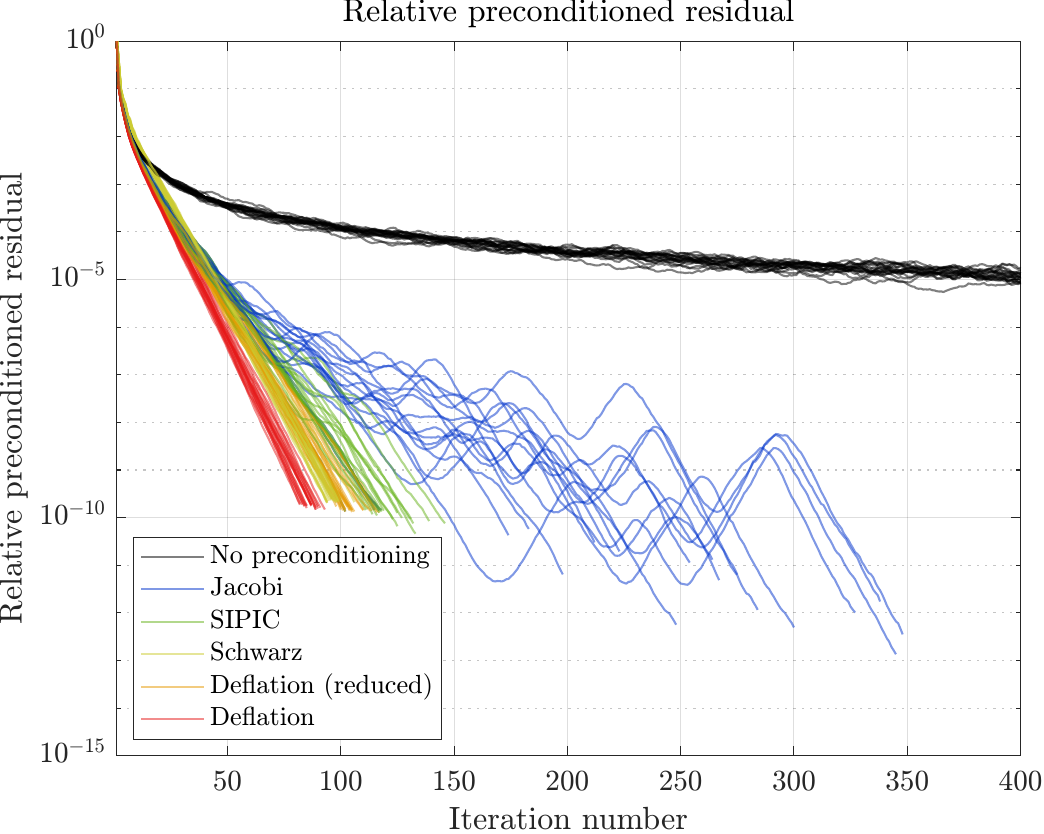}
    \caption{$p=2$}
    \label{fig: 2D_Laplace_rotating_lattice_conv_prec_res_M_Bspline_p2_Cp-1_n70}
     \end{subfigure}
     \hfill
     \begin{subfigure}[t]{0.48\textwidth}
    \centering
    \includegraphics[width=\textwidth]{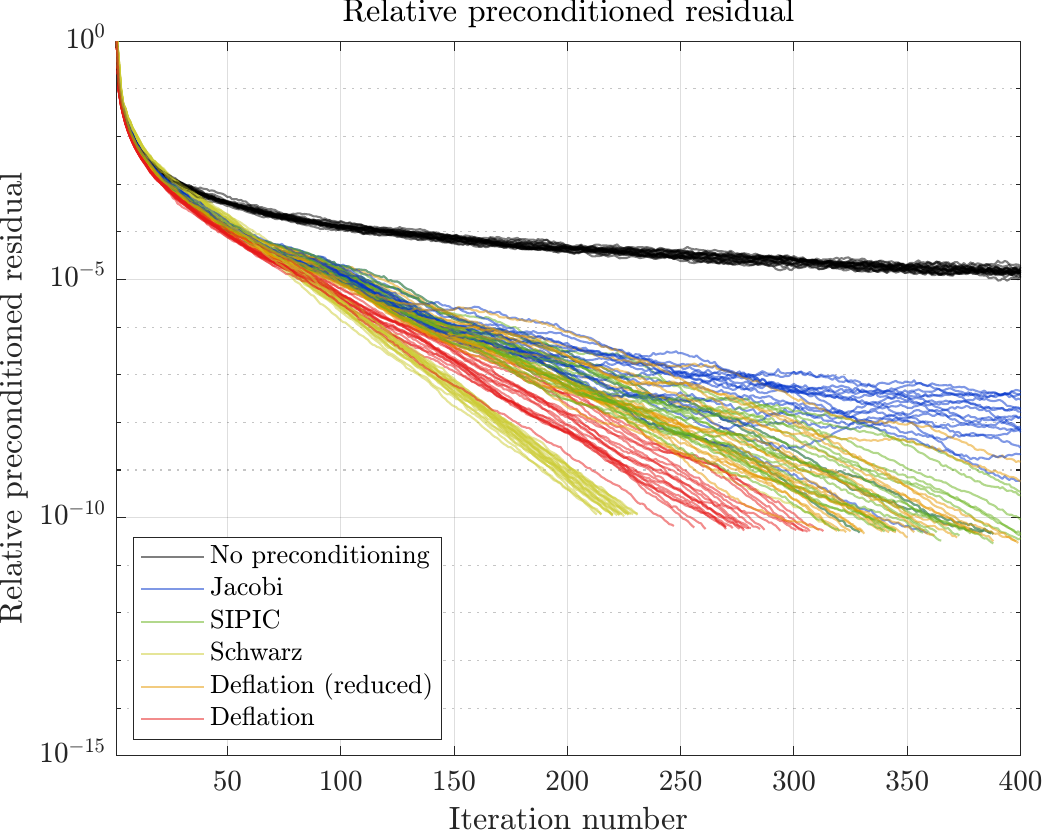}
    \caption{$p=3$}
    \label{fig: 2D_Laplace_rotating_lattice_conv_prec_res_M_Bspline_p3_Cp-1_n70}
     \end{subfigure}
     \hfill
    \caption{Convergence of the relative preconditioned residual (\Cref{ex: l2_projection})}
    \label{fig: 2D_Laplace_rotating_lattice_conv_prec_res_M_Bspline_Cp-1_n70}
\end{figure}
\end{example}

\begin{example}[Poisson problem]
\label{ex: poisson_problem}
For this second example, we solve the standard Poisson problem \eqref{eq: discrete_weak_form_elliptic} on a square plate with a cut-out in the middle as shown in \Cref{fig: 2D_Laplace_extrusion_mesh}. The cut-out consists of two circular arcs of radius $r=\sqrt{5}h-\delta$ centered in $\bm{c}_1=(0.5,0.25)$ and $\bm{c}_2=(0.5,0.75)$ and connected by vertical line segments. A uniform mesh of classical $C^0$ quadratic Lagrange finite elements of mesh size $h=1/56$ is laid in the background, thereby reproducing along the circular arcs the problematic trimmed configuration encountered in \Cref{ex: square_with_hole}. The finite element discretization leads to a system of size $12146$ and for studying the convergence of the iterative solvers, we consider the manufactured solution
\begin{equation*}
    u(x,y) = x(1-x)\sin(3\pi x)^2\sin(\pi y)
\end{equation*}
shown in \Cref{fig: 2D_Laplace_extrusion_exact_sol}. It satisfies homogeneous Dirichlet boundary conditions on the left and right edges, and we prescribe Neumann boundary conditions on all other edges, including the internal ones.

\begin{figure}[H]
     \centering
     \begin{subfigure}[t]{0.48\textwidth}
    \centering
    \includegraphics[width=\textwidth]{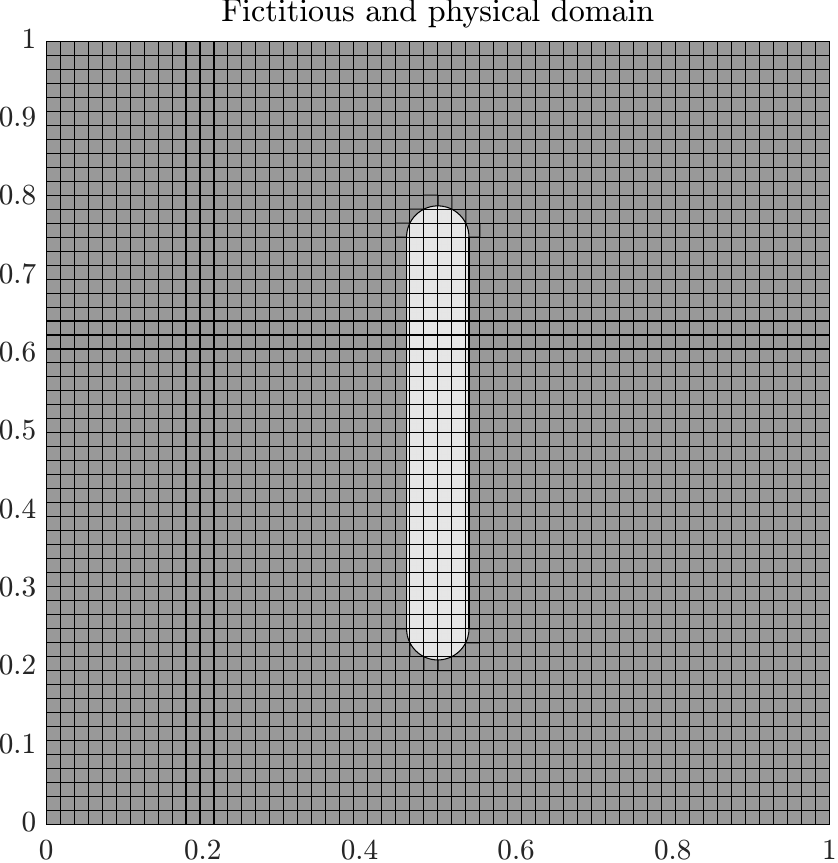}
    \caption{Fictitious domain $\widehat{\Omega}$ (light gray) and physical domain $\Omega$ (dark gray)}
    \label{fig: 2D_Laplace_extrusion_mesh}
     \end{subfigure}
     \hfill
     \begin{subfigure}[t]{0.48\textwidth}
    \centering
    \includegraphics[width=\textwidth]{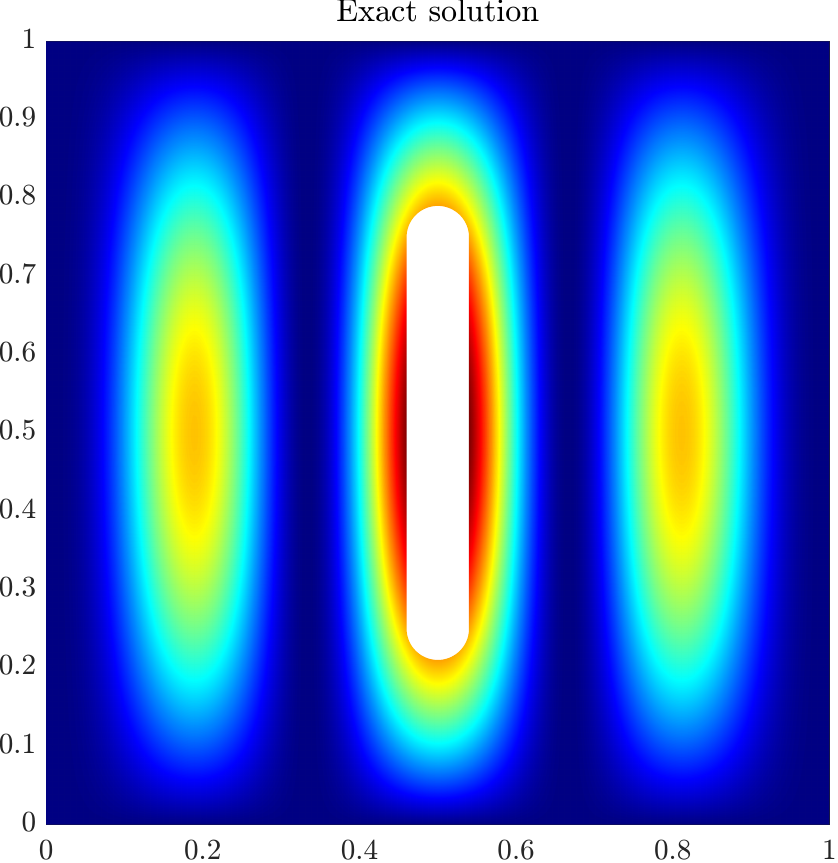}
    \caption{Solution}
    \label{fig: 2D_Laplace_extrusion_exact_sol}
     \end{subfigure}
     \hfill
    \caption{Geometry and manufactured solution (\Cref{ex: poisson_problem})}
    \label{fig: 2D_Laplace_extrusion}
\end{figure}

This example compares the robustness of the Jacobi, SIPIC, Schwarz, and deflation preconditioners in a realistic setting. The index blocks of the Schwarz preconditioner regroup the indices of basis functions overlapping on cut elements \cite{de2019preconditioning} and the local inverses are approximately formed by discarding eigenvalues below $10^{-14}$ \cite{jomo2019robust}. Similarly to \Cref{ex: square_with_hole} and as shown in \Cref{fig: 2D_Laplace_extrusion_cond_K_Lagrange_p2}, the condition number of the preconditioned stiffness matrix increases as $\delta \to 0$, except for the deflation strategy. Unfortunately, the extensive overlap of basis functions in the $C^0$ case prevents eliminating any of them with the rank-reduction technique. Thus, the size of the coarse matrix in this case $(292)$ is larger but still relatively small in comparison to the size of the system $(12146)$. The differences in condition number are reflected through the preconditioned spectrum, partly shown in \Cref{fig: 2D_Laplace_extrusion_eig_K_Lagrange_p2_s20}. While the SIPIC and Schwarz preconditioners successfully eliminate the smallest eigenvalue, two small eigenvalues remain and also converge to zero, although at a milder rate. This explains the slower growth of the condition number in \Cref{fig: 2D_Laplace_extrusion_cond_K_Lagrange_p2}. In contrast, our deflation technique eliminates \emph{all} trimming-related eigenvalues. The remaining small eigenvalues are dependent on other discretization parameters and must be taken care of with $h,p$-robust preconditioners.

\begin{figure}[H]
     \centering
     \begin{subfigure}[t]{0.48\textwidth}
    \centering
    \includegraphics[width=\textwidth]{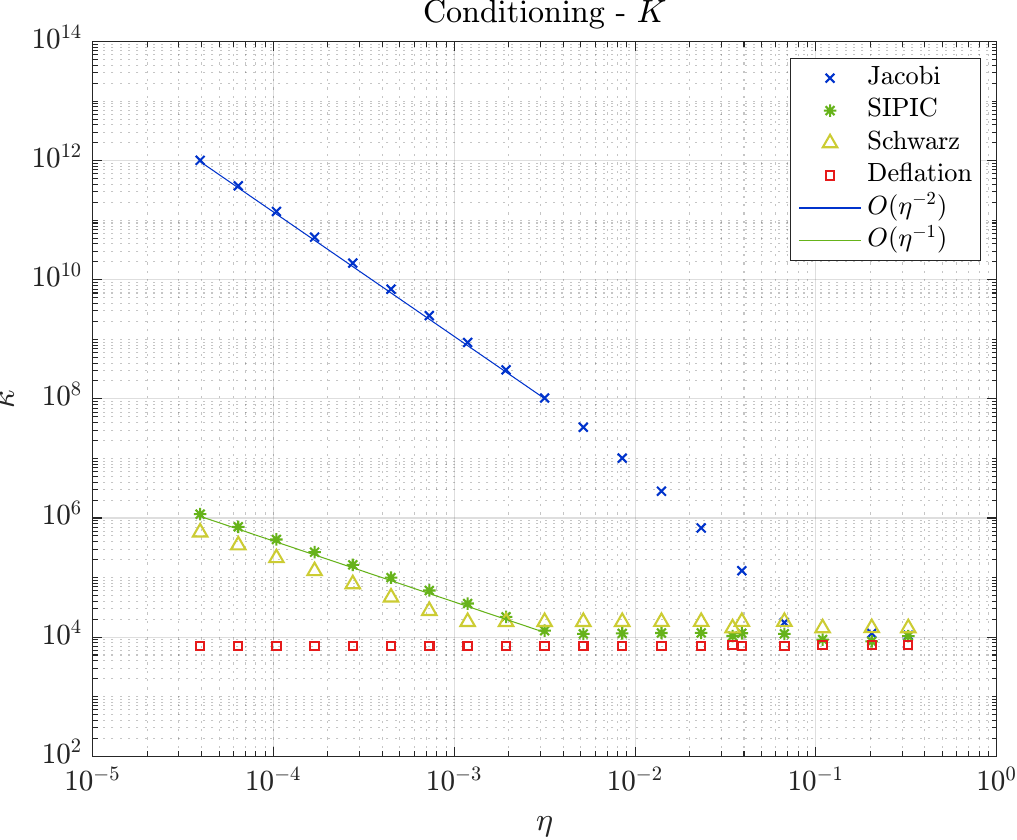}
    \caption{Condition numbers for $20$ logarithmically spaced values of $\delta$ between $10^{-4}$ and $10^{-2}$}
    \label{fig: 2D_Laplace_extrusion_cond_K_Lagrange_p2}
     \end{subfigure}
     \hfill
     \begin{subfigure}[t]{0.48\textwidth}
    \centering
    \includegraphics[width=0.9\textwidth]{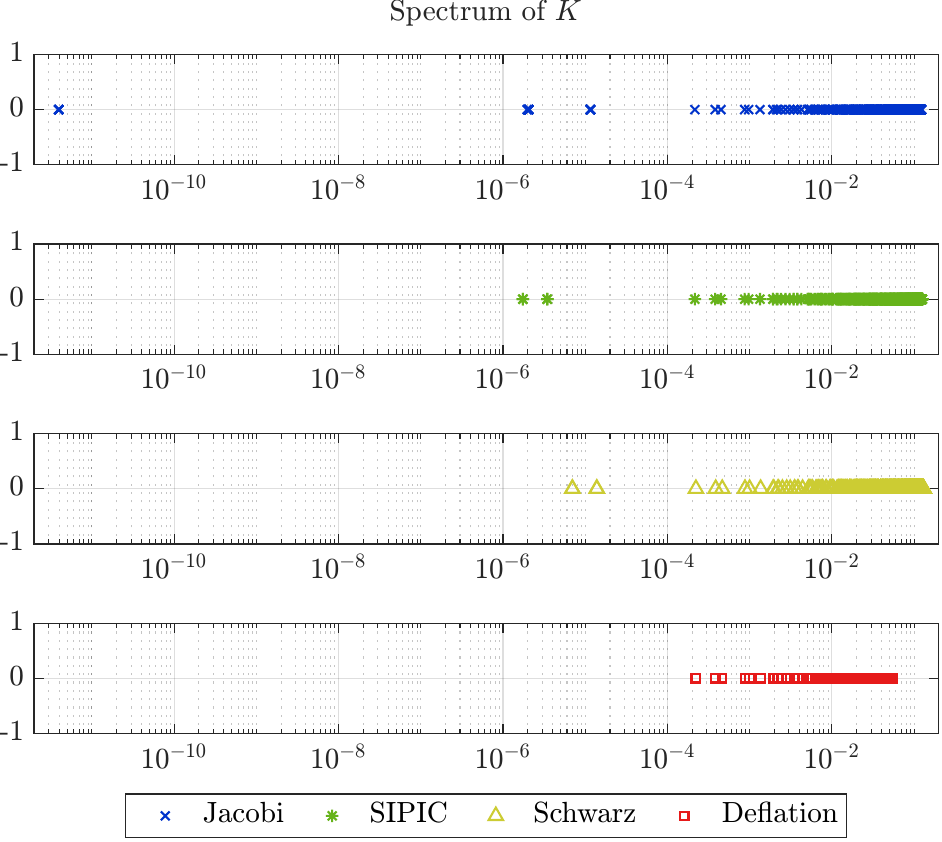}
    \caption{First $500$ eigenvalues of the preconditioned spectrum for $\delta=10^{-4}$}
    \label{fig: 2D_Laplace_extrusion_eig_K_Lagrange_p2_s20}
     \end{subfigure}
     \hfill
    \caption{Condition numbers and eigenvalues of the preconditioned system matrices (\Cref{ex: poisson_problem})}
    \label{fig: 2D_Laplace_extrusion_eig_K_Lagrange_p2}
\end{figure}

An increase of the condition number is expected to translate into an increase of the iteration count. This is confirmed in \Cref{fig: 2D_Laplace_extrusion_conv_err_K_Lagrange_p2,fig: 2D_Laplace_extrusion_conv_prec_res_K_Lagrange_p2} showing the convergence of the relative error and relative preconditioned residuals, respectively. Note that the relative error in the $A$-norm reduces in this case to the relative error in the $H^1$ seminorm:
\begin{equation*}
    \frac{\|\bm{u}-\bm{u}_k\|_A}{\|\bm{u}\|_A} = \frac{|u_h-u_{h,k}|_{H^1}}{|u_h|_{H^1}},
\end{equation*}
where the ``exact'' discrete solution of the Poisson problem, represented by the coefficient vector $\bm{u}$, was computed with a direct solver. Similarly to the previous example, we set a tolerance of $\epsilon=10^{-9}$ and prematurely terminate the algorithm if the iteration count exceeds $10000$. Clearly, the increased condition number delays the convergence of the relative error, which decays in a stepwise manner. Surprisingly, the iteration counts for the Schwarz preconditioner are generally larger than those of SIPIC, despite the smaller condition number. The preconditioned residual behaves more sporadically but follows overall a similar trend. In contrast, the convergence of the relative error or preconditioned residual for the deflation preconditioner is almost insensitive to $\delta$. However, the deflation preconditioner might eventually run into trouble if the condition number of the Jacobi preconditioned matrix becomes excessively large (e.g., greater than $10^{15}$). Indeed, similarly to other preconditioning techniques, the deflation strategy is built on top of Jacobi, and if the condition number for the latter becomes too large, it may compromise the accuracy of the solution to the coarse system. As a matter of fact, early signs of instabilities are already visible in \Cref{fig: 2D_Laplace_extrusion_conv_err_K_Lagrange_p2} and appear for a condition number of roughly $10^{12}$.

\begin{figure}[H]
     \centering
     \begin{subfigure}[t]{0.48\textwidth}
    \centering
    \includegraphics[width=\textwidth]{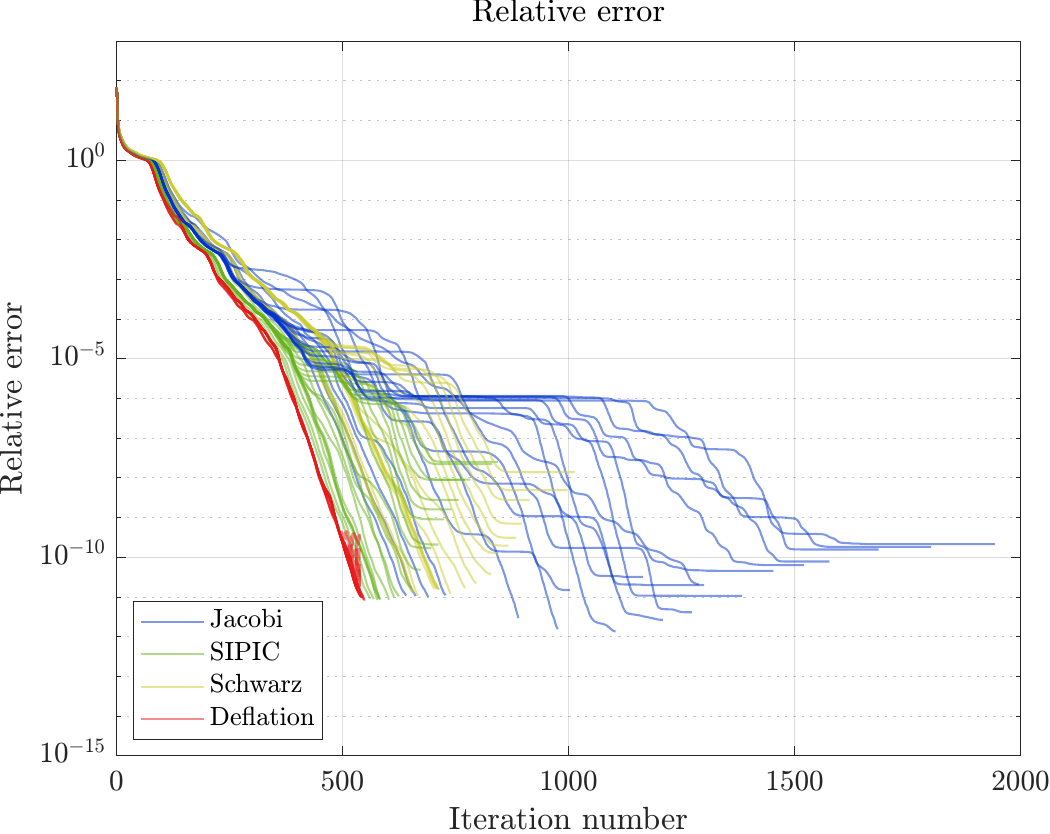}
    \caption{Relative error}
    \label{fig: 2D_Laplace_extrusion_conv_err_K_Lagrange_p2}
     \end{subfigure}
     \hfill
     \begin{subfigure}[t]{0.48\textwidth}
    \centering
    \includegraphics[width=\textwidth]{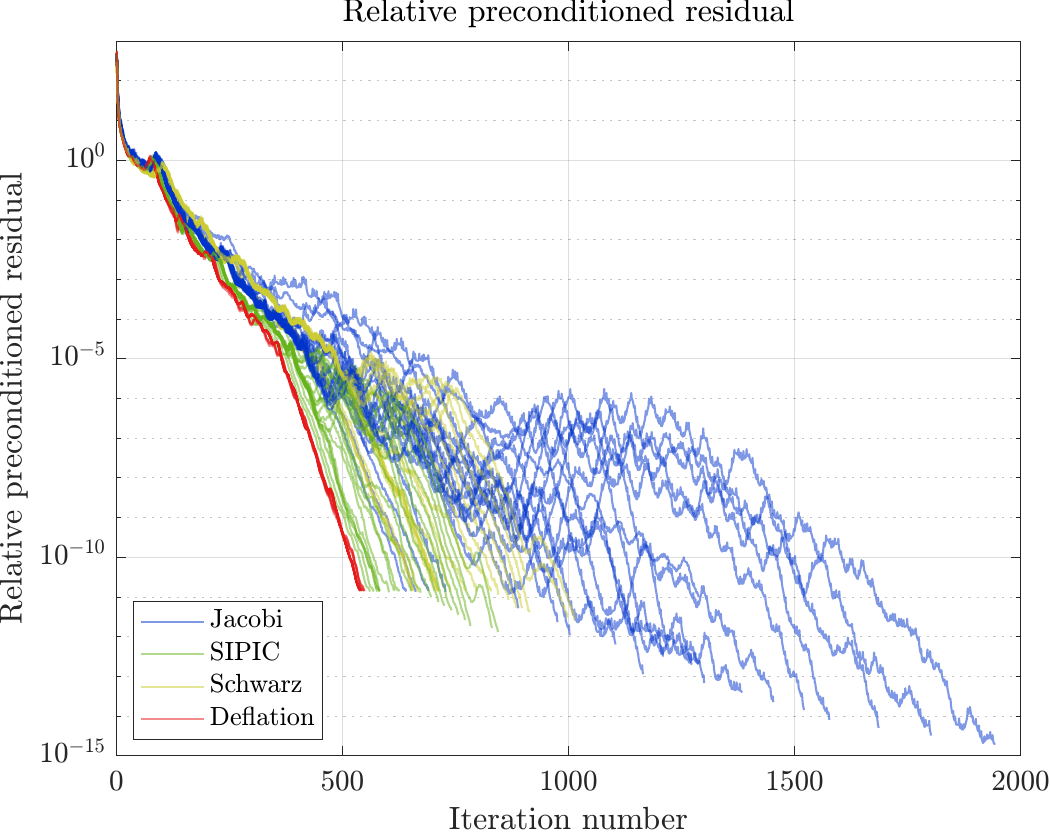}
    \caption{Relative preconditioned residual}
    \label{fig: 2D_Laplace_extrusion_conv_prec_res_K_Lagrange_p2}
     \end{subfigure}
     \hfill
    \caption{Convergence of the relative error and preconditioned residuals (\Cref{ex: poisson_problem})}
    \label{fig: 2D_Laplace_extrusion_conv_err_prec_res_K_Lagrange_p2}
\end{figure}
\end{example}

\begin{remark}
Deflation techniques are sometimes stabilized by shifting undesirable eigenvalues elsewhere in the spectrum \cite{tang2009comparison,dwarka2022scalable}. However, the implementation and analysis of those so-called adapted (or modified) deflation techniques are much more intricate, and given that the instabilities only arise for excessively large condition numbers, they should be taken care of with simpler techniques; e.g., by removing basis functions \cite{elfverson2018cutiga,verhoosel2015image}.
\end{remark}

\begin{example}[Implicit time stepping]
\label{ex: implicit_time_stepping}
In this last example, we explore the performance of the preconditioners for a wave propagation problem, described by the standard wave equation
\begin{align}
\partial_{tt} u - \Delta u &=f & &\text{in } \Omega \times (0,T], \label{eq: wave_equation} \\
 \partial_n u &=0 & &\text{on } \partial \Omega \times (0,T], \nonumber \\
 u(0)&=u_0 & &\text{in } \Omega,  \nonumber \\
 \partial_t u(0)&=v_0 & &\text{in } \Omega, \nonumber
\end{align}
where $u_0$ and $v_0$ are two initial conditions, and we assume homogeneous Neumann boundary conditions for simplicity. Finite element discretizations of \eqref{eq: wave_equation} commonly lead to a large coupled system of ordinary differential equations (ODEs) \cite{hughes2012finite,quarteroni2009numerical}
\begin{align}
\label{eq: semi_discrete_pb}
\begin{split}
M\ddot{\bm{u}}(t) + K\bm{u}(t) &= \bm{f}(t) \qquad \text{for } t \in [0,T], \\
\bm{u}(0) &= \bm{u}_0,\\
\dot{\bm{u}}(0) &= \bm{v}_0.
\end{split}
\end{align} 
Computing an approximate solution with implicit time stepping schemes requires solving a linear system $A\bm{x}=\bm{b}$ at each time step, where $A = a_1M + a_2K$ is a linear combination of the mass and stiffness matrices. For instance, for the Trapezoidal rule (an implicit unconditionally stable Newmark method), $A=M+\beta \Delta t^2 K$, where $\Delta t$ is the step size and $\beta=1/4$ is a parameter of the method. In contrast, with explicit time stepping, the system matrix reduces to the mass matrix, which is also the foundation for mass lumping techniques \cite{voet2023mathematical,stoter2023critical}.

This example focuses on implicit time stepping and reproduces the steps necessary in simulating the propagation of a wave in an indented spiky structure shown in \Cref{fig: 2D_Laplace_waveguide_mesh_n32_48}, originally inspired from a perforated waveguide in \cite[Fig. 13]{eisentrager2024eigenvalue}. The spacing of the spikes mimics the pathological configuration of \Cref{ex: 3ridges} constructed for maximally smooth cubic B-splines and generalizes it to multiple aligned spikes. We consider those same discretization parameters in this example. The tip of a spike resembles a square of side length $\delta$ centered within an element (\Cref{fig: 2D_Laplace_waveguide_mesh_n32_48_detail}), thereby imitating \ref{conf: middle_cut}.

\begin{figure}[H]
     \centering
     \begin{subfigure}[t]{0.48\textwidth}
    \centering
    \includegraphics[width=\textwidth]{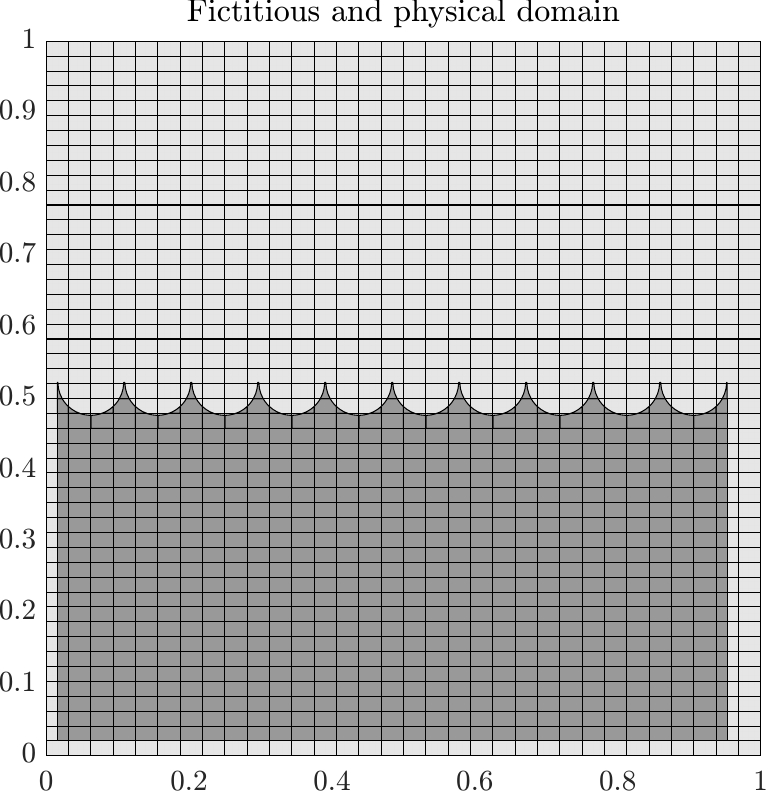}
    \caption{Indented structure}
    \label{fig: 2D_Laplace_waveguide_mesh_n32_48}
     \end{subfigure}
     \hfill
     \begin{subfigure}[t]{0.48\textwidth}
    \centering
    \includegraphics[width=\textwidth]{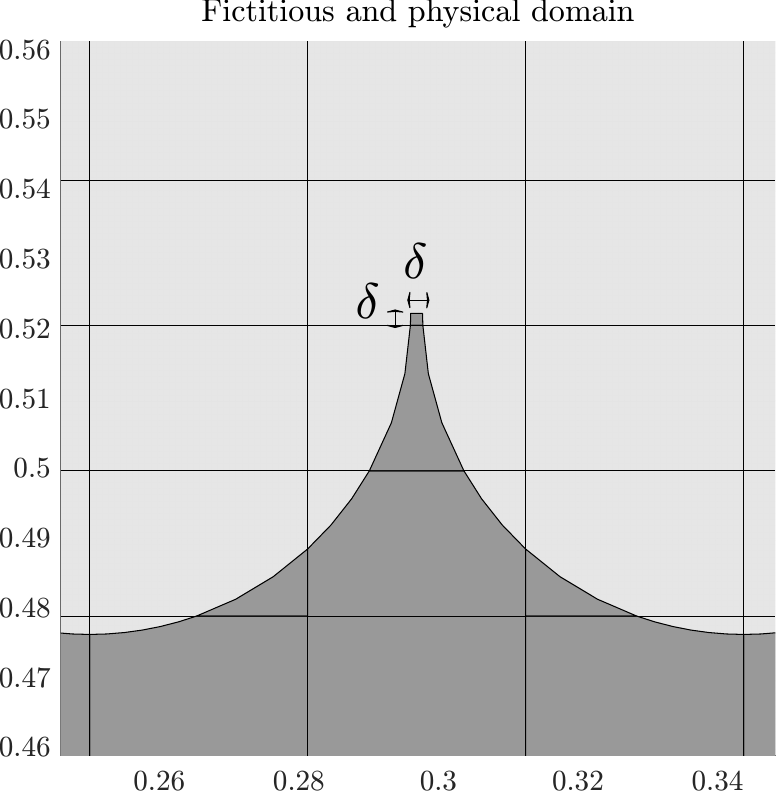}
    \caption{Detail of a ridge}
    \label{fig: 2D_Laplace_waveguide_mesh_n32_48_detail}
     \end{subfigure}
     \hfill
    \caption{Waveguide-inspired indented structure (\Cref{ex: implicit_time_stepping})}
    \label{fig: 2D_Laplace_waveguide_mesh}
\end{figure}

For simulating a propagating wave, we set $f=0$ and prescribe as initial data $u_0(x,y)=\mathrm{e}^{-(\frac{x-x_c}{\sigma})^2}$ and $v_0=0$ with $x_c=0.4844$ and $\sigma=0.1$. This creates a vertical wave front initially centered at $x_c$ and propagating in two opposite directions toward the boundaries. The simulation ends at time $T=0.4$, when the wave reaches the boundaries. Solution snapshots computed with the Trapezoidal rule and $600$ time steps are shown in \Cref{fig: 2D_Laplace_waveguide_snapshots_n32_48}.

\begin{figure}[H]
    \centering
    \includegraphics[width=\textwidth]{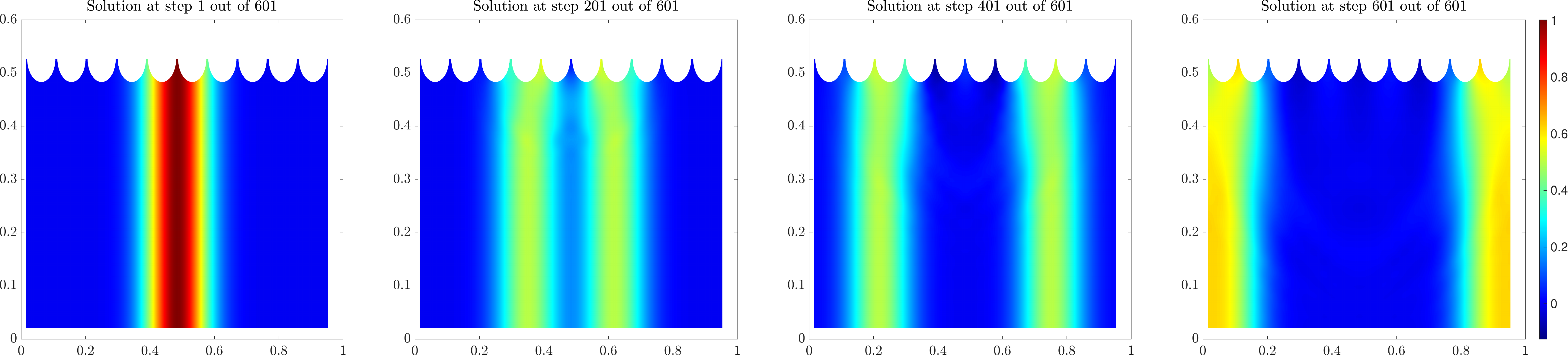}
    \caption{Solution snapshots (\Cref{ex: implicit_time_stepping})}
    \label{fig: 2D_Laplace_waveguide_snapshots_n32_48}
\end{figure}

We now turn to the iterative solution of the linear systems arising in the time stepping process for increasingly small values of $\delta$. For the mesh size shown in \Cref{fig: 2D_Laplace_waveguide_mesh_n32_48}, the linear systems have $952$ unknowns. For the Schwarz preconditioner, we choose the advanced block selection strategy \eqref{eq: block_sel_supp_int} based on intersecting supports, which performed best in \Cref{ex: 3ridges}. As shown in \Cref{fig: 2D_Laplace_waveguide_cond_Bspline_p3_Cp-1_n32_48}, the condition number of the Jacobi and SIPIC preconditioned systems increases at the same rate as $\delta \to 0$, although SIPIC is better by a constant factor. The Schwarz preconditioner together with the advanced block selection strategy \eqref{eq: block_sel_supp_int} reduces the growth of the condition number to a linear rate. In contrast, the condition number for the deflation preconditioned matrix remains independent of $\delta$. The same holds true with the rank-reduction technique, although the condition number is offset by a constant.
The preconditioned spectra for $\delta=10^{-3.5}$ are shown in \Cref{fig: 2D_Laplace_waveguide_eig_A_Bspline_p3_Cp-1_n32_48_s20}. One of the striking differences with previous examples is the relatively large number of small eigenvalues for the Jacobi preconditioned matrix. Their number is most likely related to the number of spikes and to the discretization parameters. Although the SIPIC preconditioner achieves a better clustering of the eigenvalues, they remain relatively small, and SIPIC only reduces the condition number by a constant factor. The Schwarz preconditioner with the advanced block selection strategy \eqref{eq: block_sel_supp_int} is much more effective at reducing the condition number, but it still increases at a milder linear rate. This deficiency is apparently caused by a single small eigenvalue around $10^{-4}$ in \Cref{fig: 2D_Laplace_waveguide_eig_A_Bspline_p3_Cp-1_n32_48_s20}.

\begin{figure}[H]
     \centering
     \begin{subfigure}[t]{0.48\textwidth}
    \centering
    \includegraphics[width=\textwidth]{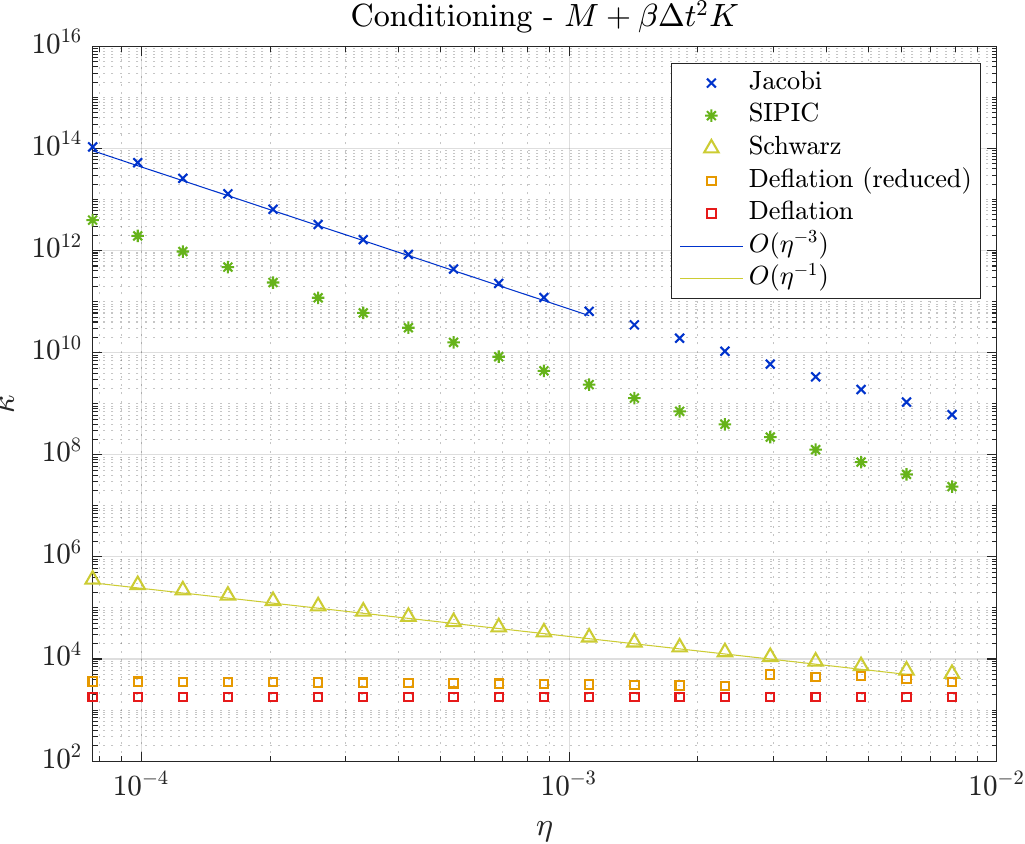}
    \caption{Condition numbers for $20$ logarithmically spaced values of $\delta$ between $10^{-3.5}$ and $10^{-2.5}$}
    \label{fig: 2D_Laplace_waveguide_cond_Bspline_p3_Cp-1_n32_48}
     \end{subfigure}
     \hfill
     \begin{subfigure}[t]{0.48\textwidth}
    \centering
    \includegraphics[width=\textwidth]{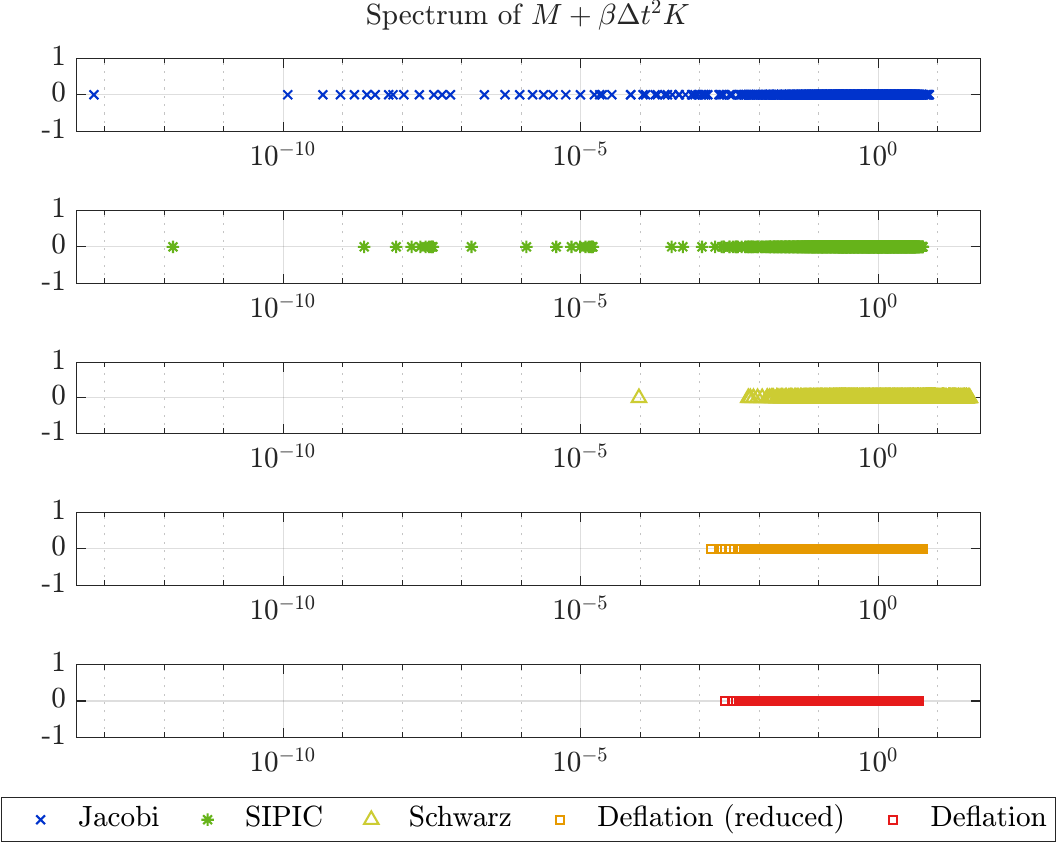}
    \caption{Preconditioned spectrum for $\delta=10^{-3.5}$}
    \label{fig: 2D_Laplace_waveguide_eig_A_Bspline_p3_Cp-1_n32_48_s20}
     \end{subfigure}
     \hfill
    \caption{Condition number and preconditioned spectrum (\Cref{ex: implicit_time_stepping})}
    \label{fig: 2D_Laplace_waveguide_eig_A_Bspline_p3_Cp-1_n32_48}
\end{figure}

The existence of many small eigenvalues in the case of the Jacobi and SIPIC preconditioners is already a red flag for the convergence of iterative solvers. This is unfortunately confirmed in the first time step of the Trapezoidal rule, which requires solving a linear system $A\bm{x}=\bm{b}$ where the right-hand side vector depends on the initial conditions in \eqref{eq: semi_discrete_pb}. For comparison, the ``exact'' solution is computed with a direct solver after rescaling the system, while the approximate solutions are computed for a tolerance of $\epsilon=10^{-9}$. The convergence of the relative error in the $A$-norm is shown in \Cref{fig: 2D_Laplace_waveguide_conv_err_A_Bspline_p3_Cp-1_n32_48} along with the relative preconditioned residual in \Cref{fig: 2D_Laplace_waveguide_conv_prec_res_A_Bspline_p3_Cp-1_n32_48}. Both quantities roughly convey the same message: with a deflation-based preconditioner, the solution converges almost independently of $\delta$. When reducing the deflation rank, the convergence follows two distinct patterns. They are due to a small adjustment of the deflation rank signaled by a slight change in the condition number in \Cref{fig: 2D_Laplace_waveguide_cond_Bspline_p3_Cp-1_n32_48}. Nevertheless, for a fixed rank, convergence takes place almost independently of $\delta$. On the contrary, the number of iterations for the Schwarz preconditioner clearly tends to increase, as expected from the growth of the condition number. The occasional flattening of the relative error in \Cref{fig: 2D_Laplace_waveguide_conv_err_A_Bspline_p3_Cp-1_n32_48} may be due to inaccuracies in the ``exact'' solution that serves to measure the error. Indeed, the ill-conditioning of a matrix also reduces the accuracy that is achievable with a direct solver \cite{higham2002accuracy}. The simultaneous flattening of the error for Schwarz and deflation further supports this explanation. The other two preconditioners perform very poorly, and the solution even fails to converge after $1000$ iterations, although CG would have terminated in exact arithmetic. This behavior is most likely attributed to the large number of small eigenvalues in the preconditioned systems.

\begin{figure}[H]
     \centering
     \begin{subfigure}[t]{0.48\textwidth}
    \centering
    \includegraphics[width=\textwidth]{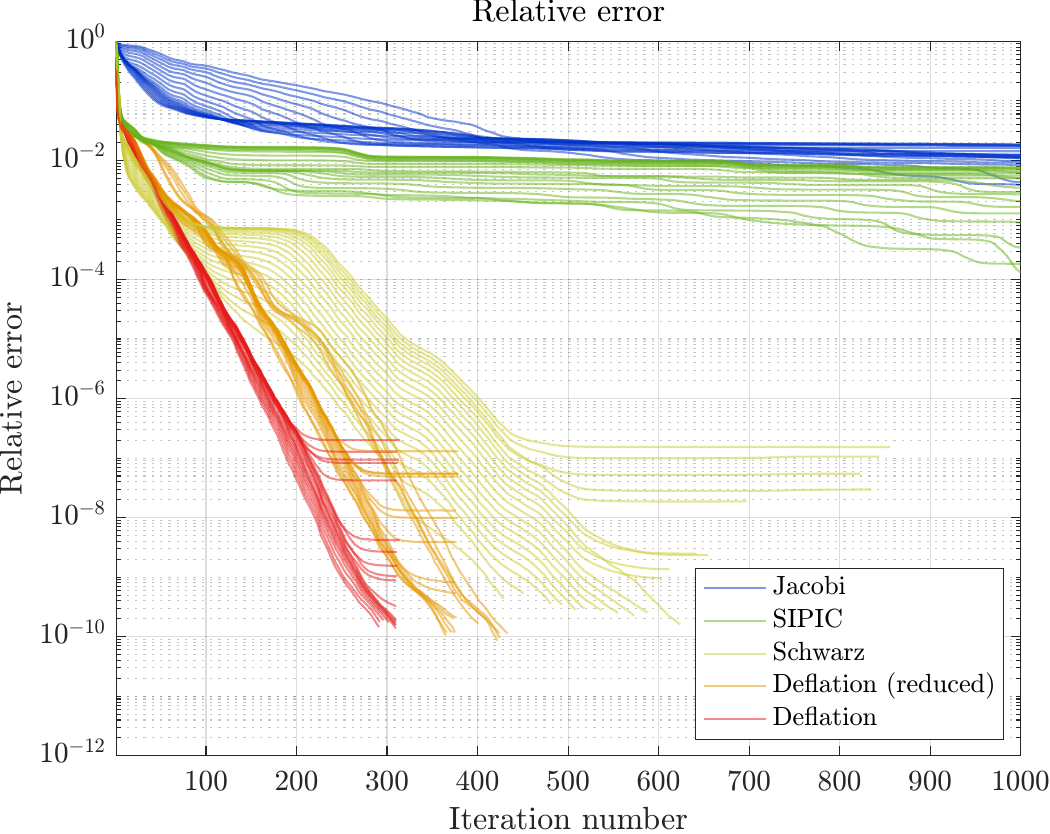}
    \caption{Relative error}
    \label{fig: 2D_Laplace_waveguide_conv_err_A_Bspline_p3_Cp-1_n32_48}
     \end{subfigure}
     \hfill
     \begin{subfigure}[t]{0.48\textwidth}
    \centering
    \includegraphics[width=\textwidth]{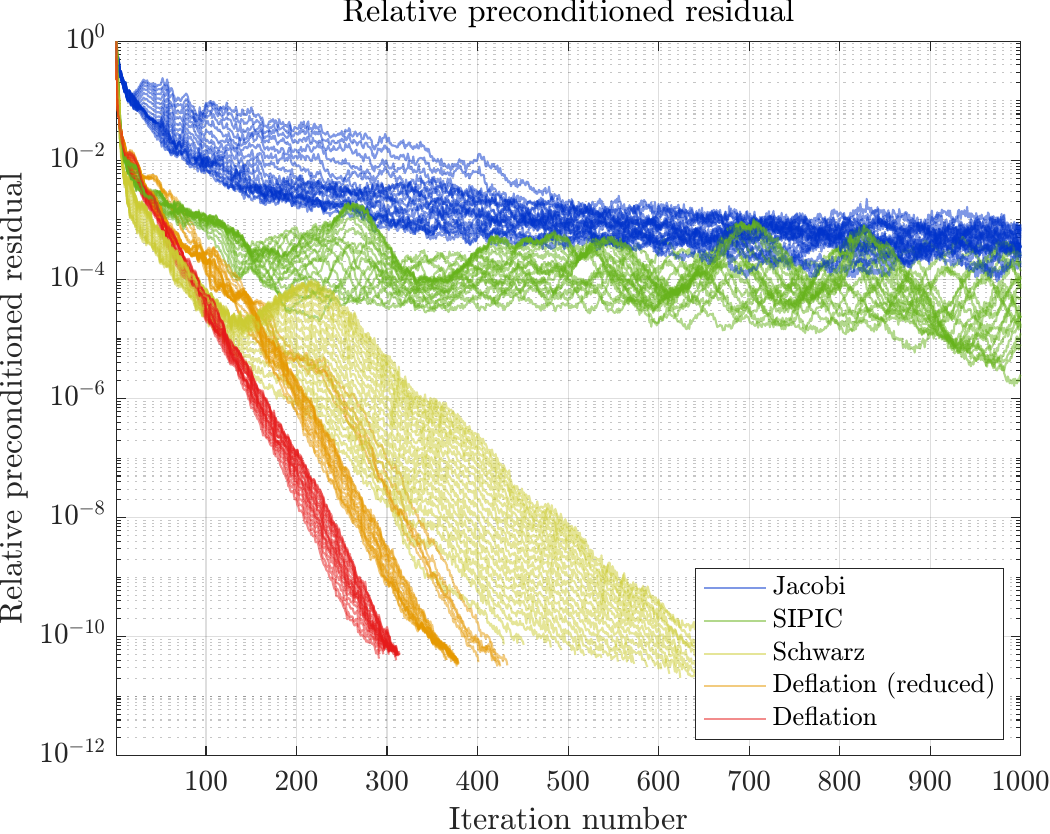}
    \caption{Relative preconditioned residual}
    \label{fig: 2D_Laplace_waveguide_conv_prec_res_A_Bspline_p3_Cp-1_n32_48}
     \end{subfigure}
     \hfill
    \caption{Convergence of the relative error and preconditioned residuals (\Cref{ex: implicit_time_stepping})}
    \label{fig: 2D_Laplace_waveguide_conv_Bspline_p3_Cp-1_n32_48}
\end{figure}
\end{example}

\section{Conclusion}
\label{se: conclusion}
In immersed finite element methods, small cut elements cause severely ill-conditioned system matrices requiring dedicated penalization, stabilization, or preconditioning techniques. In this article, we have revisited multiple preconditioning strategies for symmetric positive definite systems, as they commonly arise for elliptic PDEs but also for parabolic or hyperbolic ones in conjunction with time stepping schemes. In particular, our analysis reveals that the effectiveness of Jacobi preconditioning is way more complicated than earlier literature suggested. Even for a classical $C^0$ finite element discretization of a trimmed 1D line segment, its robustness critically depends on the choice of basis, contrary to common intuition. The analysis of Jacobi preconditioning becomes significantly more challenging in higher dimensions, where the cut configuration and the position of knots or interpolation points also play a role. 

The scatter in the effectiveness of Jacobi preconditioning had already led to various improvements in the form of Symmetric Incomplete Permuted Inverse Cholesky (SIPIC) and Schwarz preconditioners. Unfortunately, we have shown through a series of counter-examples that none of them are completely robust. The simple and almost realistic geometries on which our counter-examples were built prompted the development of a new deflation-based preconditioner in combination with Jacobi preconditioning. Deflation techniques are ideally suited to immersed finite element discretizations, where the ill-conditioning typically stems from a few small eigenvalues originating from basis functions that are only supported on small trimmed elements. Contrary to competing techniques, deflation treats those functions globally, in a single ``block'', by defining a deflation subspace that nearly coincides with the smallest eigenspace. This property is key to its success, as testified in the experiments at the end of the article. In fact, since most functions are already cured by diagonal scaling, we later suggested only retaining truly pathological basis functions based on near overlapping supports. Numerical experiments have confirmed the effectiveness of this rank-reduction technique in preserving the condition number. However, this strategy is less effective at maintaining the iteration count. There certainly exist better rank-reduction techniques, but this problem is left as future work.

In any case, whether with or without rank-reduction, deflation only treats the ill-conditioning originating from small trimmed elements. Combining it with $h,p$-robust (multigrid) preconditioners might become necessary in practice, but this is a whole different research topic. In future work, we plan to extend our deflation-based preconditioner to nonsymmetric or indefinite systems arising from mixed finite element discretizations.

\section*{Acknowledgments}
This research was mostly carried out while the first author was visiting the Numerical Analysis group at TU Delft during fall 2025. The hospitality and support of the group are kindly acknowledged. The authors would also like to thank Frits de Prenter for insightful discussions that inspired many of the counter-examples presented in this work.
Pablo Antolin acknowledges the financial support of the Swiss National Science Foundation through
the project FLAS$_h$ (200021\_214987).

\section*{Code availability}
The code for reproducing the numerical experiments will be available prior to publication.

\begin{appendices}
\section{Additional examples}
\label{se: additional_examples}

\begin{example}[Non-uniform grid]
\label{ex: perturbed_grid}
In order to assess the robustness of the scaling relations in \Cref{tab: scaling_jacobi}, we consider a slight variant of \Cref{ex: ridge}, where the vertical grid lines are randomly perturbed, thereby de-centering the cut configuration in \Cref{fig: 2D_Laplace_ridge_corner_mesh}. Due to the complete lack of symmetry and structure, $z_1$ should neither coincide with a knot (for spline bases) nor an interpolation point (for Lagrange bases). Therefore, the expected scaling from \Cref{tab: scaling_jacobi} is $O(\eta^{-p})$ for spline bases and $O(\eta^{-(2p-1)})$ for Lagrange bases. \Cref{fig: 2D_Laplace_ridge_center_pert_cond_M_Jacobi} confirms those scalings. Despite a few outlier points, the trend is very clear. Hence, the scaling relations in \Cref{tab: scaling_jacobi} also hold for non-uniform grids.

\begin{figure}[H]
    \centering
    \includegraphics[width=0.5\textwidth]{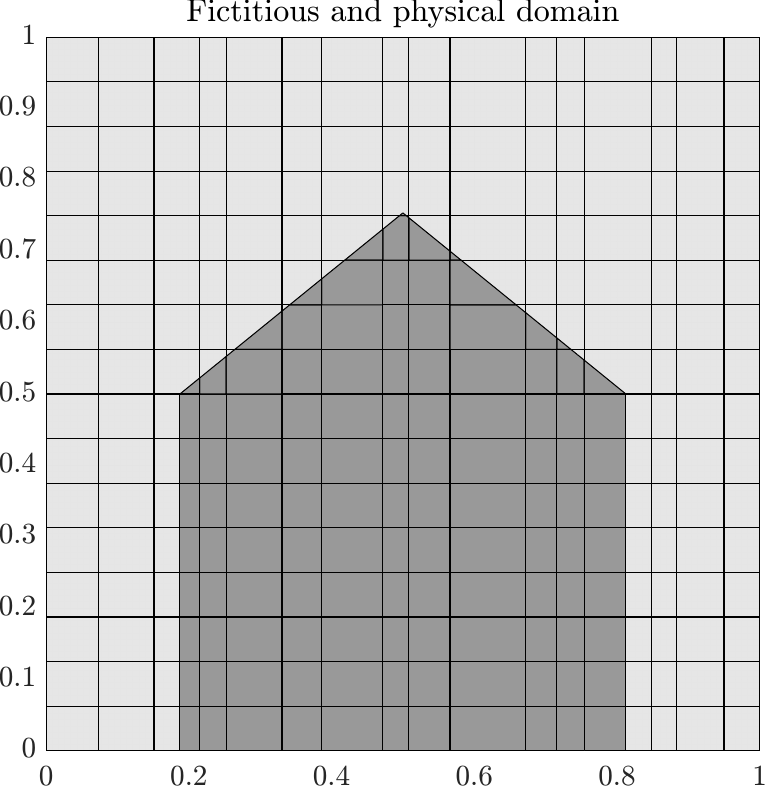}
    \caption{House-like structure with perturbed grid lines (\Cref{ex: perturbed_grid})}
    \label{fig: 2D_Laplace_ridge_corner_pert_mesh}
\end{figure}

\begin{figure}[H]
     \centering
     \begin{subfigure}[t]{0.31\textwidth}
    \centering
    \includegraphics[width=\textwidth]{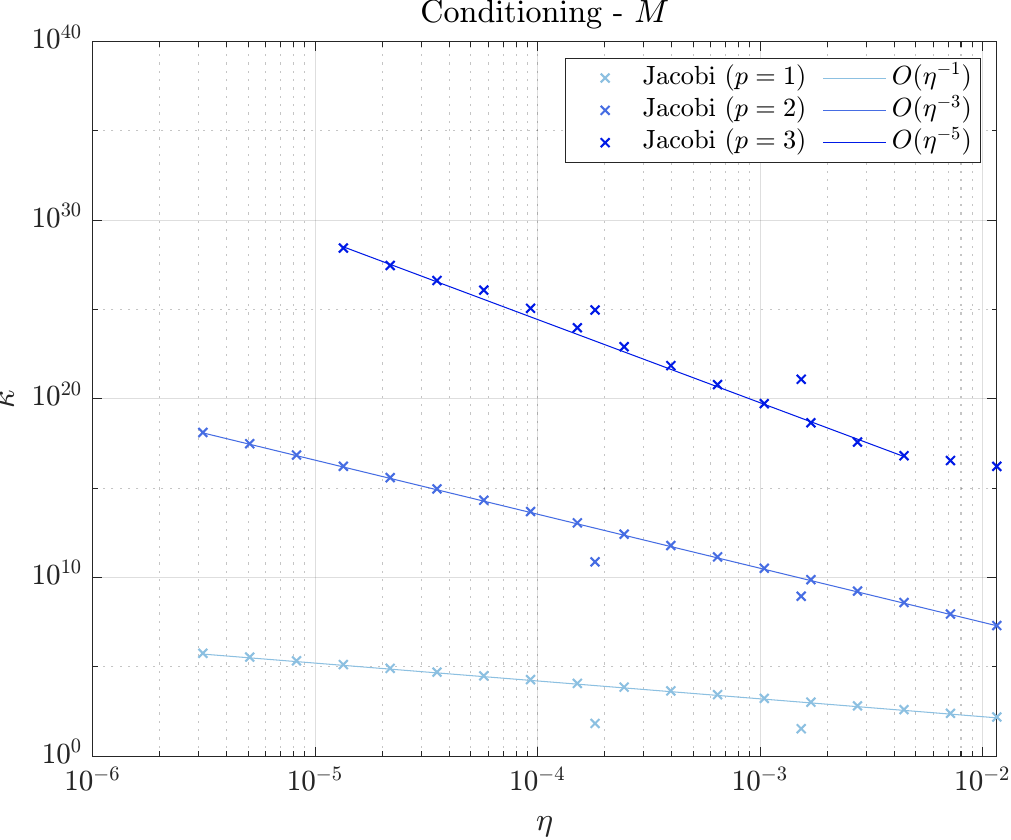}
    \caption{Lagrange basis}
    \label{fig: 2D_Laplace_ridge_center_pert_cond_M_Jacobi_Lagrange}
     \end{subfigure}
     \hfill
     \begin{subfigure}[t]{0.31\textwidth}
    \centering
    \includegraphics[width=\textwidth]{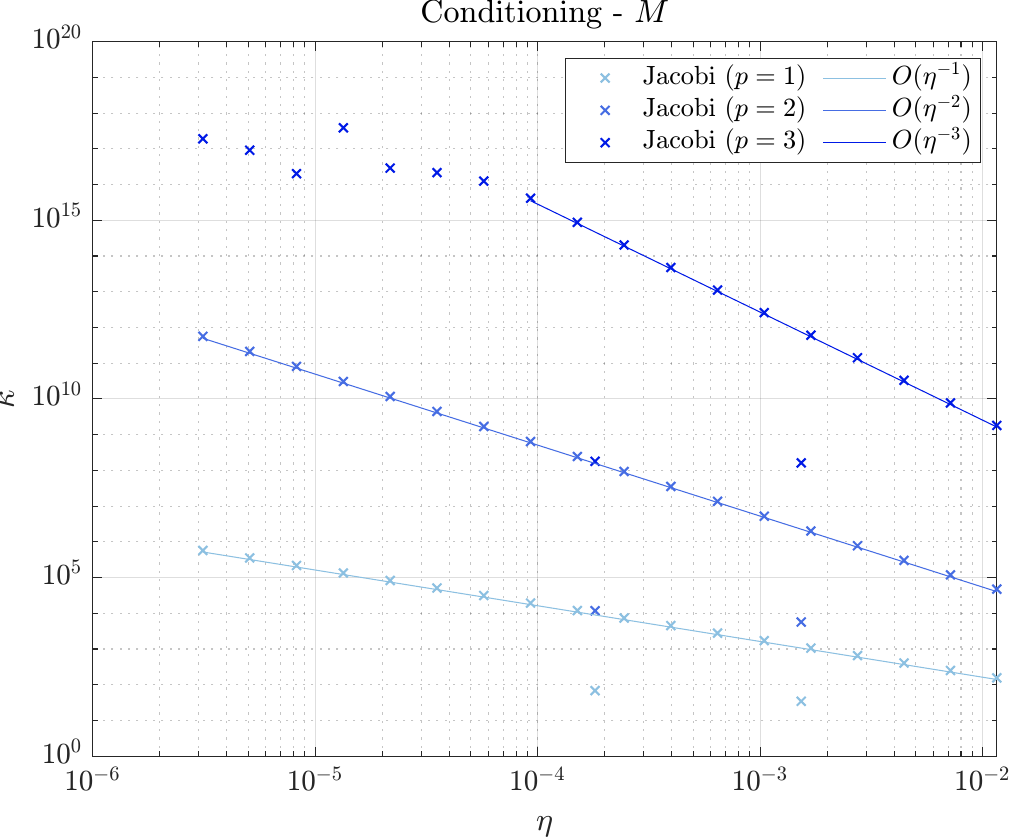}
    \caption{$C^0$ B-spline basis}
    \label{fig: 2D_Laplace_ridge_center_pert_cond_M_Jacobi_Bspline_C0}
     \end{subfigure}
     \hfill
    \begin{subfigure}[t]{0.31\textwidth}
    \centering
    \includegraphics[width=\textwidth]{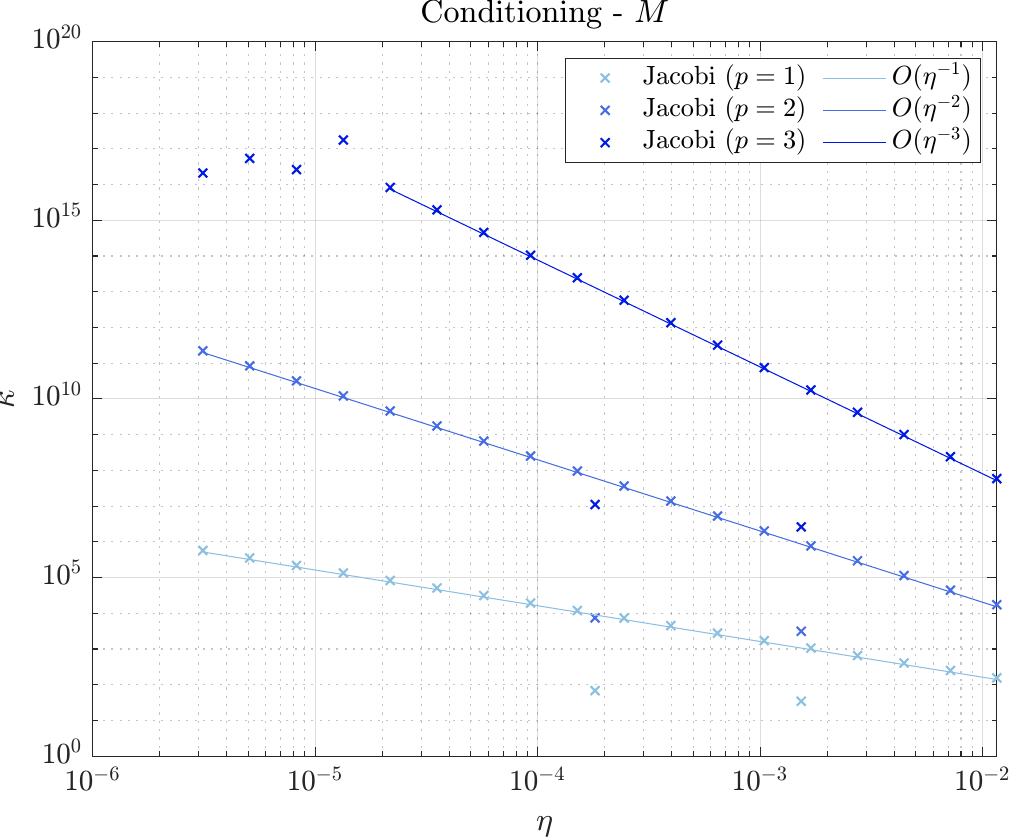}
    \caption{$C^{p-1}$ B-spline basis}
    \label{fig: 2D_Laplace_ridge_center_pert_cond_M_Jacobi_Bspline_Cp-1}
     \end{subfigure}
     \hfill
    \caption{Conditioning of the Jacobi preconditioned mass matrix for the house-like structure with perturbed grid lines (\Cref{ex: perturbed_grid})}
    \label{fig: 2D_Laplace_ridge_center_pert_cond_M_Jacobi}
\end{figure}
\end{example}

\section{Derivation of the deflated preconditioned conjugate gradient method}
\label{app: dpcg}
The Deflated Preconditioned Conjugate Gradient (DPCG) method directly follows from applying the standard CG method (recalled in \Cref{algo: cg}) to the deflated preconditioned system 
\begin{equation}
\label{eq: deflated_preconditioned_system_app}
    L^{-1}PAL^{-T}\hat{\bm{x}}=L^{-1}P\bm{b}
\end{equation}
and redefining variables to eliminate the Cholesky factors. This process is analogous to the derivation of the preconditioned conjugate gradient (PCG) method but is slightly more technical. Therefore, the transformations are detailed below for each step of the algorithm. For clarity, the quantities obtained from applying \Cref{algo: cg} to \eqref{eq: deflated_preconditioned_system_app} are identified with a hat.

Firstly, given an initial vector for iteratively solving $A\bm{x}=\bm{b}$, line 1 of \Cref{algo: cg} leads to
\begin{equation*}
    \hat{\bm{r}}_0=L^{-1}P(\bm{b}-A\bm{x}_0) = L^{-1}P\bm{r}_0 \quad \text{and} \quad \hat{\bm{p}}_0=\hat{\bm{r}}_0.
\end{equation*}
In order to derive a ``Cholesky-free'' version of deflated PCG, we introduce the auxiliary tilde-quantities
\begin{equation*}
    \tilde{\bm{r}}_j = L\hat{\bm{r}}_j \quad \text{and} \quad \tilde{\bm{p}}_j = L^{-T}\hat{\bm{p}}_j.
\end{equation*}
Similarly to PCG, we rewrite the algorithm in terms of the tilde-quantities:
\begin{itemize}[leftmargin=*, align=left, labelwidth=!]
    \item[Line 3:]
    \begin{equation*}
        \hat{\alpha}_j = (\hat{\bm{r}}_j,\hat{\bm{r}}_j)/(\hat{\bm{p}}_j,L^{-1}PAL^{-T}\hat{\bm{p}}_j) = (\tilde{\bm{r}}_j,H^{-1}\tilde{\bm{r}}_j)/(\tilde{\bm{p}}_j,PA\tilde{\bm{p}}_j).
    \end{equation*}
    \item[Line 4:] denoting $\tilde{\bm{x}}_j=L^{-T}\hat{\bm{x}}_j$ the conjugate gradient iterates of the deflated system $PA\bm{\tilde{x}}=P\bm{b}$, we obtain
    \begin{equation*}
    \hat{\bm{x}}_{j+1}=\hat{\bm{x}}_j+\hat{\alpha}_j\hat{\bm{p}}_j \iff \tilde{\bm{x}}_{j+1}=\tilde{\bm{x}}_j+\hat{\alpha}_j\tilde{\bm{p}}_j.
    \end{equation*}
    \item[Line 5:] 
    \begin{equation*}
        \hat{\bm{r}}_{j+1}=\hat{\bm{r}}_j-\hat{\alpha}_jL^{-1}PAL^{-T}\hat{\bm{p}}_j \iff \tilde{\bm{r}}_{j+1}=\tilde{\bm{r}}_j-\hat{\alpha}_jPA\tilde{\bm{p}}_j.
    \end{equation*}
    \item[Line 6:]
    \begin{equation*}
        \hat{\beta}_j=(\hat{\bm{r}}_{j+1},\hat{\bm{r}}_{j+1})/(\hat{\bm{r}}_j,\hat{\bm{r}}_j)=(\tilde{\bm{r}}_{j+1},H^{-1}\tilde{\bm{r}}_{j+1})/(\tilde{\bm{r}}_j,H^{-1}\tilde{\bm{r}}_j).
    \end{equation*}
    \item[Line 7:]
    \begin{equation*}
        \hat{\bm{p}}_{j+1}=\hat{\bm{r}}_{j+1}+\hat{\beta}_j \hat{\bm{p}}_j \iff  \tilde{\bm{p}}_{j+1}=H^{-1}\tilde{\bm{r}}_{j+1}+\hat{\beta}_j\tilde{\bm{p}}_{j}.
    \end{equation*}
\end{itemize}
The residuals $\tilde{\bm{r}}_j=P(\bm{b}-A\tilde{\bm{x}}_j)$ computed during the course of the algorithm are those of the deflated system $PA\bm{\tilde{x}}=P\bm{b}$. However, if we define the iterates of the original system as (see \eqref{eq: solution})
\begin{equation}
\label{eq: approximate_solution}
    \bm{x}_j = P^T \tilde{\bm{x}}_j + Z(Z^TAZ)^{-1}Z^T\bm{b},
\end{equation}
then, thanks to \Cref{lem: basic_properties}, \ref{prop: conjugacy}, the residuals for the original system are
\begin{align*}
    \bm{r}_j &= \bm{b}-A\bm{x}_j \\
    &=\bm{b}-AZ(Z^TAZ)^{-1}Z^T\bm{b}-AP^T\tilde{\bm{x}}_j \\
    &=P\bm{b}-PA\tilde{\bm{x}}_j \\
    &=P(\bm{b}-A\tilde{\bm{x}}_j)=\tilde{\bm{r}}_j.
\end{align*}
Therefore, the residuals computed in the algorithm are those of the original system. Finally, we obtain \Cref{algo: dpcg_app} by dropping the tilde over all quantities (see also \cite[Algorithm 6]{tang2008two}).

\begin{algorithm}[H]
\begin{algorithmic}[1]
\caption{Conjugate Gradient (CG) \cite[Algorithm 6.18]{saad2003iterative}}
\label{algo: cg}
\State Set $\bm{r}_0=\bm{b}-A\bm{x}_0$ and $\bm{p}_0=\bm{r}_0$.
\For{$j=0,1,\dots$ until convergence}
    \State $\alpha_j = (\bm{r}_j,\bm{r}_j)/(\bm{p}_j,A\bm{p}_j)$
    \State $\bm{x}_{j+1}=\bm{x}_j+\alpha_j\bm{p}_j$
    \State $\bm{r}_{j+1}=\bm{r}_j-\alpha_jA\bm{p}_j$
    \State $\beta_j=(\bm{r}_{j+1},\bm{r}_{j+1})/(\bm{r}_j,\bm{r}_j)$
    \State $\bm{p}_{j+1}=\bm{r}_{j+1}+\beta_j \bm{p}_j$
\EndFor
\end{algorithmic}
\end{algorithm}

\begin{algorithm}[H]
\begin{algorithmic}[1]
\caption{Deflated Preconditioned Conjugate Gradient (DPCG) \cite[Algorithm 1]{vermolen2004deflation}}
\label{algo: dpcg_app}
\State Set $\bm{r}_0=P(\bm{b}-A\bm{x}_0)$ and $\bm{p}_0=H^{-1}\bm{r}_0$.
\For{$j=0,1,\dots$ until convergence}
    \State $\alpha_j = (\bm{r}_j,H^{-1}\bm{r}_j)/(\bm{p}_j,PA\bm{p}_j)$
    \State $\bm{x}_{j+1}=\bm{x}_j+\alpha_j\bm{p}_j$
    \State $\bm{r}_{j+1}=\bm{r}_j-\alpha_jPA\bm{p}_j$
    \State $\beta_j=(\bm{r}_{j+1},H^{-1}\bm{r}_{j+1})/(\bm{r}_j,H^{-1}\bm{r}_j)$
    \State $\bm{p}_{j+1}=H^{-1}\bm{r}_{j+1}+\beta_j \bm{p}_j$
\EndFor
\State $\bm{x}_{j+1} = Z(Z^TAZ)^{-1}Z^T\bm{b} + P^T\bm{x}_{j+1}$
\end{algorithmic}
\end{algorithm}

In order to devise a reliable stopping criterion for \Cref{algo: dpcg_app}, we must understand which quantities control the relative error. We begin with a preliminary lemma for a positive semidefinite matrix.

\begin{lemma}
\label{lem: error_bound_spsd}
Let $\{\bm{x}_j\}$ be a sequence of approximate solutions to the consistent linear system $A\bm{x}=\bm{b}$, where $A \in \mathbb{R}^{n \times n}$ is symmetric positive semidefinite with nullity $r$ and $\bm{b} \neq \bm{0}$. Then,
\begin{equation*}
    \frac{\|\bm{x}-\bm{x}_j\|_A}{\|\bm{x}\|_A} \leq \sqrt{\keff(A)} \frac{\|\bm{r}_j\|_2}{\|\bm{b}\|_2}.
\end{equation*}
\end{lemma}
\begin{proof}
First note that solutions to $A\bm{x}=\bm{b}$ only differ by a vector $\bm{v} \in \ker(A)$. Therefore, $\|\bm{x}\|_A$ is invariant and does not identically vanish since $\bm{b} \neq \bm{0}$.
Then, from the equation
\begin{equation*}
    \bm{r}_j = \bm{b} - A\bm{x}_j = A(\bm{x}-\bm{x}_j),
\end{equation*}
we deduce that
\begin{equation}
    \|\bm{r}_j\|_2^2 = \|A(\bm{x}-\bm{x}_j)\|_2^2 =(\bm{x}-\bm{x}_j)^TA^{1/2}AA^{1/2}(\bm{x}-\bm{x}_j) \geq \lambda_{r+1}(A)\|\bm{x}-\bm{x}_j\|_A^2, \label{eq: numerator}
\end{equation}
where the last step follows from the fact that $A^{1/2}(\bm{x}-\bm{x}_j)$ does not have any components in the null space of $A$. Similarly,
\begin{equation}
\label{eq: denominator}
    \|\bm{b}\|_2^2 = \|A\bm{x}\|_2^2 = \bm{x}^TA^{1/2}AA^{1/2}\bm{x} \leq \lambda_n(A)\|\bm{x}\|_A^2.
\end{equation}
The result now follows from combining \eqref{eq: numerator} with \eqref{eq: denominator}.
\end{proof}
The next corollary proves the error bound \eqref{eq: dp_bound}.

\begin{corollary}
The approximate solutions \eqref{eq: approximate_solution} of the deflated preconditioned system satisfy
\begin{equation*}
    \frac{\|\bm{x}-\bm{x}_j\|_{PA}}{\|\bm{x}\|_{PA}} \leq \sqrt{\keff(H^{-1}PA)} \frac{\|\bm{r}_j\|_{H^{-1}}}{\|P\bm{b}\|_{H^{-1}}}.
\end{equation*}
\end{corollary}
\begin{proof}
The result follows from applying \Cref{lem: error_bound_spsd} to the deflated preconditioned system \eqref{eq: deflated_preconditioned_system_app}.
\end{proof}
\end{appendices}

\end{document}